\documentclass[hidelinks, 10pt]{article}

\usepackage[UKenglish]{babel}
\usepackage[utf8]{inputenc}
\usepackage{lmodern}
\usepackage{amsmath}
\usepackage{amsthm}
\usepackage{amsfonts}
\usepackage{amssymb}
\usepackage{mathtools}
\usepackage{dsfont} 
\usepackage[outline]{contour}
\usepackage{anyfontsize}

\usepackage[dvipsnames]{xcolor}
\usepackage{hyperref}
\usepackage{fullpage}
\usepackage{microtype}
\usepackage{newpxmath,newpxtext}
\usepackage{booktabs}
\usepackage{enumitem} 
\usepackage{authblk} 
\usepackage{comment}

\usepackage{tikz}
\usetikzlibrary{positioning}
\usetikzlibrary{arrows,backgrounds}
\usetikzlibrary{decorations.pathmorphing}
\usepackage{tikz-cd}
\usepackage{adjustbox}

\tikzset{node distance=2cm, auto}
\tikzcdset{row sep/normal=2.7em,column sep/normal=3.5em}

\allowdisplaybreaks

\setlist[description]{font=\normalfont\bfseries\:\!}

\renewcommand\labelenumi{(\alph{enumi})}
\renewcommand\theenumi\labelenumi

\hypersetup{
    colorlinks=true,
    linkcolor=gray,
    filecolor=gray,      
    urlcolor=gray,
    citecolor=gray
    }
\numberwithin{equation}{section}

\newtheorem{theorem}{Theorem}[section]
\newtheorem{proposition}[theorem]{Proposition}
\newtheorem{lemma}[theorem]{Lemma}
\newtheorem{corollary}[theorem]{Corollary}

\theoremstyle{definition}
\newtheorem{definition}[theorem]{Definition}
\newtheorem{example}[theorem]{Example}
\newtheorem{remark}[theorem]{Remark}

\newenvironment{acknowledgements}{\vskip .5cm\noindent\small\textbf{Acknowledgements}.\quad\noindent}{}

\newcommand{\CategoryFont}[1]{\mathsf{\uppercase{#1}}}
\newcommand{\FunctorFont}[1]{\mathsf{#1}}

\renewcommand{\bar}[1]{\mkern 1.5mu\overline{\mkern-1.5mu#1\mkern-1.5mu}\mkern 1.5mu}

\renewcommand{\,}{,\dots,}

\newcommand{\1}{\mathds{1}}

\renewcommand{\=}{\:\colon\!=}
\renewcommand{\epsilon}{\varepsilon}
\renewcommand{\o}{\circ}
\renewcommand{\b}{\bullet}
\newcommand{\<}{\langle}
\renewcommand{\>}{\rangle}
\newcommand{\op}{\mathsf{op}}
\newcommand{\co}{\mathsf{co}}
\renewcommand{\^}[1]{^{(#1)}}
\renewcommand{\*}{\ast}
\newcommand{\nto}{\nrightarrow}

\newcommand{\Comp}{\mathsf{Comp}}
\newcommand{\Unit}{\mathsf{Unit}}
\newcommand{\Inc}{\mathsf{Inc}}

\newcommand{\lax}{\mathsf{lax}}

\newcommand{\N}{\mathbb{N}}
\newcommand{\Z}{\mathbb{Z}}

\newcommand{\id}{\mathsf{id}}

\newcommand{\CC}{{\mathbf{K}}}

\newcommand{\X}{\mathbb{X}}
\newcommand{\Y}{\mathbb{Y}}
\newcommand{\End}{\CategoryFont{End}}

\newcommand{\Cat}{\CategoryFont{Cat}}
\newcommand{\Mnd}{\CategoryFont{Mnd}}

\newcommand{\Alg}{\CategoryFont{Alg}}

\newcommand{\Fib}{\CategoryFont{Fib}}
\newcommand{\DFib}{{\dagger}\!-\!\CategoryFont{Fib}}
\newcommand{\Indx}{\CategoryFont{Indx}}
\newcommand{\RestrCat}{\CategoryFont{Rstr}}
\newcommand{\sRestrCat}{\CategoryFont{sRstr}}
\newcommand{\TwoCat}{\CategoryFont{TwoCat}}
\newcommand{\DsplyCat}{\Dsply\text{-}\CategoryFont{Cat}}
\newcommand{\dDsplyCat}{\Dsply\text{-}\CategoryFont{Cat}^\*}

\newcommand{\MCat}{\M\text{-}\CategoryFont{Cat}}
\newcommand{\MTngCat}{\M\text{-}\CategoryFont{TngCat}}

\newcommand{\TngCat}{\CategoryFont{TngCat}}
\newcommand{\TTngCat}{\mathbb{TNG}\CategoryFont{Cat}}

\newcommand{\TngMnd}{\CategoryFont{TngMnd}}

\newcommand{\Weil}{\CategoryFont{Weil}_1}

\newcommand{\T}{\mathrm{T}}
\newcommand{\TT}{\mathbb{T}}
\newcommand{\Tng}{\CategoryFont{Tng}}
\newcommand{\TTng}{\mathbb{T}\CategoryFont{ng}}

\newcommand{\Cart}{\CategoryFont{Cart}}

\newcommand{\Leung}{\mathfrak{L}}
\newcommand{\q}{\mathbf{q}}
\newcommand{\p}{\mathbf{p}}
\newcommand{\TngFib}{\CategoryFont{TngFib}}
\newcommand{\TngIndx}{\CategoryFont{TngIndx}}

\newcommand{\TngsRestrCat}{\CategoryFont{TngsRstr}}

\newcommand{\U}{\CategoryFont{U}}
\newcommand{\M}{\mathscr{M}}

\newcommand{\Dsply}{\mathscr{D}}

\newcommand{\IND}{\mathscr{I}}

\newcommand{\f}{\flat}
\newcommand{\pp}{\mathrm{pp}}

\newcommand{\DSpan}{\CategoryFont{DSpan}}
\newcommand{\Exp}{\CategoryFont{Exp}}
\newcommand{\LMnd}{\CategoryFont{LMnd}}
\newcommand{\RMnd}{\CategoryFont{RMnd}}
\newcommand{\El}{\mathrm{El}}
\newcommand{\AdjTng}{\CategoryFont{AdjTng}}

\newcommand{\FF}{\mathscr{F}}
\newcommand{\Par}{\FunctorFont{Par}}

\newcommand{\DB}{\CategoryFont{DB}}

\title{Tangentads: a formal approach to tangent categories}
\author{\uppercase{Marcello Lanfranchi}}
\affil{\normalsize\textit{Macquarie University, School of Mathematical and Physical Sciences}}
\date{}

\begin{document}

\maketitle

\begin{abstract}\noindent
Tangent category theory is a well-established categorical context for differential geometry. In a previous paper, a formal approach was adopted to provide a genuine Grothendieck construction in the context of tangent categories by introducing tangentads. A tangentad is to a tangent category as a formal monad is to a monad of a category. In this paper, we discuss the formal notion of tangentads, construct a $2$-comonad structure on the $2$-functor of tangentads, and introduce Cartesian, adjunctable, and representable tangentads. We also reinterpret the subtangent structure with negatives of a tangent structure as a right Kan extension. Furthermore, we present numerous examples of tangentads, such as tangent (split) restriction categories, tangent fibrations, tangent monads, display tangent categories, and infinitesimal objects. Finally, we employ the formal approach to prove that every tangent monad admits the construction of algebras, provided the underlying monad does, and show that tangent split restriction categories are $2$-equivalent to tangent $\M$-categories.
\end{abstract}

\begin{acknowledgements}
The first notion of tangentads was introduced in~\cite{lanfranchi:grothendieck-tangent-cats} under the name of \textit{tangent objects}. We thank Ross Street, Alexander Campbell, Steve Lack, Jean-Simon Lemay, Robin Cockett, Dorette Pronk, Geoff Vooys, Rory Lucyshyn-Wright, and Geoffrey Cruttwell for suggesting alternative names and participating in the discussion on naming. We opted for \textit{tangentads}, in agreement with Street's suggestion of adopting the ending \textit{-ad}. We would also like to thank Rory Lucyshyn-Wright and Michael Ching for useful discussions, which helped clarify this paper. This material is based upon work supported by the AFOSR under award number FA9550-24-1-0008.
\end{acknowledgements}


\tableofcontents

\section{Introduction}
\label{section:introduction}
Fibred, indexed, internal, and enriched categories represent different flavours of category theory. The multitude of these flavours is reflected in the theory of tangent categories, introduced by Rosick\'y ``\textit{to axiomatize properties of the tangent functor on the category of smooth manifolds}''~\cite{rosicky:tangent-cats} and later generalized and extended by Cockett and Cruttwell in~\cite{cockett:tangent-cats}.

\par Cartesian tangent categories~\cite[Definition~2.8]{cockett:tangent-cats}, tangent restriction categories~\cite[Definition~6.14]{cockett:tangent-cats}, tangent fibrations~\cite[Definition~5.2]{cockett:differential-bundles}, and tangent monads~\cite[Definition~19]{cockett:tangent-monads} represent different evolutions of tangent category theory.

\par Formal category theory offers a common theoretical framework for different constructions of category theory, which captures the similarities of the various theories. A formal approach to tangent category theory will be beneficial in keeping track of the several directions in which the theory is developing.

\par In a recent paper~\cite{lanfranchi:grothendieck-tangent-cats}, a formal approach to tangent category theory was adopted to characterize tangent fibrations as fibrations equipped with a \textit{tangent structure}. This approach was employed to construct a Grothendieck construction for tangent fibrations (\cite[Theorem~5.5]{lanfranchi:grothendieck-tangent-cats}).

\par The goal of this paper is to take further steps into this intuition and explore a formal theory of tangentads, previously introduced with the name of tangent objects.\\
\textit{A tangentad is to a tangent category as a formal monad is to a monad of a category}.

\par Using a formal approach to axiomatize categorical constructions is not a new idea. In particular, our work is inspired by Street's formal theory of monads~\cite{street:formal-theory-monads} in which he introduced the notion of a monad in a $2$-category and formalized the construction of algebras.

\par For starters, we recall and revisit the definition of a tangentad in an arbitrary $2$-category; we study the global structures and properties of tangentads; we equip the $2$-functor $\Tng_\co$ with a $2$-comonad structure; we characterize Cartesian tangentads and compare with tangentads of Cartesian objects, and we introduce adjunctable and representable tangentads.
\par We also present numerous examples of tangentads, such as tangent fibrations, reverse tangent categories, tangent restriction categories, and display tangent categories. We employ these characterizations to prove that tangent monads admit the construction of algebras provided the underlying monads do. Finally, we adopt the formal approach to proving that tangent split restriction categories are $2$-equivalent to $\M$-tangent categories. In particular, this $2$-equivalence shows that the restriction domain of a morphism $f\colon M\to N$ in the context of a tangent restriction category is an open subobject of $M$ (Remark~\ref{remark:open-subobjects}).

\par The formal theory of tangentads proves to be useful in the following circumstances:
\begin{itemize}
\item When two $2$-categories are known to be $2$-equivalent and the tangentads of at least one of the two are known, one can lift the $2$-equivalence to the tangentads of the two $2$-categories. In~\cite{lanfranchi:grothendieck-tangent-cats}, the formal approach was used to construct a genuine Grothendieck equivalence between tangent fibrations and tangent indexed categories (Theorem~\ref{theorem:full-Grothendieck}). In this paper, this technique is employed to prove the $2$-equivalence between tangent split restriction categories and $\M$-tangent categories (Theorem~\ref{theorem:split-restriction-cats-vs-M-cats});

\item To lift a construction of a given $2$-category to the corresponding tangentads. We adopt this technique to prove that tangent monads admit the construction of algebras (in the sense of Street~\cite{street:formal-theory-monads}) in the $2$-category of tangentads, provided the underlying monads do;

\item To simplify definitions which involve several structures. This, for example, applies to the notion of a tangent indexed category~\ref{definition:tangent-indexed-category}, in which several structures and axioms are required to define this concept.
\end{itemize}
In upcoming work, we aim to investigate this formal approach further to address the following goals:
\begin{itemize}
\item Extending vector fields, differential objects, differential bundles, and connections, to new tangent-like notions. In particular, when the tangentads of a given $2$-category can be characterized as tangent categories with extra structure, one needs to ``carry around'' such a structure;

\item Formalizing some key constructions of tangent category theory, such as the Lie bracket of vector fields, the covariant derivative, Riemann and the torsion tensors of connections.
\end{itemize}


\section{A formal theory of tangentads}
\label{section:formal-theory}
The objects and morphisms of a tangent category can be regarded as locally linear geometric spaces and locally linear functions, respectively, where the notion of local linearity is not intrinsic but contextual and dictated by the tangent structure. Inspired by algebraic and differential geometry, one of the goals of tangent category theory is to categorify geometric constructions in the context of arbitrary tangent categories. As part of this research program, some categorical concepts, such as fibration and monad, have been translated into tangent-theoretic versions. One is left to wonder if the interpretation of local linear behaviour for these concepts still applies to these notions.
\par In this section, we first recall the notion of tangent category and associated morphisms and recall an equivalent presentation of tangent categories due to Leung~\cite{leung:weil-algebras} which comes in handy in our story. We employ such a construction to introduce tangentad, unwrap this definition, and introduce the related morphisms.

\subsection{A brief overview on tangent categories}
\label{subsection:tangent-categories}
Tangent categories were first introduced by Rosick\'y in~\cite{rosicky:tangent-cats}, and later revisited and generalized by Cockett and Cruttwell in~\cite{cockett:tangent-cats}.

\par For starters, we recall the definition of additive bundle and related morphisms.

\begin{definition}
\label{definition:additive-bundle}
In a category $\X$, an \textbf{additive bundle} consists of a pair of objects $B$ and $E$ together with three morphisms $q\colon E\to B$, $z_q\colon B\to E$, and $s_q\colon E_2\to E$, respectively called the \textbf{projection}, the \textbf{zero morphism}, and the \textbf{sum morphism}, where $E_n$ denotes the $n$-fold pullback of $q$ along itself. Moreover, $q$, $z_q$, and $s_q$ need to satisfy the following conditions:
\begin{enumerate}
\item For every positive integer $n$, the $n$-fold pullback
\begin{equation*}
\begin{tikzcd}
{E_n} & E \\
E & M
\arrow["{\pi_n}", from=1-1, to=1-2]
\arrow["{\pi_1}"', from=1-1, to=2-1]
\arrow["\lrcorner"{anchor=center, pos=0.125}, draw=none, from=1-1, to=2-2]
\arrow["\dots"{marking, allow upside down}, shift right=3, draw=none, from=1-2, to=2-1]
\arrow["q", from=1-2, to=2-2]
\arrow["q"', from=2-1, to=2-2]
\end{tikzcd}
\end{equation*}
of the projection $q$ along itself exists;

\item The zero morphism is a section of the projection:
\begin{equation*}
\begin{tikzcd}
M & E \\
& M
\arrow["{z_q}", from=1-1, to=1-2]
\arrow[equals, from=1-1, to=2-2]
\arrow["q", from=1-2, to=2-2]
\end{tikzcd}
\end{equation*}

\item The sum morphism is compatible with the projection. Moreover, $s_q$ is associative, unital, and commutative
\begin{equation*}
\begin{tikzcd}
{E_2} & E \\
E & M
\arrow["{s_q}", from=1-1, to=1-2]
\arrow["{\pi_1}"', from=1-1, to=2-1]
\arrow["q", from=1-2, to=2-2]
\arrow["q"', from=2-1, to=2-2]
\end{tikzcd}\hfill\quad
\begin{tikzcd}
{E_3} & {E_2} \\
{E_2} & E
\arrow["{\id_E\times_Ms_q}", from=1-1, to=1-2]
\arrow["{s_q\times_M\id_E}"', from=1-1, to=2-1]
\arrow["{s_q}", from=1-2, to=2-2]
\arrow["{s_q}"', from=2-1, to=2-2]
\end{tikzcd}\hfill\quad
\begin{tikzcd}
E & {E_2} \\
& E
\arrow["{\<z_q\o q,\id_E\>}", from=1-1, to=1-2]
\arrow[equals, from=1-1, to=2-2]
\arrow["{s_q}", from=1-2, to=2-2]
\end{tikzcd}\hfill\quad
\begin{tikzcd}
{E_2} & {E_2} \\
& E
\arrow["\tau", from=1-1, to=1-2]
\arrow["{s_q}"', from=1-1, to=2-2]
\arrow["{s_q}", from=1-2, to=2-2]
\end{tikzcd}
\end{equation*}
where $\tau\colon E_2\to E_2$ denotes the canonical symmetry $\pi_2\times\pi_1$.
\end{enumerate}
In the following, an additive bundle $(q\colon E\to B,z_q,s_q)$ will be denoted by $\q\colon E\to B$ or by $\q^E_B$. When the symbol that denotes the additive bundle is decorated with a superscript or a subscript, the same decoration is applied to the projection, the zero morphism, and the sum morphism. For instance, the additive bundle $\q'\colon E'\to B'$ consists of the projection $q'$, the zero morphism $z_q'$, and the sum morphism $s_q'$.
\end{definition}

\begin{definition}
\label{definition:additive-bundle-morphism}
An \textbf{additive bundle morphism} from an additive bundle $\q\colon E\to B$ to an additive bundle $\q'\colon E'\to B'$ consists of a pair of morphisms $f\colon B\to B'$ and $g\colon E\to E'$ satisfying the following conditions:
\begin{equation*}
\begin{tikzcd}
E & {E'} \\
B & {B'}
\arrow["g", from=1-1, to=1-2]
\arrow["q"', from=1-1, to=2-1]
\arrow["{q'}", from=1-2, to=2-2]
\arrow["f"', from=2-1, to=2-2]
\end{tikzcd}\hfill\quad
\begin{tikzcd}
E & {E'} \\
B & {B'}
\arrow["g", from=1-1, to=1-2]
\arrow["{z_q}", from=2-1, to=1-1]
\arrow["f"', from=2-1, to=2-2]
\arrow["{z_q'}"', from=2-2, to=1-2]
\end{tikzcd}\hfill\quad
\begin{tikzcd}
{E_2} & {E'_2} \\
E & {E'}
\arrow["{g\times_fg}", from=1-1, to=1-2]
\arrow["{s_q}"', from=1-1, to=2-1]
\arrow["{s_q'}", from=1-2, to=2-2]
\arrow["g"', from=2-1, to=2-2]
\end{tikzcd}
\end{equation*}
An additive bundle morphism is denoted by $(f,g)\colon\q\to\q'$, where, in the first position $f$ indicates the base morphism and in the second one $g$ indicates the top morphism. 
\end{definition}

The ``correct'' notion of limits in a tangent category is the notion of $\T$-limits.

\begin{definition}
\label{definition:tangent-limit}
Let $\X$ be a category and $\T\colon\X\to\X$ be an endofunctor of $\X$. A \textbf{$\T$-limit diagram} consists of a limit diagram in $\X$ whose universal property is preserved by all iterates $\T^n$ of $\T$. In the following, we refer to pullbacks, equalizers, etc., whose underlying limit diagram is a $\T$-limit, as \textbf{$\T$-pullbacks}, \textbf{$\T$-equalizers}, etc. When the endofunctor $\T$ is the tangent bundle functor of a tangent structure, $\T$-limits are also called \textbf{tangent limits} and the convention extends to $\T$-pullbacks, $\T$-equalizers, etc.
\end{definition}

\begin{definition}
\label{definition:T-additive-bundle}
In a category $\X$ equipped with an endofunctor $\T$, a \textbf{$\T$-additive bundle} consists of an additive bundle $\q\colon E\to B$ such that the $n$-fold pullback $E_n$ of the projection $q$ along itself is an $n$-fold $\T$-pullback diagram.
\end{definition}

\begin{definition}
\label{definition:tangent-category}
A \textbf{tangent structure} on a category $\X$ consists of the following data:
\begin{description}
\item[Tangent bundle functor] An endofunctor $\T\colon\X\to\X$;

\item[Projection]  A natural transformation $p_M\colon\T M\to M$, natural in $M$, which admits all $n$-fold $\T$-pullbacks
\begin{equation*}
\begin{tikzcd}
{\T_nM} & {\T M} \\
{\T M} & M
\arrow["{\pi_n}", from=1-1, to=1-2]
\arrow["{\pi_1}"', from=1-1, to=2-1]
\arrow["\lrcorner"{anchor=center, pos=0.125}, draw=none, from=1-1, to=2-2]
\arrow["\dots"{marking, allow upside down}, shift right=3, draw=none, from=1-2, to=2-1]
\arrow["{p_M}", from=1-2, to=2-2]
\arrow["{p_M}"', from=2-1, to=2-2]
\end{tikzcd}
\end{equation*}
along itself;

\item[Zero morphism] A natural transformation $z_M\colon M\to\T M$, natural in $M$;

\item[Sum morphism] A natural transformation $s_M\colon\T_2M\to\T M$, natural in $M$;
\end{description}
such that, for each $M\in\X$, $\p_M\=(p_M,z_M,s_M)$ is a $\T$-additive bundle of $\X$;
\begin{description}
\item[Vertical lift] A natural transformation $l_M\colon\T M\to\T^2M$, natural in $M$, where $\T^2M\=\T\T M$, such that
\begin{align*}
&(z_M,l_M)\colon\p_M\to\T\p_M
\end{align*}
is an additive bundle morphism;

\item[Canonical flip] A natural transformation $c_M\colon\T^2M\to\T^2M$ such that
\begin{align*}
&(\id_M,c_M)\colon\T\p_M\to\p_{\T M}
\end{align*}
is an additive bundle morphism;
\end{description}
Moreover, the following coherence conditions must hold:
\begin{enumerate}
\item The vertical lift is coassocative, and compatible with the canonical flip:
\begin{equation*}
\begin{tikzcd}
{\T M} & {\T^2M} \\
{\T^2M} & {\T^3M}
\arrow["l_M", from=1-1, to=1-2]
\arrow["lM"', from=1-1, to=2-1]
\arrow["{\T l_M}", from=1-2, to=2-2]
\arrow["{l_{\T M}}"', from=2-1, to=2-2]
\end{tikzcd}\hfill\quad
\begin{tikzcd}
{\T M} & {\T^2M} \\
& {\T^2M}
\arrow["l_M", from=1-1, to=1-2]
\arrow["l_M"', from=1-1, to=2-2]
\arrow["c_M", from=1-2, to=2-2]
\end{tikzcd}\hfill\quad
\begin{tikzcd}
{\T^2M} & {\T^3M} & {\T^3M} \\
{\T^2M} && {\T^3M}
\arrow["{l_{\T M}}", from=1-1, to=1-2]
\arrow["c_M"', from=1-1, to=2-1]
\arrow["{\T c_M}", from=1-2, to=1-3]
\arrow["{c_{\T M}}", from=1-3, to=2-3]
\arrow["{\T l_M}"', from=2-1, to=2-3]
\end{tikzcd}
\end{equation*}

\item The canonical flip is a symmetric braiding:
\begin{equation*}
\begin{tikzcd}
{\T^2M} & {\T^2M} \\
& {\T^2M}
\arrow["c_M", from=1-1, to=1-2]
\arrow[equals, from=1-1, to=2-2]
\arrow["c_M", from=1-2, to=2-2]
\end{tikzcd}\hfill\quad
\begin{tikzcd}
{\T^3M} & {\T^3M} & {\T^3M} \\
{\T^3M} & {\T^3M} & {\T^3M}
\arrow["{\T c_M}", from=1-1, to=1-2]
\arrow["{c_{\T M}}"', from=1-1, to=2-1]
\arrow["{c_{\T M}}", from=1-2, to=1-3]
\arrow["{\T c_M}", from=1-3, to=2-3]
\arrow["{\T c_M}"', from=2-1, to=2-2]
\arrow["{c_{\T M}}"', from=2-2, to=2-3]
\end{tikzcd}
\end{equation*}

\item The tangent bundle is locally linear, namely, the following is a pullback diagram:
\begin{equation*}
\begin{tikzcd}
{\T_2M} & {\T\T_2M} & {\T^2M} \\
M && {\T M}
\arrow["{z_{\T M}\times l_M}", from=1-1, to=1-2]
\arrow["{\pi_1p_M}"', from=1-1, to=2-1]
\arrow["{\T s_M}", from=1-2, to=1-3]
\arrow["{\T p_M}", from=1-3, to=2-3]
\arrow["z_M"', from=2-1, to=2-3]
\end{tikzcd}
\end{equation*}
\end{enumerate}
A \textbf{tangent category} consists of a category $\X$ equipped with a tangent structure $\TT\=(\T,p,z,s,l,c)$. A tangent structure \textbf{has negatives} when it is equipped with:
\begin{description}
\item[Negation] A natural transformation $n_M\colon\T M\to\T M$, natural in $M$, satisfying the following condition:
\begin{equation*}
\begin{tikzcd}
\T M & {\T_2M} \\
{\id_M} & \T M
\arrow["{\<n_M,\id_{\T M}\>}", from=1-1, to=1-2]
\arrow["p_M"', from=1-1, to=2-1]
\arrow["s_M", from=1-2, to=2-2]
\arrow["z_M"', from=2-1, to=2-2]
\end{tikzcd}
\end{equation*}
\end{description}
\end{definition}

In the following, we denote the tangent bundle functor of a tangent structure with the same symbol adopted to denote a tangent structure, e.g., $\TT$, in the font $\T$. The projection, the zero morphism, the sum morphism, the vertical lift, and the canonical flip are respectively denoted by the letters $p$, $z$, $s$, $l$, and $c$, and when the tangent structure has negatives, the negation is denoted by $n$. When the symbol denoting the tangent structure is decorated with a subscript or with a superscript, the tangent bundle functor and each of the structural natural transformations are decorated in the same way, e.g., for a tangent structure $\TT'_\o$, the tangent bundle functor is denoted by $\T'_\o$, the projection by $p'_\o$ and so on.

\begin{definition}
\label{definition:fundamental-tangent-limits-preservation}
Given a tangent category $(\X,\TT)$, a functor $G\colon\X\to\X'$ \textbf{preserves the fundamental tangent limits} when it preserves the universality of the $n$-fold tangent pullback diagram of the projection along itself and the local linearity of the tangent bundle. Concretely, these two conditions amount to say that the following diagram:
\begin{equation*}
\begin{tikzcd}
{G\T_nM} & {G\T M} \\
{G\T M} & GM
\arrow["{G\pi_n}", from=1-1, to=1-2]
\arrow["{G\pi_1}"', from=1-1, to=2-1]
\arrow["p", from=1-2, to=2-2]
\arrow["{{\dots}}"{marking, allow upside down}, shift left=4, draw=none, from=2-1, to=1-2]
\arrow["p"', from=2-1, to=2-2]
\end{tikzcd}
\end{equation*}
is an $n$-fold tangent pullback diagram in $\X'$ and that:
\begin{equation*}
\begin{tikzcd}
{G\T_2M} & {G\T\T_2M} & {G\T^2M} \\
GM && {G\T M}
\arrow["{G(z_\T\times_zl)}", from=1-1, to=1-2]
\arrow["{G(\pi_1p)}"', from=1-1, to=2-1]
\arrow["{G\T s}", from=1-2, to=1-3]
\arrow["{G\T p}", from=1-3, to=2-3]
\arrow["Gz"', from=2-1, to=2-3]
\end{tikzcd}
\end{equation*}
is a pullback diagram in $\X'$.
\end{definition}

A tangent $1$-morphism from a tangent category $(\X,\TT)$ to a tangent category $(\X',\TT')$ consists of a functor $F\colon\X\to\X'$ together with a distributive lax $\alpha$ between the functor $F$ and the tangent bundle functors, which is compatible with the tangent structures. Tangent $1$-morphisms come in different flavours. In the next definition, we clarify this distinction.

\begin{definition}
\label{definition:tangent-morphisms-tangent-categories}
Let $(\X,\TT)$ and $(\X',\TT')$ be two tangent categories. A \textbf{lax tangent morphism} from $(\X,\TT)$ to $(\X',\TT')$ consists of a functor $F\colon\X\to\X'$ together with a natural transformation $\alpha_M\colon F\T M\to\T'FM$, natural in $M$, called a \textbf{lax distributive law} satisfying the following conditions:
\begin{enumerate}
\item Additivity:
\begin{align*}
&(\id_F,\alpha)\colon F\p\to\p'_F
\end{align*}
is an additive bundle morphism, where $F\p$ denotes $(Fp,Fz,Fs)$ and $\p'_F$ denotes $(p'_F,z'_F,s'_F)$;

\item Compatibility with the vertical lift:
\begin{equation*}
\begin{tikzcd}
{F\T^2M} & {\T' F\T M} & {\T'^2FM} \\
{F\T M} && {\T' FM}
\arrow["{\alpha\T}", from=1-1, to=1-2]
\arrow["{\T'\alpha}", from=1-2, to=1-3]
\arrow["Fl", from=2-1, to=1-1]
\arrow["\alpha"', from=2-1, to=2-3]
\arrow["{l'F}"', from=2-3, to=1-3]
\end{tikzcd}
\end{equation*}

\item Compatibility with the canonical flip:
\begin{equation*}
\begin{tikzcd}
{F\T^2M} & {\T' F\T M} & {\T'^2FM} \\
{F\T^2M} & {\T' F\T M} & {\T'^2FM}
\arrow["{\alpha\T}", from=1-1, to=1-2]
\arrow["{\T'\alpha}", from=1-2, to=1-3]
\arrow["Fc", from=2-1, to=1-1]
\arrow["{\alpha\T}"', from=2-1, to=2-2]
\arrow["{\T'\alpha}"', from=2-2, to=2-3]
\arrow["{c'F}"', from=2-3, to=1-3]
\end{tikzcd}
\end{equation*}
\end{enumerate}
A \textbf{colax tangent morphism} from $(\X,\TT)$ to $(\X',\TT')$ consists of a functor $G\colon\X\to\X'$ together with a natural transformation $\beta_M\colon\T'GM\to G\T M$, natural in $M$, called a \textbf{colax distributive law}, which satisfies the dual of the conditions of a lax distributive law, namely, additivity and the compatibility with the vertical lift and the canonical flip, in which the direction of the distributive law is now reversed.\newline
A \textbf{strong tangent morphism} is a lax tangent morphism whose distributive law is invertible. Finally, a \textbf{strict tangent morphism} is a strong tangent morphism whose distributive law is the identity.\newline
A tangent morphism \textbf{preserves the fundamental tangent limits} when its underlying functor does it.
\par A lax tangent morphism is denoted by $(F,\alpha)\colon(\X,\TT)\to(\X',\TT')$ while for colax tangent morphisms we adopt the notation $(F,\alpha)\colon(\X,\TT)\nto(\X',\TT')$. Finally, a strict tangent morphism is simply denoted by its underlying functor.
\end{definition}

\begin{remark}
\label{remark:conventions-colax-tangent-morphism}
Different conventions have been adopted in the literature around the definitions of colax and strong tangent morphisms. In particular, our convention aligns with~\cite{lanfranchi:differential-bundles-operadic-affine-schemes}, but differs from~\cite{cockett:tangent-cats}, in which strong tangent morphisms were required to preserve the fundamental tangent limits.
\end{remark}

\begin{remark}
\label{remark:strong-colax-tangent-morphisms}
As pointed out in~\cite[Section~4]{cockett:differential-bundles}, a strong tangent morphism is also a colax tangent morphism (called a \textit{comorphism} in that paper).
\end{remark}

Let us recall the definition of a tangent natural transformation between tangent morphisms.

\begin{definition}
\label{definition:tangent-natural-transformation}
Given two lax tangent morphisms $(F,\alpha),(G,\beta)\colon(\X,\TT)\to(\X',\TT')$, a \textbf{tangent natural transformation} from $(F,\alpha)$ to $(G,\beta)$ consists of a natural transformation $\varphi_M\colon FM\to GM$, natural in $M$, satisfying the following compatibility with the distributive laws:
\begin{equation*}
\begin{tikzcd}
{F\T M} & {\T'FM} \\
{G\T M} & {\T'GM}
\arrow["\alpha", from=1-1, to=1-2]
\arrow["{\varphi\T}"', from=1-1, to=2-1]
\arrow["{\T'\varphi}", from=1-2, to=2-2]
\arrow["\beta"', from=2-1, to=2-2]
\end{tikzcd}
\end{equation*}
A tangent natural transformation between two colax tangent morphisms $(F,\alpha),(G,\beta)\colon(\X,\TT)\nto(\X',\TT')$ consists of a natural transformation $\varphi_M\colon FM\to GM$, natural in $M$, which satisfies the same compatibilities as a tangent natural transformation between lax tangent morphisms, where the distributive laws are reversed.\newline
Finally, given two lax tangent morphisms $(F_\o,\alpha_\o)\colon(\X_\o,\TT_\o)\to(\X_\o',\TT_\o')$ and two colax tangent morphisms $(G,\beta)\colon(\X_\o,\TT_\o)\nto(\X_\b,\TT_\b)$ and $(G',\beta')\colon(\X_\o',\TT_\o')\nto(\X_\b',\TT_\b')$, a \textbf{double tangent morphism}
\begin{equation*}
\begin{tikzcd}
{(\X_\o,\TT_\o)} & {(\X_\b,\TT_\b)} \\
{(\X'_\o,\TT'_\o)} & {(\X_\b',\TT_\b')}
\arrow["{(G,\beta)}", "\shortmid"{marking}, from=1-1, to=1-2]
\arrow["{(F_\o,\alpha_\o)}"', from=1-1, to=2-1]
\arrow["\varphi"{description}, Rightarrow, from=1-2, to=2-1]
\arrow["{(F_\b,\alpha_\b)}", from=1-2, to=2-2]
\arrow["{(G,',\beta')}"', "\shortmid"{marking}, from=2-1, to=2-2]
\end{tikzcd}
\end{equation*}
consists of a natural transformation $\varphi_M\colon F_\b GM\to G'F_\o M$, natural in $M$, satisfying the following compatibility with the distributive laws:
\begin{equation*}
\begin{tikzcd}
{F_\b\T_\b GM} & {F_\b G\T_\o M} & {G'F_\o\T_\o M} \\
{\T'_\b F_\b GM} & {\T_\b'G'F_\o M} & {G'\T'_\o F_\o M}
\arrow["{F_\b\beta}", from=1-1, to=1-2]
\arrow["{\alpha_\b G}"', from=1-1, to=2-1]
\arrow["{\varphi\T}", from=1-2, to=1-3]
\arrow["{G'\alpha_\o}", from=1-3, to=2-3]
\arrow["{\T'_\b\varphi}"', from=2-1, to=2-2]
\arrow["{\beta'F_\o}"', from=2-2, to=2-3]
\end{tikzcd}
\end{equation*}
\end{definition}

Tangent categories with their tangent $1$ and $2$-morphisms can be organized into four distinct $2$-categories:
\begin{itemize}
\item Tangent categories, lax tangent morphisms, and corresponding tangent natural transformations form the $2$-category $\TngCat$;

\item Tangent categories, colax tangent morphisms, and corresponding tangent natural transformations form the $2$-category $\TngCat_\co$;

\item Tangent categories, strong tangent morphisms, and corresponding tangent natural transformations form the $2$-category $\TngCat_\cong$;

\item Tangent categories, strict tangent morphisms, and corresponding tangent natural transformations form the $2$-category $\TngCat_=$.
\end{itemize}
Furthermore,
\begin{itemize}
\item Tangent categories, lax tangent morphisms for vertical morphisms, colax tangent morphisms for horizontal morphisms, and double tangent morphisms for double morphisms form a double category denoted by $\TTngCat$.
\end{itemize}

In~\cite{leung:weil-algebras}, Leung presented a different approach to defining tangent categories, which comes in handy for our story. Let us briefly recall it here.

\par First, let $W_n$ denote the rig (a.k.a., semiring) freely generated by $n$ generators $\epsilon_1\,\epsilon_n$, subject to the following relations:
\begin{align*}
&\epsilon_i\epsilon_j=0
\end{align*}
for every $i,j=1\,n$. For $n=0$, $W_0$ denotes the rig $\N$ of natural numbers and for $n=1$, $W_1$ is simply denoted by $W$. $\Weil$ denotes the strict symmetric monoidal category freely generated by all $W_n$ for every positive integer $n$ and by the following morphisms:
\begin{align*}
&p\colon W\to\N         &&z\colon\N\to W         &&s\colon W_2\to W\\
&p(\epsilon)=0          &&z(1)=1                 &&s(\epsilon_1)=\epsilon=s(\epsilon_2)\\
&&&&&\\
&l\colon W\to W\otimes W=W^2         &&c\colon W^2\to W^2    &&\\
&l(\epsilon)=\epsilon'\epsilon       &&c(\epsilon)=\epsilon'\qquad c(\epsilon')=\epsilon
\end{align*}
where $W^2$ denotes the rig freely generated by two generators $\epsilon$ and $\epsilon'$ subject to the following relations
\begin{align*}
&\epsilon^2=0=\epsilon'^2
\end{align*}
and where the monoidal product coincides with the usual tensor product of rigs, namely, the coproduct in the category of rigs.
\par The monoidal category $\Weil$ comes equipped with a tangent structure whose tangent bundle functor is the functor $W\otimes(-)$ and whose structural natural transformations are induced by the structural morphisms listed above. Leung's classification theorem states that a tangent structure $\TT$ on a category $\X$ is equivalent to a strong monoidal functor
\begin{align*}
&\Leung[\TT]\colon\Weil\to\End(\X)
\end{align*}
from the monoidal category $\Weil$ to the monoidal category $\End(\X)$ of endofunctors of $\X$, which preserves the fundamental tangent limits. Moreover, these limits are preserved in a pointwise way, meaning that the limit operator commutes with the evaluation functor.

\par Furthermore, a tangent structure with negatives $\TT$ is equivalent to a strong monoidal functor
\begin{align*}
&\Leung^-[\TT]\colon\Weil^-\to\End(\X)
\end{align*}
preserving the fundamental pointwise tangent limits, where $\Weil^-$ is obtained by replacing the generators $W_n$ of $\Weil$ with the rings $W_n^-$ (not just rigs, so in particular, now $W_0^-$ denotes the ring $\Z$ of integers). In particular, $\Weil^-$ comes equipped with an extra structural morphism:
\begin{align*}
&n\colon W^-\to W^-\\
&n(\epsilon)=-\epsilon
\end{align*}
Let us introduce the following naming convention.

\begin{definition}
\label{definition:leung-functor}
A \textbf{Leung functor} consists of a strong monoidal functor
\begin{align*}
&\Leung[\TT]\colon\Weil\to\End(\X)
\end{align*}
which preserves the fundamental pointwise tangent limits, namely, which commutes with the evaluation functor. Moreover, a \textbf{Leung functor with negatives} is a strong monoidal functor
\begin{align*}
&\Leung^-[\TT]\colon\Weil^-\to\End(\X)
\end{align*}
which preserves the fundamental pointwise tangent limits.
\end{definition}

\subsection{Tangentads formalize tangent categories}
\label{subsection:tangent-objects}
In this section, we introduce the protagonist of this paper: the formal notion of a tangentad. The first definition of this concept appears in~\cite{lanfranchi:grothendieck-tangent-cats} in the context of tangent fibrations, with the name of \textit{tangent object}.

\par In this section, we start by recalling the definition of a tangentad, by extending Leung's classification of tangent categories, and by spelling out the details in a more concrete form. Secondly, Lax, colax, strong, and strict tangent $1$-morphisms are introduced, together with the corresponding notions of $2$ and double morphisms.

\par For this section, let $\CC$ denote a strict $2$-category. In order to define the notion of a tangentad, first we need to clarify the notion of \textit{pointwise limits} in a given $2$-category.

\begin{definition}
\label{definition:pointwise-limit}
Given a strict $2$-category $\CC$ and two objects $\X$ and $\Y$ of $\CC$, a limit in the category $\CC(\X,\Y)$ is \textbf{pointwise} when it is preserved by each functor $\CC(f,\Y)\colon\CC(\X,\Y)\to\CC(\X',\Y)$ induced by $1$-morphisms $f\colon\X'\to\X$ of $\CC$.
\end{definition}

\begin{remark}
\label{remark:pointwise-limit}
The author is grateful to Rory Lucyshyn-Wright for pointing out the importance of this assumption for tangentads during an informal discussion. This aspect was missing in the original definition provided by the author.
\end{remark}

When the $2$-category $\CC$ is the $2$-category $\Cat$ of categories, pointwise limits of $\CC(\X,\Y)$ are those limit diagrams in the category of functors of type $F\colon\X\to\Y$ that are preserved by the evaluation functor. Concretely, this means that, for an object $X$ of $\X$, and a diagram $D\colon I\to\CC(\X,\Y)$, the functor $\lim D\colon\X\to\Y$ evaluated at $X$ is isomorphic to the object $\lim D(X)$ of $\Y$, where $D(X)$ represents the diagram $I\to\Y$ obtained by evaluating each functor $D_A\colon\X\to\Y$, for each $A$ of $I$, at $X$.

\par When the target category $\Y$ has all finite limits, so does the category of functors $\CC(\X,\Y)$ and each limit is pointwise. However, when the target category is not known to be finitely complete, there is no guarantee for the limits of $\CC(\X,\Y)$ to be pointwise. A counterexample can be found in~\cite[Section~3.3]{kelly:enriched-cats}.

\par Unfortunately, in tangent category theory, the requirement of the existence of limits is a subtle matter since in differential geometry, not every pair of morphisms admits a pullback. In particular, a tangent category cannot be required to be finitely complete since this would rule out one of the main examples of a tangent category. Therefore, to make our definition of tangentads compatible with the usual notion of a tangent category, the limits involved in the definition of a tangentad must be pointwise.

\begin{definition}
\label{definition:tangent-object}
A \textbf{tangentad} in a $2$-category $\CC$ consists of an object $\X$ of $\CC$ together with a strong monoidal functor:
\begin{align*}
&\Leung[\TT]\colon\Weil\to\End_\CC(\X)
\end{align*}
which preserves the fundamental pointwise tangent limits. Moreover, a tangentad \textbf{has negatives} when the strong monoidal functor $\Leung[\TT]$ factors through the inclusion $\Weil\to\Weil^-$.
\end{definition}

\begin{remark}
\label{remark:poon-and-christine-approaches}
In~\cite[Definition~14.2]{leung:weil-algebras}, Leung already suggested a generalization of a tangent category, by introducing the notion of a \textit{tangent structure internal} to a monoidal category $\mathcal{E}$ as a strong monoidal functor $\Leung[\TT]\colon\Weil\to\mathcal{E}$ which preserves the fundamental pointwise tangent limits. In this sense, a tangentad consists of an object $\X$ in a $2$-category $\CC$ together with a tangent structure internal to the monoidal category $\End(\X)$.

\par Another approach to generalize tangent structures was employed in~\cite{bauer:infinity-tangent-cats} in which $\infty$-tangent category is defined as a strong monoidal functor from the $\infty$-category of Weil algebras $\Weil^\infty$ to the monoidal category $\End(\X)$ of endomorphisms of an $\infty$-category $\X$.

\par As pointed out by Michael Ching in an informal discussion with the author, $\infty$-tangent categories generalize to strong monoidal functors from $\Weil^\infty$ to the monoidal category $\End_\CC(\X)$, where $\CC$ is a $(\infty,2)$-category, which preserve the fundamental tangent limits. Such functors split along the homotopy category of $\Weil^\infty$, which coincides with $\Weil$. Thus, every such functor gives rise to a tangentad.
\end{remark}

Definition~\ref{definition:tangent-object} is rather abstract; to have a better grasp of this concept, let us unwrap it into a more concrete form. A tangentad of $\CC$ comprises:
\begin{description}
\item[Base object] An object $\X$;

\item[Tangent bundle $1$-morphism] A $1$-endomorphism $\T\colon\X\to\X$;

\item[Projection]  A $2$-morphism $p\colon\T\Rightarrow\id_\X$, which admits all pointwise $n$-fold pullbacks
\begin{equation*}
\begin{tikzcd}
	{\T_n} & \T \\
	\T & {\id_\X}
	\arrow["{\pi_n}", from=1-1, to=1-2]
	\arrow["{\pi_1}"', from=1-1, to=2-1]
	\arrow["\lrcorner"{anchor=center, pos=0.125}, draw=none, from=1-1, to=2-2]
	\arrow["\dots"{marking, allow upside down}, shift right=3, draw=none, from=1-2, to=2-1]
	\arrow["p", from=1-2, to=2-2]
	\arrow["p"', from=2-1, to=2-2]
\end{tikzcd}
\end{equation*}
which are preserved  by all iterates $\T^n$ of $\T$;

\item[Zero morphism] A $2$-morphism $z\colon\id_\X\Rightarrow\T$;

\item[Sum morphism] A $2$-morphism $s\colon\T_2\Rightarrow\T$;
\end{description}
such that, $\p\=(p,z,s)$ is a $\bar\T$-additive bundle of $\End(\X)$, where $\bar\T$ sends an endomorphism $F\colon\X\to\X$ to $\T\o F$;
\begin{description}
\item[Vertical lift] A $2$-morphism $l\colon\T\Rightarrow\T^2$ such that
\begin{align*}
&(z,l)\colon\p\to\T\p
\end{align*}
is an additive bundle morphism;

\item[Canonical flip] A $2$-morphism $c\colon\T^2\Rightarrow\T^2$ such that
\begin{align*}
&(\id_\X,c)\colon\T\p\to\p_\T
\end{align*}
is an additive bundle morphism;
\end{description}
Moreover, the following coherence conditions must hold:
\begin{enumerate}
\item The vertical lift is coassocative, and compatible with the canonical flip:
\begin{equation*}
\begin{tikzcd}
{\T} & {\T^2} \\
{\T^2} & {\T^3}
\arrow["l", from=1-1, to=1-2]
\arrow["l"', from=1-1, to=2-1]
\arrow["{\T l}", from=1-2, to=2-2]
\arrow["{l_\T}"', from=2-1, to=2-2]
\end{tikzcd}\hfill\quad
\begin{tikzcd}
{\T} & {\T^2} \\
& {\T^2}
\arrow["l", from=1-1, to=1-2]
\arrow["l"', from=1-1, to=2-2]
\arrow["c", from=1-2, to=2-2]
\end{tikzcd}\hfill\quad
\begin{tikzcd}
{\T^2} & {\T^3} & {\T^3} \\
{\T^2} && {\T^3}
\arrow["{l_\T}", from=1-1, to=1-2]
\arrow["c"', from=1-1, to=2-1]
\arrow["{\T c}", from=1-2, to=1-3]
\arrow["{c_\T}", from=1-3, to=2-3]
\arrow["{\T l}"', from=2-1, to=2-3]
\end{tikzcd}
\end{equation*}

\item The canonical flip is a symmetric braiding:
\begin{equation*}
\begin{tikzcd}
{\T^2} & {\T^2} \\
& {\T^2}
\arrow["c", from=1-1, to=1-2]
\arrow[equals, from=1-1, to=2-2]
\arrow["c", from=1-2, to=2-2]
\end{tikzcd}\hfill\quad
\begin{tikzcd}
{\T^3} & {\T^3} & {\T^3} \\
{\T^3} & {\T^3} & {\T^3}
\arrow["{\T c}", from=1-1, to=1-2]
\arrow["{c_\T}"', from=1-1, to=2-1]
\arrow["{c_\T}", from=1-2, to=1-3]
\arrow["{\T c}", from=1-3, to=2-3]
\arrow["{\T c}"', from=2-1, to=2-2]
\arrow["{c_\T}"', from=2-2, to=2-3]
\end{tikzcd}
\end{equation*}

\item The tangent bundle is locally linear, namely, the following is a pointwise pullback diagram:
\begin{equation*}
\begin{tikzcd}
{\T_2} & {\T\o\T_2} & {\T^2} \\
\id_\X && {\T}
\arrow["{z_\T\times l}", from=1-1, to=1-2]
\arrow["{\pi_1p}"', from=1-1, to=2-1]
\arrow["{\T s}", from=1-2, to=1-3]
\arrow["{\T p}", from=1-3, to=2-3]
\arrow["z"', from=2-1, to=2-3]
\end{tikzcd}
\end{equation*}
\end{enumerate}
Furthermore, a tangentad admits negatives when it is equipped with:
\begin{description}
\item[Negation] A $2$-morphism $n\colon\T\Rightarrow\T$ such that:
\begin{equation*}
\begin{tikzcd}
\T & {\T_2} \\
{\id_\X} & \T
\arrow["{\<n,\id_\T\>}", from=1-1, to=1-2]
\arrow["p"', from=1-1, to=2-1]
\arrow["s", from=1-2, to=2-2]
\arrow["z"', from=2-1, to=2-2]
\end{tikzcd}
\end{equation*}
\end{description}

We can now define morphisms of tangentads.

\begin{definition}
\label{definition:tangent-morphisms}
Let $(\X,\TT)$ and $(\X',\TT')$ be two tangentads of $\CC$. A \textbf{lax tangent $1$-morphism} from $(\X,\TT)$ to $(\X',\TT')$ consists of a $1$-morphism $F\colon\X\to\X'$ together with a $2$-morphism $\alpha\colon F\o\T\Rightarrow\T'\o F$, called a \textbf{lax distributive law} satisfying the following conditions:
\begin{enumerate}
\item Additivity:
\begin{align*}
&(\id_F,\alpha)\colon F\p\to\p'_F
\end{align*}
is an additive bundle morphism, where $F\p$ denotes $(Fp,Fz,Fs)$ and $\p'_F$ denotes $(p'_F,z'_F,s'_F)$;

\item Compatibility with the vertical lift:
\begin{equation*}
\begin{tikzcd}
{F\o\T^2} & {\T'\o F\o\T} & {\T'^2\o F} \\
{F\o\T} && {\T'\o F}
\arrow["{\alpha\T}", from=1-1, to=1-2]
\arrow["{\T'\alpha}", from=1-2, to=1-3]
\arrow["Fl", from=2-1, to=1-1]
\arrow["\alpha"', from=2-1, to=2-3]
\arrow["{l'F}"', from=2-3, to=1-3]
\end{tikzcd}
\end{equation*}

\item Compatibility with the canonical flip:
\begin{equation*}
\begin{tikzcd}
{F\o\T^2} & {\T'\o F\o\T} & {\T'^2\o F} \\
{F\o\T^2} & {\T'\o F\o\T} & {\T'^2\o F}
\arrow["{\alpha\T}", from=1-1, to=1-2]
\arrow["{\T'\alpha}", from=1-2, to=1-3]
\arrow["Fc", from=2-1, to=1-1]
\arrow["{\alpha\T}"', from=2-1, to=2-2]
\arrow["{\T'\alpha}"', from=2-2, to=2-3]
\arrow["{c'F}"', from=2-3, to=1-3]
\end{tikzcd}
\end{equation*}
\end{enumerate}
A \textbf{colax tangent $1$-morphism} from $(\X,\TT)$ to $(\X',\TT')$ consists of a $1$-morphism $G\colon\X\to\X'$ together with a $2$-morphism $\beta\colon\T'\o G\Rightarrow G\o\T$, called a \textbf{colax distributive law}, which satisfies the dual conditions of a lax distributive law, namely, additivity and the compatibility with the vertical lift and the canonical flip, in which the direction of the distributive law is now reversed.
\par A \textbf{strong tangent $1$-morphism} is a lax tangent $1$-morphism whose distributive law is invertible. A \textbf{strict tangent $1$-morphism} is a strong tangent $1$-morphism whose distributive law is the identity.\newline
Finally, a tangent morphism $(F,\alpha)\colon(\X,\TT)\to(\X',\TT')$ \textbf{preserves the fundamental tangent limits} when the functor $\CC(\X,F)\colon\End(\X)\to\CC(\X,\X')$ preserves the fundamental pointwise tangent limits of $(\X,\TT)$.
A lax tangent $1$-morphism is denoted by $(F,\alpha)\colon(\X,\TT)\to(\X',\TT')$ while for colax tangent $1$-morphisms we adopt the notation $(F,\alpha)\colon(\X,\TT)\nto(\X',\TT')$. Finally, a strict tangent $1$-morphism is denoted by its underlying $1$-morphism.
\end{definition}

\begin{definition}
\label{definition:tangent-2-morphisms}
Given two lax tangent morphisms $(F,\alpha),(G,\beta)\colon(\X,\TT)\to(\X',\TT')$ of tangentads in $\CC$, a \textbf{tangent $2$-morphism} from $(F,\alpha)$ to $(G,\beta)$ consists of a $2$-morphism $\varphi\colon F\Rightarrow G$ satisfying the following compatibility between the distributive laws:
\begin{equation*}
\begin{tikzcd}
{F\o\T} & {\T'\o F} \\
{G\o\T} & {\T'\o G}
\arrow["\alpha", from=1-1, to=1-2]
\arrow["{\varphi\T}"', from=1-1, to=2-1]
\arrow["{\T'\varphi}", from=1-2, to=2-2]
\arrow["\beta"', from=2-1, to=2-2]
\end{tikzcd}
\end{equation*}
A tangent $2$-morphism between two colax tangent morphisms $(F,\alpha),(G,\beta)\colon(\X,\TT)\nto(\X',\TT')$ consists of a $2$-morphism $\varphi\colon F\Rightarrow G$ which satisfies the same compatibility as a tangent $2$-morphism between lax tangent morphisms, where the distributive laws are reversed.\newline
Finally, given two lax tangent morphisms $(F_\o,\alpha_\o)\colon(\X_\o,\TT_\o)\to(\X_\o',\TT_\o')$ and two colax tangent morphisms $(G,\beta)\colon(\X_\o,\TT_\o)\nto(\X_\b,\TT_\b)$ and $(G',\beta')\colon(\X_\o',\TT_\o')\nto(\X_\b',\TT_\b')$, a \textbf{double tangent morphism}
\begin{equation*}
\begin{tikzcd}
{(\X_\o,\TT_\o)} & {(\X_\b,\TT_\b)} \\
{(\X'_\o,\TT'_\o)} & {(\X_\b',\TT_\b')}
\arrow["{(G,\beta)}", "\shortmid"{marking}, from=1-1, to=1-2]
\arrow["{(F_\o,\alpha_\o)}"', from=1-1, to=2-1]
\arrow["\varphi"{description}, Rightarrow, from=1-2, to=2-1]
\arrow["{(F_\b,\alpha_\b)}", from=1-2, to=2-2]
\arrow["{(G,',\beta')}"', "\shortmid"{marking}, from=2-1, to=2-2]
\end{tikzcd}
\end{equation*}
consists of a $2$-morphism $\varphi\colon F_\b\o G\Rightarrow G'\o F_\o$ satisfying the following compatibility with the distributive laws:
\begin{equation*}
\begin{tikzcd}
{F_\b\o\T_\b\o G} & {F_\b\o G\o\T_\o} & {G'\o F_\o\o\T_\o} \\
{\T'_\b\o F_\b\o G} & {\T_\b'\o G'\o F_\o} & {G'\o\T'_\o\o F_\o}
\arrow["{F_\b\beta}", from=1-1, to=1-2]
\arrow["{\alpha_\b G}"', from=1-1, to=2-1]
\arrow["{\varphi\T}", from=1-2, to=1-3]
\arrow["{G'\alpha_\o}", from=1-3, to=2-3]
\arrow["{\T'_\b\varphi}"', from=2-1, to=2-2]
\arrow["{\beta'F_\o}"', from=2-2, to=2-3]
\end{tikzcd}
\end{equation*}
\end{definition}

Tangentads with their tangent $1$ and $2$-morphisms can be organized into four distinct $2$-categories:
\begin{itemize}
\item Tangentads of $\CC$, lax tangent morphisms, and corresponding $2$-morphisms form the $2$-category $\Tng(\CC)$;

\item Tangentads of $\CC$, colax tangent morphisms, and corresponding $2$-morphisms form the $2$-category $\Tng_\co(\CC)$;

\item Tangentads of $\CC$, strong tangent morphisms, and corresponding $2$-morphisms form the $2$-category $\Tng_\cong(\CC)$;

\item Tangent categories, strict tangent morphisms, and corresponding $2$-morphisms form the $2$-category $\Tng_=(\CC)$.
\end{itemize}
Furthermore:
\begin{itemize}
\item Tangentads of $\CC$, lax tangent $1$-morphisms for vertical morphisms, colax tangent $1$-morphisms for horizontal morphisms, and double tangent morphisms for double morphisms form a double category denoted by $\TTng(\CC)$.
\end{itemize}

\begin{example}
\label{example:tangent-categories}
The archetypal example of a tangentad is a tangent category. Indeed, tangent categories are precisely tangentads in the $2$-category $\Cat$ of categories.
\end{example}

\begin{example}
\label{example:trivial-tangent-object}
In a $2$-category $\CC$, every object $\X$ is trivially a tangentad once equipped with the trivial tangent structure $\1$, whose tangent $1$-morphism is the identity over $\X$ and all the structural $2$-morphisms are the identity over $\id_\X$.
\end{example}

Example~\ref{example:trivial-tangent-object} induces a $2$-functor:
\begin{align*}
&\Inc\colon\CC\to\Tng(\CC)\\
&\Inc(\X)\=(\X,\1)\\
&\Inc(F\colon\X\to\X')\=F\colon(\X,\1)\to(\X',\1)\\
&\Inc(\varphi\colon F\Rightarrow G)\=\varphi\colon F\Rightarrow G
\end{align*}
which sends each object to the trivial tangentad.

\begin{proposition}
\label{proposition:Inc-U-adjunction}
The forgetful $2$-functor $\U\colon\Tng(\CC)\to\CC$, which forgets the tangent structure, is both a left and a right $2$-adjoint to the inclusion $2$-functor $\Inc\colon\CC\to\Tng(\CC)$.
\end{proposition}
\begin{proof}
First, let us prove that $\Inc\dashv\U$ forms an adjunction. The unit
\begin{align*}
&\eta\colon(\X,\TT)\to(\X,\1)=\Inc(\U(\X,\TT))
\end{align*}
is the lax tangent morphism $(\id_\X,p)\colon(\X,\TT)\to(\X,\1)$, induced by the projection $p$ of $(\X,\TT)$. While the counit
\begin{align*}
&\epsilon\colon\U(\Inc(\X))=\X\to\X
\end{align*}
is the identity. Let us now show that $\U\dashv\Inc$ forms a second adjunction. The unit
\begin{align*}
&\theta\colon\X\to\X=\U(\Inc(\X))
\end{align*}
is the identity, while the counit
\begin{align*}
&\tau\colon\Inc(\U(\X,\TT))=(\X,\1)\to(\X,\TT)
\end{align*}
is the lax tangent morphism $(\id_\X,z)\colon(\X,\1)\to(\X,\TT)$, induced by the zero morphism $z$ of $(\X,\TT)$.
\end{proof}

Dually, one can also observe that the $2$-functor $\U_\co\colon\Tng_\co(\CC)\to\CC$ which forgets the tangent structure also admits a left and a right adjoint $\Inc_\co\colon\CC\to\Tng_\co(\CC)$. Since $\U_\co$ plays an important role in the next section, let us state this in a proposition.

\begin{proposition}
\label{proposition:co-Inc-U-adjunction}
The forgetful $2$-functor $\U_\co\colon\Tng_\co(\CC)\to\CC$ which forgets the tangent structure is both a left and a right $2$-adjoint to the inclusion $2$-functor $\Inc_\co\colon\CC\to\Tng_\co(\CC)$.
\end{proposition}

\subsection{The monad structure of a tangentad}
\label{subsection:T-monad}
In~\cite{cockett:tangent-cats}, Cockett and Cruttwell showed that the tangent bundle functor $\T\colon\X\to\X$ of a tangent category $(\X,\TT)$ comes equipped with a monad structure for which $c\colon\T^2\Rightarrow\T^2$ defines a distributive law and such that $l\colon\T\Rightarrow\T^2$ becomes a monad morphism between $\T$ and the monad $\T^2$ induced by the distributive law $c$. In particular, these are the data of a tangent monad, as introduced in~\cite[Definition~19]{cockett:tangent-monads}, which is a monad internal to the $2$-category $\TngCat$.
\par This construction can be easily formalized within the context of tangentad theory. In this section, we briefly recall this construction.

\begin{lemma}
\label{lemma:T-monad}
The tangent bundle $1$-morphism $\T\colon\X\to\X$ of a tangentad $(\X,\TT)$ in a $2$-category $\CC$ comes equipped with a tangent monad structure on $(\X,\TT)$. Concretely, this structure comprises a monad structure on $\T$, whose unit is defined by the zero morphism $z\colon\id_\X\Rightarrow\T$ and whose multiplication is defined as follows:
\begin{align*}
&\mu\colon\T^2\xrightarrow{\<\T p,p_\T\>}\T_2\xrightarrow{s}\T
\end{align*}
Moreover, such a structure also includes a distributive law, represented by the canonical flip $c$, between the monad $(\T,z,\mu)$ and the tangent bundle $1$-morphism $\T$, which makes $\T$ into a strong tangent morphism $(\T,c)\colon(\X,\TT)\to(\X,\TT)$.
\end{lemma}
\begin{proof}
The proof is precisely the same as the one of~\cite[Propositions~3.4 and~3.5]{cockett:tangent-cats}.
\end{proof}

\begin{lemma}
\label{lemma:functoriality-T-monad}
The operation defined by Lemma~\ref{lemma:T-monad} which sends a tangentad $(\X,\TT)$ to a tangent monad $(\X,\TT;\T,z,\mu,c)$ extends to a $2$-functor:
\begin{align*}
&\mathrm{M}\colon\Tng_\co(\CC)\to\TngMnd(\CC)
\end{align*}
\end{lemma}
\begin{proof}
The compatibility of the colax distributive law of a lax tangent morphism $(F,\alpha)\colon(\X,\TT)\nto(\X',\TT')$ with the tangent structures makes $\alpha\colon\T'\o F\Rightarrow F\o\T$ into a $2$-morphism of tangentads:
\begin{equation*}
\begin{tikzcd}
{(\X,\TT)} & {(\X',\TT')} \\
{(\X,\TT)} & {(\X',\TT')}
\arrow["{(F,\alpha)}", from=1-1, to=1-2]
\arrow["{(\T,c)}"', from=1-1, to=2-1]
\arrow["\alpha"{description}, Rightarrow, from=1-2, to=2-1]
\arrow["{(\T',c')}", from=1-2, to=2-2]
\arrow["{(F,\alpha)}"', from=2-1, to=2-2]
\end{tikzcd}
\end{equation*}
Moreover, $\alpha$ is compatible with the multiplications and the units of the monads $(\T,\mu,z)$ and $(\T',\mu',z')$. Thus, $(F,\alpha)$ becomes a morphism of tangent monads. Finally, a $2$-morphism $\varphi\colon(F,\alpha)\Rightarrow(G,\beta)\colon(\X,\TT)\nto(\X',\TT')$ of tangentads defines a $2$-morphism of tangent monads.
\end{proof}


\section{Global properties and structures of tangentads}
\label{section:global-properties-structures}
In previous sections, we introduced tangentads and discussed some examples. In this section, we investigate some of the global structures and the global properties of the operation which sends a $2$-category to the $2$-category of tangentads. 

\par First, in Section~\ref{subsection:functoriality-TNG}, we show that this assignment is $2$-functorial. In particular, we establish that pointwise pullback-preserving $2$-functors $\Gamma$ are sent to $2$-functors $\Tng_\co(\Gamma)$.

\par In Section~\ref{subsection:comonad-TNG}, we prove that each $\Tng_\co(\Gamma)$ is also pointwise pullback-preserving and that $\Tng_\co$ becomes an endofunctor on the $2$-category $\TwoCat_\pp$. We use this argument to construct a $2$-comonad structure on $\Tng_\co$.

\par In Section~\ref{subsection:Cartesian-tangent-objects}, we introduce Cartesian tangentads and discuss the relationship between Cartesian objects in $\Tng(\CC)$ and tangentads in $\Cart(\CC)$, proving they are $2$-equivalent.

\par In Section~\ref{subsection:adjunctable-representable-tangent-objects}, we extend the notion of adjunctability and representability to the formal context of tangentads and in~\ref{subsection:negatives-Kan-extension}, we formalize Cockett and Cruttwell's construction of the subtangent category with negatives~\cite[Proposition~3.12]{cockett:tangent-cats} as a right Kan extension.

\subsection{The functoriality of tangentads and some technical lemmas}
\label{subsection:functoriality-TNG}
For starters, let u prove a useful lemma.

\begin{lemma}
\label{lemma:U-reflects-pointwise-limits}
The $2$-functor $\U_\co\colon\Tng_\co(\CC)\to\CC$ reflects pointwise limits. Concretely, given a diagram
\begin{align*}
D&\colon I\to\Tng_\co(\CC)((\X,\TT),(\X',\TT'))
\end{align*}
and a cone $\lambda^c\colon L\to D(c)$ of $D$, if the cone
\begin{align*}
&\U[\lambda^c]\colon\U[L]\to\U[D(c)]
\end{align*}
is a pointwise limit cone of the diagram
\begin{align*}
&\U[D]\colon I\xrightarrow{D}\Tng_\co(\CC)((\X,\TT),(\X',\TT'))\xrightarrow{\U}\CC(\X,\X')
\end{align*}
then $\lambda^c$ is a pointwise limit cone of $D$.
\end{lemma}
\begin{proof}
Consider a diagram $D\colon I\to\Tng_\co(\CC)((\X,\TT),(\X',\TT'))$ and let $\lambda^c\colon L\to D(c)$ be a cone of $D$. Since $D$ is a diagram in $\Tng_\co(\CC)((\X,\TT),(\X',\TT'))$ each $D(c)$ consists of a colax tangent morphism $(F^c,\alpha^c)\colon(\X,\TT)\to(\X',\TT')$, and each morphism $D(f\colon c\to d)$ consists of a $2$-morphism of tangentads $\varphi^f\colon(F^c,\alpha^c)\to(F^d,\alpha^d)$. Similarly, $L$ consists of a colax tangent morphism $(L,\kappa)\colon(\X,\TT)\to(\X',\TT')$ and $\lambda^c\colon(L,\kappa)\to(F^c,\alpha^c)$ is a $2$-morphism of tangentads.
\par Suppose that $\U[\lambda^c]\colon\U[L]\to\U[D(c)]$ is a pointwise limit cone of $\CC(\X,\X')$. Being pointwise means that for each morphism $G\colon\X''\to\X$, the cone
\begin{align*}
&\U[\lambda^c]_G\colon\U[L]\o G\to\U[D(c)]\o G
\end{align*}
is a limit cone of $\U[D]\o\CC(G,\X')$ in $\CC(\X'',\X')$.
\par Consider a colax tangent morphism $(G,\beta)\colon(\X'',\TT'')\to(\X,\TT)$. We want to show that
\begin{align*}
&\lambda^c_{(G,\beta)}\colon(L,\kappa)\o(G,\beta)\to(F^c,\alpha^c)\o(G,\beta)
\end{align*}
is a limit cone in $\Tng_\co(\CC)((\X'',\TT''),(\X',\TT'))$. Consider a cone $\gamma^c\colon(H,\theta)\to(F^c,\alpha^c)\o(G,\beta)$. Since $\U[\lambda^c]$ is pointwise, $\U[\lambda^c_{(G,\beta)}]=\U[\lambda^c]_G$ is a limit cone in $\CC(\X'',\X')$, thus, by the universal property of $\U[\lambda^c]_G$, we obtain a unique morphism of cones
\begin{align*}
\varphi&\colon\U[\gamma^c]\to\U[\lambda^c_{(G,\beta)}]
\end{align*}
which corresponds to a $2$-morphism $\varphi\colon H\Rightarrow L\o G$ of $\CC$, satisfying
\begin{equation*}
\begin{tikzcd}
{L\o G} & {F^c\o G} \\
H
\arrow["{\lambda^c}", from=1-1, to=1-2]
\arrow["\varphi", dashed, from=2-1, to=1-1]
\arrow["{\gamma^c}"', from=2-1, to=1-2]
\end{tikzcd}
\end{equation*}
for every $c\in I$. We need to show that $\varphi$ is a $2$-morphism of tangentads. Let us start by considering the following diagram:
\begin{equation}
\label{equation:diagram-reflection-pointwise-limits}
\begin{tikzcd}
{\T'\o H} &&&& {H\o\T''} \\
& {\T'\o L\o G} & {L\o\T\o G} & {L\o G\o\T''} \\
{\T'\o F^c\o G} && {F^c\o\T\o G} && {F^c\o G\o\T''}
\arrow["\theta", from=1-1, to=1-5]
\arrow["{\T'\varphi}"', from=1-1, to=2-2]
\arrow["{\T'\gamma^c}"', from=1-1, to=3-1]
\arrow["{\varphi\T''}", from=1-5, to=2-4]
\arrow["{\gamma^c\T''}", from=1-5, to=3-5]
\arrow["{\kappa G}", from=2-2, to=2-3]
\arrow["{\T'\lambda^c_{(G,\beta)}}"', from=2-2, to=3-1]
\arrow["{L\beta}", from=2-3, to=2-4]
\arrow["{\lambda^c_{\T G}}"{description}, from=2-3, to=3-3]
\arrow["{\lambda^c_{(G,\beta)}\T''}", from=2-4, to=3-5]
\arrow["{\alpha^cG}"', from=3-1, to=3-3]
\arrow["{F^c\beta}"', from=3-3, to=3-5]
\end{tikzcd}
\end{equation}
The two bottom trapeze diagrams commute since $\lambda^c$ is a $2$-morphism of tangentads and so does the outer square since $\gamma^c$ is a $2$-morphism of tangentads. The left and right triangles commute by definition of $\varphi$. Using the universal property of $\U[\lambda^c_{(G,\beta)}]_{\T''}$ we prove that the top trapeze diagram commutes as well, implying that $\varphi$ is a $2$-morphism of tangentads. Finally, the uniqueness of $\varphi$ is implied by the universal property of $\U[\lambda^c_{(G,\beta)}]_{\T''}$.
\end{proof}

\begin{remark}
\label{remark:reflection-pointwise-limits}
One could wonder why, in Lemma~\ref{lemma:U-reflects-pointwise-limits}, we consider colax tangent morphisms instead of lax ones. The reason lies in Diagram~\eqref{equation:diagram-reflection-pointwise-limits}. Indeed, to prove the commutativity of the upper trapeze diagram, we invoked the universal property of the limit cone $\lambda^c_{(G,\beta)}\T''$. When one considers lax tangent morphisms, the distributive laws $\theta$, $\kappa$, $\alpha^c$, and $\beta$ go in the opposite direction, so to prove the commutativity of the upper trapeze, one needs the cone $\T'\lambda^c_{(G,\beta)}$ to be a limit cone. This is true for tangent limits, but it would not be true in general.
\end{remark}

Remark~\ref{remark:reflection-pointwise-limits} suggests a weaker version of Lemma~\ref{lemma:U-reflects-pointwise-limits} in the context of lax tangent morphisms.

\begin{proposition}
\label{proposition:U-reflects-pointwise-tangent-limits}
The $2$-functor $\U\colon\Tng(\CC)\to\CC$ reflects pointwise tangent limits. Concretely, given a diagram $D\colon I\to\Tng(\CC)((\X,\TT),(\X',\TT'))$ and a cone $\lambda^c\colon L\to D(c)$ of $D$, if the cone
\begin{align*}
&\U[\lambda^c]\colon\U[L]\to\U[D(c)]
\end{align*}
is a pointwise limit cone of the diagram
\begin{align*}
&\U[D]\colon I\xrightarrow{D}\Tng(\CC)((\X,\TT),(\X',\TT'))\xrightarrow{\U}\CC(\X,\X')
\end{align*}
and its universal property is preserved by all functors $\CC(\X,{\T'}^n)$, where $\T'$ is the tangent bundle $1$-morphism of $(\X',\TT')$, then $\lambda^c$ is a pointwise limit cone of $D$ and it is preserved by all functors $\Tng(\CC)((\X,\TT),(\T',c'))$.
\end{proposition}
\begin{proof}
Consider a diagram $D\colon I\to\Tng(\CC)((\X,\TT),(\X',\TT'))$ and let $\lambda^c\colon L\to D(c)$ be a cone of $D$. Since $D$ is a diagram in $\Tng(\CC)((\X,\TT),(\X',\TT'))$ each $D(c)$ consists of a lax tangent morphism $(F^c,\alpha^c)\colon(\X,\TT)\to(\X',\TT')$, and each morphism $D(f\colon c\to d)$ consists of a $2$-morphism of tangentads $\varphi^f\colon(F^c,\alpha^c)\to(F^d,\alpha^d)$. Similarly, $L$ consists of a lax tangent morphism $(L,\kappa)\colon(\X,\TT)\to(\X',\TT')$ and $\lambda^c\colon(L,\kappa)\to(F^c,\alpha^c)$ a $2$-morphism of tangentads.
\par Suppose that $\U[\lambda^c]\colon\U[L]\to\U[D(c)]$ is a pointwise tangent limit cone of $\CC(\X,\X')$. Being pointwise and tangent means that for each morphism $G\colon\X''\to\X$ and each positive integer $n$, the cone
\begin{align*}
&{\T'}^n\U[\lambda^c]_G\colon{\T'}^n\o\U[L]\o G\to{\T'}^n\o\U[D(c)]\o G
\end{align*}
is a limit cone of ${\T'}^n\o\U[D]\o\CC(G,\X')$ in $\CC(\X'',\X')$.
\par Consider a lax tangent morphism $(G,\beta)\colon(\X'',\TT'')\to(\X,\TT)$. We want to prove that
\begin{align*}
&(\T',c')^n\lambda^c_{(G,\beta)}\colon(\T',c')^n\o(L,\kappa)\o(G,\beta)\to(\T',c')^n\o(F^c,\alpha^c)\o(G,\beta)
\end{align*}
is a limit cone in $\Tng(\CC)((\X'',\TT''),(\X',\TT'))$. Consider a cone $\gamma^c\colon(H,\theta)\to(\T',c')^n\o(F^c,\alpha^c)\o(G,\beta)$. Since $\U[\lambda^c]$ is pointwise, ${\T'}^n\U[\lambda^c_{(G,\beta)}]={\T'}^n\U[\lambda^c]_G$ is a limit cone in $\CC(\X'',\X')$, thus, by its universal property we obtain a unique morphism of cones
\begin{align*}
\varphi&\colon\U[\gamma^c]\to{\T'}^n\U[\lambda^c_{(G,\beta)}]
\end{align*}
which corresponds to a $2$-morphism $\varphi\colon H\Rightarrow{\T'}^n\o L\o G$ of $\CC$, satisfying
\begin{equation*}
\begin{tikzcd}
{{\T'}^n\o L\o G} & {{\T'}^n\o F^c\o G} \\
H
\arrow["{{\T'}^n\lambda^c}", from=1-1, to=1-2]
\arrow["\varphi", dashed, from=2-1, to=1-1]
\arrow["{\gamma^c}"', from=2-1, to=1-2]
\end{tikzcd}
\end{equation*}
for every $c\in I$. We need to show that $\varphi$ is a $2$-morphism of tangentads. Let us start by considering the following diagram:
\begin{equation*}
\adjustbox{width=\linewidth,center}{
\begin{tikzcd}
{H\o\T''} &&&& {\T'\o H} \\
& {{\T'}^n\o L\o G\o\T''} & {{\T'}^n\o L\o\T\o G} & {{\T'}^{n+1}\o L\o G} \\
{{\T'}^n\o F^c\o G\o\T''} && {{\T'}^n\o F^c\o\T\o G} && {{\T'}^{n+1}\o F^c\o G}
\arrow["\theta", from=1-1, to=1-5]
\arrow["{\varphi\T''}"', from=1-1, to=2-2]
\arrow["{\gamma^c\T''}"', from=1-1, to=3-1]
\arrow["{\T'\varphi}", from=1-5, to=2-4]
\arrow["{\T'\gamma^c}", from=1-5, to=3-5]
\arrow["{{\T'}^nL\beta}", from=2-2, to=2-3]
\arrow["{{\T'}^n\lambda^c_{(G,\beta)}\T''}"', from=2-2, to=3-1]
\arrow["{{\T'}^n\kappa G}", from=2-3, to=2-4]
\arrow["{{\T'}^n\lambda^c_{\T G}}"{description}, from=2-3, to=3-3]
\arrow["{{\T'}^{n+1}\lambda^c_{(G,\beta)}}", from=2-4, to=3-5]
\arrow["{{\T'}^nF^c\beta}"', from=3-1, to=3-3]
\arrow["{{\T'}^n\alpha^cG}"', from=3-3, to=3-5]
\end{tikzcd}
}
\end{equation*}
The two bottom trapeze diagrams commute since $\lambda^c$ is a $2$-morphism of tangentads, and so does the outer square since $\gamma^c$ is a $2$-morphism of tangentads. The left and right triangles commute by definition of $\varphi$. Using the universal property of ${\T'}^n\U[\lambda^c_{(G,\beta)}]$, we prove that the top trapeze diagram commutes as well, implying that $\varphi$ is a $2$-morphism of tangentads. Finally, the uniqueness of $\varphi$ is implied by the universal property of ${\T'}^n\U[\lambda^c_{(G,\beta)}]$.
\end{proof}

To extend the operation $\Tng$ which sends a $2$-category $\CC$ into the $2$-category $\Tng(\CC)$ of tangentads of $\CC$, first, we need to introduce the ``correct'' type of $2$-functor which behaves well with the operation $\Tng$.

\begin{definition}
\label{definition:pullback-preserving}
A $2$-functor $\Gamma\colon\CC\to\CC'$ is \textbf{pointwise pullback-preserving} if it preserves every pointwise pullback in each hom category $\End(\X)$.
\par In the following, we denote by $\TwoCat_\pp$ the $2$-category of $2$-categories, whose $1$-morphisms are pointwise pullback-preserving $2$-functors, and whose $2$-morphisms are $2$-natural transformations.
\end{definition}

In the following, we repeatedly make use of the following technical lemma.

\begin{lemma}
\label{lemma:right-adjoints-preserve-pointwise-limits}
If a $2$-functor $\Gamma\colon\CC'\to\CC$ admits a left $2$-adjoint $\Omega\colon\CC\to\CC'$, $\Gamma$ preserves pointwise limits. Concretely, this means that if a diagram $D\colon I\to\CC'(\X,\X')$ admits a pointwise limit, so does the diagram:
\begin{align*}
&\Gamma[D]\colon I\xrightarrow{D}\CC'(\X,\X')\xrightarrow{\Gamma}\CC(\Gamma[\X],\Gamma[\X'])
\end{align*}
In particular, $\Gamma$ is pointwise pullback-preserving.
\end{lemma}
\begin{proof}
Consider a diagram $D\colon I\to\CC'(\X,\Y)$ and let $\lambda^c\colon L\to D(c)$ be a pointwise limit cone of $D$. The limit cone being pointwise means that for each morphism $G\colon\X'\to\X$, the cone
\begin{align*}
&\lambda^c_G\colon L\o G\to D(c)\o G
\end{align*}
is a limit cone for $\CC'(G,\Y)\o D\colon I\to\CC'(\X',\Y)$. Showing that $\Gamma$ preserves pointwise limits amounts to prove that for each $G\colon\X'\to\Gamma[\X]$ in $\CC$, the cone
\begin{align*}
&\Gamma[\lambda^c]_G\colon\Gamma[L]\o G\to\Gamma[D(c)]\o G
\end{align*}
is a limit cone for:
\begin{align*}
&I\xrightarrow{\Gamma[D]}\CC(\Gamma[\X],\Gamma[\Y])\xrightarrow{\CC(G,\Y)}\CC(\X',\Gamma[\Y])
\end{align*}
Let $\eta\colon\id_\CC\to\Gamma\o\Omega$ and $\epsilon\colon\Omega\o\Gamma\to\id_{\CC'}$ denote the unit and the counit of the $2$-adjunction, respectively. We can start by invoking the $2$-adjunction and define $G^\f$ as follows:
\begin{align*}
&G^\f\colon\Omega[\X']\xrightarrow{\Omega[G]}\Omega\Gamma[\X]\xrightarrow{\epsilon}\X
\end{align*}
Since the limit cone $\lambda^c$ is pointwise
\begin{align*}
&\lambda^c_{G^\f}\colon L\o G^\f\to D(c)\o G^\f
\end{align*}
is a limit cone of:
\begin{align*}
&D\o G^\f\colon I\xrightarrow{D}\CC'(\X,\Y)\xrightarrow{\CC'(G^\f,\Y)}\CC'(\Omega[\X],\Y)
\end{align*}
Furthermore, since $\Gamma$ is a right $2$-adjoint it preserves limits, therefore
\begin{align*}
&\Gamma[\lambda^c_{G^\f}]\colon\Gamma[L\o G^\f]\to\Gamma[D(c)\o G^\f]
\end{align*}
is a limit cone of $\Gamma[D\o G^\f]$. By employing the $2$-functoriality of $\Gamma$ and the triangle identities, one can show that
\begin{align*}
&I\xrightarrow{\Gamma[D\o G^\f]}\CC(\Gamma\Omega[\X],\Gamma[\Y])\xrightarrow{\CC(\eta,\Gamma[\Y])}\CC(\X,\Gamma[\Y])
\end{align*}
is exactly $\CC(G,\Y)\o\Gamma[D]$. So, we need to show that
\begin{align*}
&\Gamma[\lambda^c_{G^\f}]_\eta\colon\Gamma[L\o G^\f]\o\eta\to\Gamma[D(c)\o G^\f]\o\eta
\end{align*}
is a limit cone for $\CC(G,\Y)\o\Gamma[D]$. Consider a cone $\gamma^c\colon H\to\Gamma[D(c)\o G^\f]\o\eta$ in $\CC(\X',\Gamma[\Y])$. Using triangle identities, one can see that:
\begin{align*}
&\Gamma[\epsilon]\o\Gamma\Omega\left[\Gamma[D(c)\o G^\f]\o\eta\right]=\Gamma[D(c)\o G^\f]
\end{align*}
Therefore, by the universal property of $\Gamma[\lambda^c_{G^\f}]$, we obtain a unique morphism of cones $\varphi$ satisfying:
\begin{equation*}
\begin{tikzcd}
{\Gamma[L\o G^\b]} & {\Gamma[D(c)\o G^\f]} \\
{\Gamma[\epsilon]\o\Gamma\Omega[H]}
\arrow["{\Gamma[\lambda^c_{G^\f}]}", from=1-1, to=1-2]
\arrow["\varphi", dashed, from=2-1, to=1-1]
\arrow["{\Gamma[\epsilon](\Gamma\Omega[\gamma^c])}"', from=2-1, to=1-2]
\end{tikzcd}
\end{equation*}
However, thanks to the triangle identities, the functor $\CC(\eta,\Gamma[\Y])$ sends this diagram to:
\begin{equation*}
\begin{tikzcd}
{\Gamma[L\o G^\f]\o\eta} & {\Gamma[D(c)\o G^\f]\o\eta} \\
H
\arrow["{\Gamma[\lambda^c_{G^\f}]}", from=1-1, to=1-2]
\arrow["{\varphi_\eta}", dashed, from=2-1, to=1-1]
\arrow["{\Gamma[\epsilon](\Gamma\Omega[\gamma^c])}"', from=2-1, to=1-2]
\end{tikzcd}
\end{equation*}
To prove uniqueness of $\varphi_\eta$, one can consider another morphism $\psi\colon H\to\Gamma[L\o G^\f]\o\eta$ of cones, use the counit to define a morphism $\Gamma[\epsilon](\Gamma\Omega[\psi])\colon\Gamma[\epsilon]\o\Gamma\Omega[H]\to\Gamma[L\o G^\f]$ of cones, which by the universality much coincide with $\varphi$. Using the triangle identites, one can show that $\psi=\Gamma[\epsilon](\Gamma\Omega[\psi])_\eta=\varphi_\eta$.
\end{proof}

\begin{lemma}
\label{lemma:functoriality-TNG}
A pointwise pullback-preserving $2$-functor $\Gamma\colon\CC\to\CC'$ sends every tangentad $(\X,\TT)$ of $\CC$ to a tangentad $(\Gamma[\X],\Gamma[\TT])$ of $\CC'$ whose tangent structure is obtained by applying $\Gamma$ to the tangent $1$-morphism $\T$ and to the structural $2$-morphisms of $\TT$.\newline
Furthermore, $\Gamma$ sends a (co)lax (strong, strict) tangent $1$-morphism $(F,\alpha)$ to a (co)lax (strong, strict) tangent $1$-morphism $(\Gamma[F],\Gamma[\alpha])$. Finally, a tangent $2$-morphism $\varphi\colon(F,\alpha)\Rightarrow(G,\beta)$ of (co)lax tangent $1$-morphisms defines a tangent $2$-morphism $\Gamma[\varphi]$.
\end{lemma}
\begin{proof}
Let us consider a tangentad $(\X,\TT)$ of $\CC$, which corresponds to an object $\X$ of $\CC$ together with a Leung functor $\Leung[\TT]\colon\Weil\to\End_\CC(\X)$. A $2$-functor $\Gamma\colon\CC\to\CC'$ induces a strong monoidal functor $\End_\CC(\X)\to\End_{\CC'}(\Gamma[\X])$ for every object $\X$ of $\CC$. Moreover, since $\Gamma$ is pointwise pullback-preserving, the composition
\begin{align*}
&\Weil\xrightarrow{\Leung[\TT]}\End_\CC(\X)\to\End_{\CC'}(\Gamma[\X])
\end{align*}
is a strong monoidal functor which preserves the fundamental tangent limits, namely, a Leung functor. By unwrapping this construction, one finds out that this tangentad is precisely $(\Gamma[\X],\Gamma[\TT])$.
\par It is also not hard to see that $\Tng[\Gamma]$ sends (co)lax (strong/strict) tangent morphisms to (co)lax (strong/strict) tangent morphisms, since the compatibilities between the distributive laws and the tangent structures are equational.
\par Now, consider a $2$-morphism $\varphi\colon\Gamma\Rightarrow\Gamma'$ and let $\Tng[\varphi]_{(\X,\TT)}$ to be just $\varphi_\X\colon(\Gamma[\X],\Gamma[\TT])\to(\Gamma'[\X],\Gamma'[\TT])$. Thanks to the $2$-naturality, $\varphi_\X$ is a strict tangent morphism between $(\Gamma[\X],\Gamma[\TT])$ and $(\Gamma'[\X],\Gamma'[\TT])$.
\end{proof}

Lemma~\ref{lemma:pullback-preserving} induces a $2$-functor
\begin{align*}
&\Tng\colon\TwoCat_\pp\to\TwoCat
\end{align*}
where $\TwoCat$ denotes the $2$-category of $2$-categories with all $2$-functors. Adopting a similar approach, one obtains three other $2$-functors
\begin{align*}
&\Tng_\co\colon\TwoCat_\pp\to\TwoCat\\
&\Tng_\cong\colon\TwoCat_\pp\to\TwoCat\\
&\Tng_=\colon\TwoCat_\pp\to\TwoCat
\end{align*}
which send a $2$-category $\CC$ to the $2$-categories $\Tng_\co(\CC)$, $\Tng_\cong(\CC)$, and $\Tng_=(\CC)$, respectively.

\subsection{The comonad structure of tangentads}
\label{subsection:comonad-TNG}
Lemma~\ref{lemma:functoriality-TNG} establishes that the operation which sends a $2$-category $\CC$ to the $2$-category $\Tng_\co(\CC)$ of tangentads extends to a functor. In this section, we show that not only this functor restricts to an endofunctor on the $2$-category $\TwoCat_\pp$ of $2$-categories and pointwise pullback-preserving $2$-functors, but it also comes equipped with a comonad structure.

\par For starters, let us prove a technical but useful lemma.

\begin{lemma}
\label{lemma:pullback-preserving}
A $2$-functor $\Gamma\colon\CC\to\Tng_\co(\CC')$ is pointwise pullback-preserving if and only if so is the composite:
\begin{align*}
&\CC\xrightarrow{\Gamma}\Tng_\co(\CC')\xrightarrow{\U_\co}\CC'
\end{align*}
\end{lemma}
\begin{proof}
First, let us suppose that $\Gamma$ is pointwise pullback-preserving. By Proposition~\ref{proposition:co-Inc-U-adjunction}, $\U_\co$ admits a left $2$-adjoint $\Inc_\co$. Thus, by Lemma~\ref{lemma:right-adjoints-preserve-pointwise-limits}, $\U_\co$ preserves pointwise limits, so in particular is pointwise pullback-preserving. Since pointwise pullback-preserving $2$-functors are closed under composition, $\U_\co\o\Gamma$ is pointwise pullback-preserving.
\par Conversely, suppose that the composition $\U_\co\o\Gamma$ is pointwise pullback-preserving. Consider a pullback diagram $D\colon I\to\End_\CC(\X)$ which admits a pointwise limit cone $\lambda^c\colon L\to D(c)$. Its image along $\U_\co\o\Gamma$, namely, the diagram
\begin{align*}
\U_\co\Gamma[D]&\colon I\to\End_{\CC'}(\U_\co\Gamma(\X))
\end{align*}
admits a pointwise limit cone
\begin{align*}
\U_\co\Gamma[\lambda^c]&\colon\U_\co\Gamma[L]\to\U_\co\Gamma[D(c)]
\end{align*}
since $\U_\co\o\Gamma$ preserves pointwise pullbacks. However, by Lemma~\ref{lemma:U-reflects-pointwise-limits}, $\U_\co$ reflects pointwise limits, thus the cone $\Gamma[\lambda^c]\colon\Gamma[L]\to\Gamma[D(c)]$ must be a pointwise limit cone in $\Tng_\co(\CC')$, proving that $\Gamma$ is pointwise pullback-preserving.
\end{proof}

Thanks to Lemma~\ref{lemma:pullback-preserving}, we can show that $\Tng_\co(\Gamma)$ is pointwise pullback-preserving provided that $\Gamma$ is.

\begin{lemma}
\label{lemma:Tng_co-endofunctor}
If $\Gamma\colon\CC\to\CC'$ is a pointwise pullback-preserving $2$-functor, the $2$-functor $\Tng_\co(\Gamma)\colon\Tng_\co(\CC)\to\Tng(\CC')$ is also pointwise pullback-preserving.
\end{lemma}
\begin{proof}
Let us consider a tangentad $(\X,\TT)$ of $\CC$, four lax tangent morphisms
\begin{align*}
(A,\alpha),(B,\beta),(C,\gamma),(D,\delta)\colon(\X,\TT)\to(\X,\TT)
\end{align*}
and four tangent $2$-morphisms
\begin{align*}
f&\colon(B,\beta)\to(A,\alpha)          &g\colon&(C,\gamma)\to(A,\alpha)\\
\pi_1&\colon(D,\delta)\to(B,\beta)      &\pi_2&\colon(D,\delta)\to(C,\gamma)
\end{align*}
making the following commutative diagram
\begin{equation*}
\begin{tikzcd}
{(D,\delta)} & {(C,\gamma)} \\
{(B,\beta)} & {(A,\alpha)}
\arrow["{\pi_2}", from=1-1, to=1-2]
\arrow["{\pi_1}"', from=1-1, to=2-1]
\arrow["\lrcorner"{anchor=center, pos=0.125}, draw=none, from=1-1, to=2-2]
\arrow["g", from=1-2, to=2-2]
\arrow["f"', from=2-1, to=2-2]
\end{tikzcd}
\end{equation*}
a pointwise pullback diagram of $\End_{\Tng(\CC)}(\X,\TT)$. We need to prove that the $2$-functor $\Tng(\Gamma)$ preserves such a pullback. In order to simplify notation, we decorate each term by $'$ to indicate the application of $\Gamma$ on it, e.g., $A'$ indicates $\Gamma(A)$. With this notation, we need to prove that the following diagram is a pointwise pullback diagram of $\End_{\Tng(\CC')}(\X',\TT')$, where $(\X',\TT')=\Tng(\Gamma)(\X,\TT)$:
\begin{equation*}
\begin{tikzcd}
{(D',\delta')} & {(C',\gamma')} \\
{(B',\beta')} & {(A',\alpha')}
\arrow["{\pi_2'}", from=1-1, to=1-2]
\arrow["{\pi_1'}"', from=1-1, to=2-1]
\arrow["\lrcorner"{anchor=center, pos=0.125}, draw=none, from=1-1, to=2-2]
\arrow["{g'}", from=1-2, to=2-2]
\arrow["{f'}"', from=2-1, to=2-2]
\end{tikzcd}
\end{equation*}
Consider two tangent $2$-morphisms $b\colon(K,\kappa)\to(B',\beta')$ and $c\colon(K,\kappa)\to(C',\gamma')$ for which $f'\o b=g'\o c$. Since $\Gamma$ preserves pointwise pullbacks, by forgetting the tangent structure we obtain a $2$-morphism $d\colon K\to D$ of $\CC$ which makes the following diagram commutes:
\begin{equation*}
\begin{tikzcd}
K \\
& {D'} & {C'} \\
& {B'} & {A'}
\arrow["d"{description}, dashed, from=1-1, to=2-2]
\arrow["c", bend left, from=1-1, to=2-3]
\arrow["b"', bend right, from=1-1, to=3-2]
\arrow["{\pi_2'}", from=2-2, to=2-3]
\arrow["{\pi_1'}"', from=2-2, to=3-2]
\arrow["\lrcorner"{anchor=center, pos=0.125}, draw=none, from=2-2, to=3-3]
\arrow["{g'}", from=2-3, to=3-3]
\arrow["{f'}"', from=3-2, to=3-3]
\end{tikzcd}
\end{equation*}
Therefore, the $2$-functor $\U_\co\o\Tng_\co(\Gamma)$ preserves pointwise pullbacks. By Lemma~\ref{lemma:pullback-preserving}, we conclude that also $\Tng_\co(\Gamma)$ must preserve pointwise pullbacks.
\end{proof}

\begin{proposition}
\label{proposition:functoriality-TNG}
The $2$-functor $\Tng_\co\colon\TwoCat_\pp\to\TwoCat$ splits along the inclusion
\begin{align*}
&\TwoCat_\pp\to\TwoCat
\end{align*}
by defining a $2$-endofunctor $\Tng_\co\colon\TwoCat_\pp\to\TwoCat_\pp$.
\end{proposition}

\begin{remark}
\label{remark:functoriality-TNG}
Since a strong tangent morphism is also a colax tangent morphism, the splitting of Proposition~\ref{proposition:functoriality-TNG} restricts to strong and strict tangent morphisms. In particular, this yields two $2$-endofunctors $\Tng_\cong$ and $\Tng_=$ of $\TwoCat_\pp$.
\end{remark}

In order to equip $\Tng_\co$ with a comonad structure, we need two ingredients: a counit $\epsilon\colon\Tng_\co(\CC)\to\CC$, and a comultiplication $\Delta\colon\Tng_\co(\CC)\to\Tng_\co(\Tng_\co(\CC))$. To define the counit, we can use the forgetful $2$-functor $\U_\co\colon\Tng_\co(\CC)\to\CC$ which forgets the tangent structure. To define comultiplication, we first need the following lemma.

\begin{lemma}
\label{lemma:comultiplication-TNG}
A tangentad $(\X,\TT)$ of $\CC$ is a tangentad in the $2$-category $\Tng_\co(\CC)$ of tangentads of $\CC$. In particular, the tangent bundle $1$-morphism of $(\X,\TT)$ in $\Tng_\co(\CC)$ is $(\T,c)\colon(\X,\TT)\to(\X,\TT)$, and the structural $2$-morphisms are the same as the ones of $\TT$.
\end{lemma}
\begin{proof}
The proof is a straightforward, yet tedious, computation that we leave to the reader to spell out.
\end{proof}

Lemma~\ref{lemma:comultiplication-TNG} induces a $2$-functor:
\begin{align*}
&\Delta_\co\colon\Tng_\co(\CC)\to\Tng_\co(\Tng_\co(\CC))
\end{align*}
To understand how $\Delta_\co$ is defined, first notice that an object of $\Tng_\co^2(\CC)$ can be equivalently described as an object $\X$ of $\CC$ together with two tangent structures $\TT$ and $\TT'$ on $\X$ and a distributive law $\gamma\colon\T\o\T'\Rightarrow\T'\o\T$, compatible with the two tangent structures. Moreover, a morphism of $\Tng_\co^2(\CC)$ from $(\X_\o,\TT_\o,\TT_\o',\gamma_\o)$ to $(\X_\b,\TT_\b,\TT_\b',\gamma_\b)$ consists of a colax tangent morphism $(F,\alpha)\colon(\X_\o,\TT_\o)\nto(\X_\b,\TT_\b)$, together with a $2$-morphism $\beta\colon\T_\b'\o F\Rightarrow F\o\T_\o'$, compatible with the tangent structures. Finally, a $2$-morphism of $\Tng_\co^2(\CC)$ from $(F,\alpha,\beta)$ to $(F',\alpha',\beta')$ consists of a $2$-morphism $\varphi\colon F\Rightarrow F'$ compatible with the distributive laws $\alpha$, $\alpha'$, $\beta$, and $\beta'$.
\par Therefore, the $2$-functor $\Delta_\co$ sends a tangentad $(\X,\TT)$ to:
\begin{align*}
&\Delta_\co(\X,\TT)\=(\X;\TT,\TT,c)
\end{align*}
Moreover, it sends a colax tangent morphism $(F,\alpha)$ to
\begin{align*}
&\Delta(F,\alpha)\=(F,\alpha,\alpha)
\end{align*}
and a $2$-morphism $\varphi$ to:
\begin{align*}
&\Delta(\varphi)\=\varphi
\end{align*}

\begin{lemma}
\label{lemma:comultiplication-comonad-TNG}
The $2$-functor $\Delta_\co\colon\Tng_\co(\CC)\to\Tng_\co(\Tng_\co(\CC))$ preserves pointwise pullbacks.
\end{lemma}
\begin{proof}
It is easy to see that the composite $(\U_\co)_{\Tng_\co}\o\Delta_\co\colon\Tng_\co(\CC)\to\Tng_\co(\CC)$ is the identity $2$-functor, which trivially preserves pointwise pullbacks. Therefore, by Lemma~\ref{lemma:pullback-preserving}, $\Delta_\co$ preserves pointwise pullbacks.
\end{proof}

\begin{theorem}
\label{theorem:comonad-structure-TNG}
The $2$-endofunctor $\Tng_\co$ on the $2$-category $\TwoCat_\pp$ comes equipped with a $2$-comonad structure whose counit is the forgetful $2$-functor $\U_\co\colon\Tng_\co(\CC)\to\CC$ and whose comultiplication is the $2$-functor $\Delta_\co\colon\Tng_\co(\CC)\to\Tng_\co(\Tng_\co(\CC))$.
\end{theorem}
\begin{proof}
Lemma~\ref{lemma:comultiplication-comonad-TNG} proves that $\Delta_\co$ is pointwise pullback-preserving. Furthermore, $\U_\co$ is also pointwise pullback-preserving since it is a $2$-right adjoint by Proposition~\ref{proposition:Inc-U-adjunction} and thanks to Lemma~\ref{lemma:right-adjoints-preserve-pointwise-limits}. Let us prove that $\U_\co$ and $\Delta_\co$ define a comonad structure for $\Tng_\co$. First, notice that $(\U_\co)_{\Tng_\co}$ and $\Tng_\co(\U_\co)$ project $(\X;\TT,\TT',\gamma)$ to $(\X,\TT)$ and $(\X,\TT')$, respectively. Similarly, they send a morphism $(F,\alpha;\beta)$ to $(F,\alpha)$ and $(F,\beta)$, respectively and a $2$-morphism $\varphi$ to itself. Thus, it is not hard to see that, $\Delta_\co$ composed with either $(\U_\co)_{\Tng_\co}$ or $\Tng_\co(\U_\co)$ is just the identity. Let us prove that $\Delta_\co$ is coassociative. Consider a tangentad $(\X,\TT)$. Thus:
\begin{align*}
&\Tng_\co[\Delta_\co](\Delta_\co(\X,\TT))=\Tng_\co[\Delta_\co](\X;\TT,\TT,c)=(\Delta_\co(\X,\TT);\Delta_\co(\TT,c))\\
&\quad=(\X;\TT,\TT,c;\TT,c,c)=(\X;\TT,\TT,\TT,c,c,c)
\end{align*}
where we denote the generic object of $\Tng_\co^3(\CC)$ as $(\X;\TT,\TT',\TT'',\gamma,\gamma',\gamma'')$. Conversely:
\begin{align*}
&(\Delta_\co)_{\Tng_\co}(\Delta_\co(\X,\TT))=(\Delta_\co)_{\Tng_\co}(\X;\TT,\TT,c)=(\X;\TT,\TT,\TT,c,c,c)
\end{align*}
One can prove a similar result for colax tangent morphisms and $2$-morphisms, concluding that $\Delta_\co$ is coassociative.
\end{proof}

\begin{remark}
\label{remark:TNG-not-monad}
In~\cite{street:formal-theory-monads}, Street showed that the $2$-functor $\Mnd$ which sends a $2$-category $\CC$ to the associated $2$-category of monads of $\CC$ has a $2$-monad structure, whose unit sends an object $\X$ of $\CC$ to the trivial monad $(\X,\1)$ on $\X$ and whose multiplication takes a distributive law $(\X;S,S',\gamma)$ between two monads (which is equivalent to a monad in $\Mnd(\CC)$) and composes the two monads together by the distributive law $S'\o_\gamma S$. $\Tng_\co$, and similarly $\Tng_\cong$ and $\Tng_=$, has a unit, corresponding to the $2$-functor $\Inc$. Moreover, the objects of $\Tng_\co^2(\CC)$ look similar to distributive laws: they are pairs of tangent structures together with a distributive law of tangent structures between them. So it is natural to wonder whether or not, $\Tng_\co$ could come with a multiplication. However, by composing two tangent structures on an object $\X$ with a distributive law one does not obtain a tangent structure on $\X$. A counterexample is offered by the composition of the tangent bundle functor $\T$ with itself via $c$. The result is not a tangent structure on $\X$. Therefore, $\Tng_\co$ fails to be a $2$-monad.
\end{remark}

In future work, we would like to explore the classification of lax, colax, and pseudo coalgebras of the $2$-comonad $\Tng_\co$.

\subsection{Cartesian tangentads}
\label{subsection:Cartesian-tangent-objects}
In~\cite[Definition~2.8]{cockett:tangent-cats}, Cockett and Cruttwell introduced Cartesian tangent categories as tangent categories with finite products (in particular, with a terminal object), preserved by the tangent bundle functor.

\par This section is devoted to extending this notion for tangentads. For starters, we recall the notion of a Cartesian object in a $2$-category. In this section, the ambient $2$-category $\CC$ is assumed to have a strict $2$-terminal object denoted by $\*$ and finite strict $2$-products.

\begin{definition}
\label{definition:Cartesian-object}
A \textbf{Cartesian object} of $\CC$ is an object $\X$ of $\CC$ for which the unique $1$-morphism $!\colon\X\to\*$ and the diagonal $1$-morphism $\Delta\colon\X\to\X\times\X$ admit right adjoints in $\CC$, denoted by
\begin{align*}
\1&\colon\*\to\X        &\times&\colon\X\times\X\to\X
\end{align*}
respectively.
\end{definition}

Cartesian objects of $\CC$ form a $2$-category, denoted by $\Cart(\CC)$, whose $1$-morphisms are $1$-morphisms $F\colon\X\to\X'$ of $\CC$ such that, the two canonical $2$-morphisms:
\begin{equation*}
\begin{tikzcd}
\X & \X \\
\ast & {\X'}
\arrow[equals, from=1-1, to=1-2]
\arrow["{\alpha^\1}"', Rightarrow, from=1-2, to=2-1]
\arrow["F", from=1-2, to=2-2]
\arrow["\1", from=2-1, to=1-1]
\arrow["{\1'}"', from=2-1, to=2-2]
\end{tikzcd}\hfill\quad
\begin{tikzcd}
{\X\times\X} & \X \\
{\X'\times\X'} & {\X'}
\arrow["\times", from=1-1, to=1-2]
\arrow["{F\times F}"', from=1-1, to=2-1]
\arrow["{\alpha^\times}"', Rightarrow, from=1-2, to=2-1]
\arrow["F", from=1-2, to=2-2]
\arrow["{\times'}"', from=2-1, to=2-2]
\end{tikzcd}
\end{equation*}
are invertible. Notice that $\alpha^\1\colon\T\o\1\Rightarrow\1'$ and $\alpha^\times\colon F\o\times\Rightarrow\times\o(F\times F)$ are defined as follows:
\begin{align*}
&\alpha^\1\colon F\o\1\xrightarrow{\eta^\1 F\1}\1\:\o\:!\o F\o\1=\1\:\o\:!\o\1\xrightarrow{\1\epsilon^\1}\1\\
&\alpha^\times\colon F\o\times\xrightarrow{\eta^\times 
F\times}\times\o\Delta\o F\o\times=\times\o(F\times F)\o\Delta\o\times\xrightarrow{\times(F\times F)\epsilon^\times}\times\o(F\times F)
\end{align*}
where:
\begin{equation}
\begin{aligned}
\label{equation:units-counits-Cartesian}
&\eta^\1\colon\id_\X\Rightarrow\1\:\o\:!        &&\eta^\times\colon\id_\X\Rightarrow\times\o\Delta\\
&\epsilon^\1\colon!\o\1\Rightarrow\id_\1    &&\epsilon^\times\colon\Delta\o\times\Rightarrow\id_{\X\times\X}
\end{aligned}
\end{equation}
are the units and the counits of the adjunctions $\1\leftrightarrows\X\colon!$ and $\times\colon\X\times\X\leftrightarrows\X\colon\Delta$.

\par Furthermore, $2$-morphisms $\varphi\colon F\Rightarrow G\colon\X\to\X'$ of Cartesian objects of $\CC$ are $2$-morphisms of $\CC$ that commute with the $2$-isomorphisms $\alpha^\1$ and $\alpha^\times$:
\begin{equation*}
\begin{tikzcd}
{F\o\1} & {\1'} \\
{G\o\1}
\arrow["{!}", from=1-1, to=1-2]
\arrow["{\varphi\1}"', from=1-1, to=2-1]
\arrow["{!}"', from=2-1, to=1-2]
\end{tikzcd}\hfill\quad
\begin{tikzcd}
{F\o\times} & {\times'\o(F\times F)} \\
{G\o\times} & {\times'\o(G\times G)}
\arrow["\alpha^\times", from=1-1, to=1-2]
\arrow["{\varphi\times}"', from=1-1, to=2-1]
\arrow["{\times'(\varphi\times\varphi)}", from=1-2, to=2-2]
\arrow["\alpha^\times"', from=2-1, to=2-2]
\end{tikzcd}
\end{equation*}

\begin{proposition}
\label{proposition:Cartesian-tangent-categories}
A Cartesian tangent category $(\X,\TT)$ is equivalent to a tangentad in the $2$-category of Cartesian objects $\Cart(\CC)$ of $\CC$.
\end{proposition}

\begin{proposition}
\label{proposition:Cartesian-vs-tangent}
If $\CC$ has finite $2$-products and a $2$-terminal object, so does the $2$-category $\Tng(\CC)$. Moreover, under these conditions, the $2$-category $\Cart(\Tng(\CC))$ of Cartesian objects in the $2$-category of tangentads of $\CC$ is isomorphic to the $2$-category $\Tng(\Cart(\CC))$ of tangentads in the $2$-category of Cartesian objects of $\CC$.
\end{proposition}
\begin{proof}
First, it is easy to show that, if $\CC$ has a terminal object $\*$, then $\*\=(\*,\1)$, the trivial tangentad on the terminal object $\*$ of $\CC$, is terminal in $\Tng(\CC)$, and that if $\CC$ has finite products, then $(\X\times\X',\TT\times\TT')$ defines the Cartesian product between $(\X,\TT)$ and $(\X',\TT')$ in $\Tng(\CC)$, where the tangent bundle functor of $\TT\times\TT'$ is $\T\times\T'$ and the structural $2$-morphisms are defined in a similar manner.

\par Let us consider a tangentad $(\X,\TT)$ in the $2$-category $\Cart(\CC)$ of Cartesian objects of $\CC$ and let us show that $(\X,\TT)$ is a Cartesian object in $\Tng(\CC)$. The tangent bundle $1$-morphism $\T\colon\X\to\X$ preserves the Cartesian structure of $\X$. In particular, the $2$-morphisms $\alpha^\1\colon\T\o\*\Rightarrow\*$ and $\alpha^\times\colon\T\o\times\Rightarrow\times\o(\T\times\T)$ are invertible. This allows us to define the following two strong tangent morphisms:
\begin{align*}
&\*\xrightarrow{(\1,{\alpha^\1}^{-1})}(\X,\TT)\\
&(\X\times\X,\TT\times\TT)\xrightarrow{(\times,{\alpha^\times}^{-1})}(\X,\TT)
\end{align*}
The compatibility between the distributive laws and the structural $2$-morphisms of the tangent structures are a direct consequence of the structural $2$-morphisms being $2$-morphisms in $\Cart(\CC)$. Moreover, by using the triangle identities, it is easy to show that the units and the counits of Equation~\eqref{equation:units-counits-Cartesian} are compatible with ${\alpha^\1}^{-1}$ and ${\alpha^\times}^{-1}$. This shows that:
\begin{align*}
&(\1,{\alpha^\1}^{-1})\colon\*\leftrightarrows(\X,\TT)\colon!\\
&(\times,{\alpha^\times}^{-1})\colon(\X\times\X,\TT\times\TT)\leftrightarrows(\X,\TT)\colon\Delta
\end{align*}
form two adjunctions in $\Tng(\CC)$. Therefore, $(\X,\TT)$ is a Cartesian object in $\Tng(\CC)$.\\
Conversely, consider a Cartesian object of $\Tng(\CC)$. This consists of a tangentad $(\X,\TT)$ of $\CC$ for which the $1$-morphisms $\alpha^\1\colon(\X,\TT)\to\*$ and $\Delta\colon(\X,\TT)\to(\X\times\X,\TT\times\TT)$ admit right adjoints $(\*,\kappa)\colon\*\to(\X,\TT)$ and $(\times,\beta)\colon(\X\times\X,\TT\times\TT)\to(\X,\TT)$ in $\Tng(\CC)$, respectively. Since the forgetful $2$-functor $\U\colon\Tng(\CC)\to\CC$ preserves adjunctions, $\X$ is a Cartesian object in $\CC$. The goal is to show that the two lax distributive laws:
\begin{align*}
&\kappa\colon\1\Rightarrow\T\o\1\\
&\beta\colon\times\o(\T\times\T)\Rightarrow\T\o\times
\end{align*}
invert $\alpha^\1$ and $\alpha^\times$, respectively, which are induced by the Cartesian sructure of $\X$.

\par Let us start with $\kappa$. First, notice that the unit $\eta^\1$ of the adjunction $(\1,\kappa)\dashv\:!$ is compatible with $\kappa$, since the adjunction is in $\Tng(\CC)$:
\begin{equation*}
\begin{tikzcd}
\T & {\1\:\o\:!\o\T} & {\1\:\o\:!} \\
\T && {\T\o\1\:\o\:!}
\arrow["{\eta\T}", from=1-1, to=1-2]
\arrow[equals, from=1-1, to=2-1]
\arrow[equals, from=1-2, to=1-3]
\arrow["{\kappa!}", from=1-3, to=2-3]
\arrow["{\T\eta^\1}"', from=2-1, to=2-3]
\end{tikzcd}
\end{equation*}
By the universal property of the terminal object $\*$ of $\CC$, $\kappa\alpha^\1\colon\1\Rightarrow\1$ is the identity. Let us show that also $\alpha^\times\kappa\colon\T\o\1\Rightarrow\T\o\1$ is the identity:
\begin{align*}
&\quad\kappa\o\:!
&&(!=\epsilon^\1\:\o\1\:\o\eta_{\T\1}^\1)\\
=&\quad\kappa\o\epsilon^\1\:\o\1\:\o\eta_{\T\1}^\1
&&(\kappa\o\epsilon^\1\:\o\1=\epsilon^\1\:\o\T\1\:\o\kappa_{!\1})\\
=&\quad\epsilon^\1\:\o\T\1\:\o\kappa_{!\1}\o\eta^\1_{\T\1}
&&(!\o\1=!\o\T\1)\\
=&\quad\epsilon^\1\:\o\T\1\:\o\kappa_{!\T\1}\o\eta^\1_{\T\1}
&&(\kappa_!\o\eta^\1_\T=\T\eta^\1)\\
=&\quad\epsilon^\1\:\o\T\1\:\o\T\eta^\1_\1
&&(\text{triangle identity})\\
=&\quad\id_{\T\1}
\end{align*}
Let us now prove that $\beta$ inverts $\alpha^\times$. Let us start by noticing that the unit $\eta^\times$ and the counit $\epsilon^\times$ of the adjunction $\times\dashv\Delta$ are compatible with $\beta$:
\begin{equation*}
\begin{tikzcd}
\T & {\times\o\Delta\o\T} & {\times\o(\T\times\T)\o\Delta} \\
\T && {\T\o\Delta}
\arrow["{\eta^\times\T}", from=1-1, to=1-2]
\arrow[Rightarrow, no head, from=1-1, to=2-1]
\arrow[Rightarrow, no head, from=1-2, to=1-3]
\arrow["{\beta\Delta}", from=1-3, to=2-3]
\arrow["{\T\eta^\times}"', from=2-1, to=2-3]
\end{tikzcd}\hfill\quad
\begin{tikzcd}
{\Delta\o\times\o(\T\times\T)} & {\T\times\T} \\
{\Delta\o\T\o\times} & {(\T\times\T)\o\Delta\o\times}
\arrow["{\epsilon(\T\times\T)}", from=1-1, to=1-2]
\arrow["{\Delta\beta}"', from=1-1, to=2-1]
\arrow[Rightarrow, no head, from=2-1, to=2-2]
\arrow["{(\T\times\T)\epsilon}"', from=2-2, to=1-2]
\end{tikzcd}
\end{equation*}
Thus:
\begin{align*}
&\quad\alpha^\times\o\beta
&&(\alpha^\times=(\times(\T\times\T)\epsilon^\times)\o\eta^\times_{\T\times})\\
=&\quad(\times(\T\times\T)\epsilon^\times)\o\eta^\times_{\T\times}\o\beta
&&(\eta^\times_{\T\times}\o\beta=\Delta\beta\o\times\o\eta^\times_{\times(\T\times\T)})\\
=&\quad(\times(\T\times\T)\epsilon^\times)\o\Delta\beta\o\times\o\eta^\times_{\times(\T\times\T)}
&&(((\T\times\T)\times\epsilon^\times)\o\Delta\beta=\epsilon^\times_{\T\times\T})\\
=&\quad((\T\times\T)\times\epsilon^\times)\o\eta^\times_{\times(\T\times\T)}
&&(\text{triangle identity})\\
=&\quad\id_{\times(\T\times\T)}
\end{align*}
Moreover:
\begin{align*}
&\quad\beta\o\alpha^\times
&&(\alpha^\times=\epsilon^\times\o(\T\times\T)\o\times\o\eta^\times_{\T\times})\\
=&\quad\beta\o\epsilon^\times\o(\T\times\T)\o\times\o\eta^\times_{\T\times}
&&(\beta\o\epsilon^\times\o(\T\times\T)\o\times=(\T\times\epsilon^\times)\o\beta_{\Delta\times})\\
=&\quad(\T\times\epsilon^\times)\o\beta_{\Delta\times}\o\eta^\times_{\T\times}
&&\beta_\Delta\o\eta^\times_\T=\T\eta^\times_\times)\\
=&\quad(\T\times\epsilon^\times)\o\T\eta^\times_\times
&&(\text{triangle identity})\\
=&\quad\id_{\T\times}
\end{align*}
This proves that $(\X,\TT)$ is a tangentad in $\Cart(\X,\TT)$. This correspondence extends to a $2$-isomorphism between the $2$-categories $\Cart(\Tng(\CC))$ and $\Tng(\Cart(\CC))$.
\end{proof}

Propositions~\ref{proposition:Cartesian-tangent-categories} and~\ref{proposition:Cartesian-vs-tangent} suggest defining a Cartesian tangentad as follows.

\begin{definition}
\label{definition:Cartesian-tangent-object}
A \textbf{Cartesian tangentad} in $\CC$ is a Cartesian object of the $2$-category $\Tng(\CC)$ of tangentads of $\CC$, or equivalently, a tangentad of the $2$-category $\Cart(\CC)$ of Cartesian objects of $\CC$.
\end{definition}

\subsection{Adjunctable and representable tangentads}
\label{subsection:adjunctable-representable-tangent-objects}
\cite[Proposition~5.17]{cockett:tangent-cats} states that the opposite category of a tangent category $(\X,\TT)$ comes equipped with a ``\textit{dual}'' tangent structure provided the tangent bundle functor $\T$ and each functor $\T_n$ admit left adjoints, $\T^\b$, and $\T^\b_n$, respectively.
\par In this section, we formalize this construction for tangentads in a $2$-category which admits a suitable notion of ``\emph{opposite}'' objects. Let us start by introducing this concept.

\begin{definition}
\label{definition:opposite-objects}
A $2$-category with \textbf{opposite objects} comprises a $2$-category $\CC$, a $2$-equivalence $(-)^\op\colon\CC^\co\to\CC$ from the $2$-cocategory of $\CC$, namely, with inverted $2$-morphisms, to $\CC$, and a $2$-isomorphism:
\begin{equation*}
\begin{tikzcd}
{\CC^\co} & \CC \\
& {\CC^\co}
\arrow["{(-)^\op}", from=1-1, to=1-2]
\arrow[""{name=0, anchor=center, inner sep=0}, equals, from=1-1, to=2-2]
\arrow["{((-)^\op)^\co}", from=1-2, to=2-2]
\arrow["\xi"', shorten <=2pt, Rightarrow, from=0, to=1-2]
\end{tikzcd}
\end{equation*}
\end{definition}

\begin{example}
\label{example:categories}
The $2$-category $\Cat$ of categories admits opposite objects once equipped with the $2$-functor $(-)^\op$ which sends each category to its opposite, where $((-)^\op)^\co\o(-)^\op=\id_\Cat$.
\end{example}

\begin{example}
\label{example:adjunctable-monad}
Let $\CC$ be a $2$-category with opposite objects $(-)^\op$, and consider the full $2$-subcategory $\LMnd(\CC)$ of the $2$-category $\Mnd(\CC)$ of monads of $\CC$ spanned by those monads $(\X,S)$ of $\CC$ whose underlying endomorphism $S\colon\X\to\X$ admits a right adjoint $S^\b\colon\X\to\X$. Let $\eta\colon\id_\X\Rightarrow S^\b\o S$ and $\epsilon\colon S\o S^\b\Rightarrow\id_\X$ denote the unit and the counit of the adjunction, respectively. Using mates, one can define a comonad structure on $S^\b$ as follows:
\begin{align*}
&1^\b\colon S^\b\xrightarrow{1S^\b}S\o S^\b\xrightarrow{\epsilon}\id_\X\\
&\mu^\b\colon S^\b\xrightarrow{\eta S^\b}S^\b\o S\o S^\b\xrightarrow{S^\b\eta SS^\b}S^\b\o S^\b\o S\o S\o S^\b\xrightarrow{S^\b S^\b\mu S^\b}S^\b\o S^\b\o S\o S^\b\xrightarrow{S^\b S^\b\epsilon}S^\b\o S^\b
\end{align*}
where $1\colon\id_\X\Rightarrow S$ and $\mu\colon S^2\Rightarrow S$ denote the unit and the multiplication of $S$, respectively. Since $\CC$ admits opposite objects, $(-)^\op$ sends $\X$ to $\X^\op$, $S^\b$ to $(S^\b)^\op\colon\X^\op\to\X^\op$, and the comonad of $S^\b$ to a monad structure:
\begin{align*}
&(1^\b)^\op\colon\id_{\X^\op}\Rightarrow(S^\b)^\op\\
&(\mu^\b)^\op\colon(S^\b)^\op\o(S^\b)^\op\to(S^\b)^\op
\end{align*}
on $(S^\b)^\b$. Furthermore, $(-)^\op$ sends adjunctions to (co)adjunctions, therefore, $S^\op$ becomes a right adjoint to $(S^\b)^\op$, where the unit and the counit of the adjunction are given respectively by $\epsilon^\op$ and $\eta^\op$ (notice that the unit becomes the counit and vice versa). Let us define $(\X;S,1,\mu)^\op\=(\X^\b;(S^\b)^\op,(1^\b)^\op,(\mu^\b)^\op)$.
\par Given a morphism $(F,\alpha)\colon(\X,S)\to(\X',S')$ of monads, let the underlying $\CC$ $1$-morphism of the corresponding morphism of monads $(F,\alpha)^\op\colon(\X,S)^\op\to(\X',S')^\op$ be $F^\op$ and the corresponding distributive law to be $(\alpha^\b)^\op\colon(S^\b)^\op\o F^\op\Rightarrow F^\op\o(S^\b)^\op$, where $\alpha^\b$ is the mate of $\alpha\colon S\o F\Rightarrow F\o S$ along the adjunction:
\begin{align*}
&\alpha^\b\colon F\o S^\b\xrightarrow{\eta FS^\b}S^\b\o S\o F\o S^\b\xrightarrow{S^\b\alpha S^\b}S^\b\o F\o S\o S^\b\xrightarrow{S^\b\o F\epsilon}S^\b\o F
\end{align*}
Since the mate of a pasting diagram is the pasting diagram of the mates of each subdiagram provided the mate of each of them is well-defined (cf.~\cite{kelly:2-categories}), $(F,\alpha^\b)\colon(\X,S^\b)\to(\X',{S'}^\b)$ becomes a morphism of comonads and therefore, $(F^\op,(\alpha)^\op)$ is a morphism of monads. Finally, if $\varphi\colon(F,\alpha)\Rightarrow(G,\beta)$ is a $2$-morphism of monads, one can show that $\varphi$ is also a morphism of comonads $\varphi\colon(F,\alpha^\b)\Rightarrow(G,\beta^\b)\colon(\X,S^\b)\to(\X',{S'}^\b)$. Thus, $\varphi^\op\colon(F,\alpha)^\op\Rightarrow(G,\beta)^\op\colon(\X,S)^\op\to(\X',S')^\op$ becomes a morphism of monads.
\par Finally, we need to provide a $2$-isomorphism $\xi\colon\id_{\CC^\co}\Rightarrow((-)^\op)^\co)\o(-)^\op$. However, $(\X^\op,{S^\b}^\op)^\op$ is given by $(\X^\op)^\op$ and $(((S^\b)^\op)^\b)^\op$. However, since mates are preserved by $2$-functors, the monad becomes $(((S^\b)^\b)^\op)^\op$. By using triangle identities, one can show that $(S^\b)^\b$ is the original monad $S$. Therefore, $\LMnd(\CC)$ admits opposite objects provided that $\CC$ does.
\par Notice that, the $2$-subcategory $\RMnd(\CC)$ of $\Mnd(\CC)$ of monads which admit a left adjoint does not admit opposite objects, since the distributive law of the $1$-morphisms does not admit mate. One might hope to consider the $2$-category $\RMnd_\cong(\CC)$ of monads of $\CC$ admitting a left adjoint, whose $1$-morphisms are strong morphisms of monads, namely, with invertible distributive law, and whose $2$-morphisms are just usual $2$-morphisms of monads. However, the mate of the distributive law fails to be invertible, since in general, the inverse does not admit any mate.
\end{example}

\begin{example}
\label{example:dual-indexed-category}
The $2$-category $\Indx(\X)$ of indexed categories over a fixed indexing category $\X$ admits opposite objects. Concretely, the opposite of an indexed category $I\colon\X^\op\to\Cat$ which sends each $A$ of $\X$ to $\X_A$ and each $f\colon A\to B$ to $f^\*\colon\X_B\to\X_A$ can be defined as the indexed category $I^\b\colon\X^\op\to\Cat$ which sends $A$ to $\X_A^\op$ and $f\colon A\to B$ to $(f^\*)^\op\colon\X_B^\op\to\X_A^\op$. Similarly, the opposite of an indexed functor $F\colon I\to I'$ which sends each $A$ to a functor $F_A\colon\X_A\to\X_A'$ is the indexed functor $F^\b\colon I^\b\to{I'}^\b$ which sends each $A$ to $F_A^\op\X_A^\op\to{\X_A'}^\op$. Finally, the opposite of an indexed natural transformation $\varphi\colon F\Rightarrow G\colon I\to I'$ which sends $A$ to a natural transformation $\varphi_A\colon F_A\Rightarrow G_A\colon\X_A\to\X_A'$ is the indexed natural transformation $\varphi^\b\colon F^\b\Rightarrow G^\b\colon I^\b\to{I'}^\b$ which sends $A$ to $\varphi^\op_A\colon F_A^\op\leftarrow G_A^\op$.
\end{example}

\begin{example}
\label{example:dual-fibration}
Thanks to the Grothendieck $2$-equivalence $\Indx(\X)\simeq\Fib(\X)$, the construction of Example~\ref{example:dual-indexed-category} induces opposite objects on the $2$-category $\Fib(\X)$ of cloven fibrations over the fixed base category $\X$. In particular, if $\Pi\colon\X'\to\X$ is a fibration over $\X$, the fibres of $\Pi^\b$ are the opposite categories of the fibres of $\Pi$. A more general construction of the \textit{dual} of a (non-necessarily cloven) fibration is described in~\cite{kock:dual-fibration}.
\end{example}

\begin{definition}
\label{definition:adjunctable-tangent-structure}
A tangent structure $\TT$ on an object $\X$ of a $2$-category $\CC$ is \textbf{adjunctable} provided that the tangent bundle $1$-morphism $\T$ and each $1$-morphism $\T_n$ admit left adjoints $\T^\b$ and $\T^\b_n$ in the $2$-category $\CC$, respectively. By extension, a tangentad $(\X,\TT)$ is adjunctable provided its tangent structure is.
\end{definition}

\begin{definition}
\label{definition:op-tangent-structure}
An \textbf{op-tangent structure} on an object $\X$ of a $2$-category $\CC$ is a tangent structure on $\X$ in the $2$-cocategory $\CC^\co$ of $\CC$. An op-tangent structure \textbf{admits conegatives} if the corresponding tangent structure on $\CC^\co$ admits negatives.
\end{definition}

\begin{proposition}
\label{proposition:adjunctable-tangent-structure}
Let $(\X,\TT)$ be an adjunctable tangentad, where
\begin{align*}
\eta_n&\colon\id_\X\Rightarrow\T_n\o\T_n^\b         &\epsilon_n&\colon\T_n^\b\o\T_n\Rightarrow\id_\X
\end{align*}
denote the units and the counits of the adjunctions $\T_n^\b\dashv\T_n$ ($\eta\=\eta_1$ and $\epsilon\=\epsilon_1$), respectively. Thus, $\X$ carries the structure of an op-tangentad $(\X,\TT^\b)$ so defined:
\begin{description}
\item[Tangent bundle $1$-morphism] The tangent bundle $1$-morphism of $\TT^\b$ is the  left adjoint $\T^\b$ to $\T$;

\item[Coprojection] The coprojection $p^\b\colon\id_\X\Rightarrow\T^\b$ is the mate of the projection $p\colon\T\Rightarrow\id_\X$:
\begin{align*}
&p^\b\colon\id_\X\xrightarrow{\eta}\T\o\T^\b\xrightarrow{p\T^\b}\T^\b
\end{align*}

\item[Cozero morphism] The cozero morphism $z^\b\colon\T^\b\Rightarrow\id_\X$ is the mate of the zero morphism $z\colon\id_\X\Rightarrow\T$:
\begin{align*}
&z^\b\colon\T^\b\xrightarrow{\T^\b z}\T^\b\o\T\xrightarrow{\epsilon}\id_\X
\end{align*}

\item[Cosum morphism] The cosum morphism $s^\b\colon\T^\b\Rightarrow\T_2^\b$ is the mate of the sum morphism $s\colon\T_2\Rightarrow\T$:
\begin{align*}
&s^\b\colon\T^\b\xrightarrow{\T^\b\eta_2}\T^\b\o\T_2\o\T^\b_2\xrightarrow{\T^\b s\T^\b_2}\T^\b\o\T\o\T^\b_2\xrightarrow{\epsilon\T^\b_2}\T^\b_2
\end{align*}

\item[Vertical colift] The vertical colift $l^\b\colon{\T^\b}^2\Rightarrow\T^\b$ is the mate of the vertical lift $l\colon\T\Rightarrow\T^2$:
\begin{align*}
&l^\b\colon\T^\b\o\T^\b\xrightarrow{\T^\b\T^\b\eta}\T^\b\o\T^\b\o\T\o\T^\b\xrightarrow{\T^\b\T^\b l\T^\b}\T^\b\o\T^\b\o\T\o\T\o\T^\b\xrightarrow{\T^\b\epsilon\T\T^\b}\T^\b\o\T\o\T^\b\xrightarrow{\epsilon\T^\b}\T^\b
\end{align*}

\item[Canonical coflip] The canonical coflip $c^\b\colon{\T^\b}^2\Rightarrow{\T^\b}^2$ is the mate of the canonical flip $c\colon\T^2\Rightarrow\T^2$:
\begin{align*}
&c^\b\colon\T^\b\o\T^\b\xrightarrow{\T^\b\T^\b\eta}\T^\b\o\T^\b\o\T\o\T^\b\xrightarrow{\T^\b\T^\b\T\eta\T^\b}\T^\b\o\T^\b\o\T\o\T\o\T^\b\o\T^\b\to\\
&\qquad\xrightarrow{\T^\b\T^\b c\T^\b\T^\b}\T^\b\o\T^\b\o\T\o\T\o\T^\b\o\T^\b\xrightarrow{\T^\b\epsilon\T\T^\b\T^\b}\T^\b\o\T\o\T^\b\o\T^\b\xrightarrow{\epsilon\T^\b\T^\b}\T^\b\o\T^\b
\end{align*}
\end{description}
Furthermore, if $(\X,\TT)$ has negatives,$\TT^\b$ admits conegatives:
\begin{description}
\item[Conegation] The conegation $n^\b\colon\T^\b\Rightarrow\T^\b$ is the mate of the negation $n\colon\T\Rightarrow\T$:
\begin{align*}
&n^\b\colon\T^\b\xrightarrow{\T^\b\eta}\T^\b\o\T\o\T^\b\xrightarrow{\T^\b n\T^\b}\T^\b\o\T\o\T^\b\xrightarrow{\epsilon\T^\b}\T^\b
\end{align*}
\end{description}
\end{proposition}
\begin{proof}
The proof of this result is a straightforward generalization of that of~\cite[Proposition~5.17]{cockett:tangent-cats}. In particular, since mates preserve pasting diagrams, all equational axioms hold. To prove that $\T_n^\b$ is the $n$-fold pushout of $p^\b$ of itself, consider $n$ $2$-morphisms $f_1\,f_n\colon\T^\b\to H$ such that $f_i\o p^\b=f_j\o p^\b$, for each $i,j=1\,n$. Let us define:
\begin{align*}
&f_k^\b\colon\id_\X\xrightarrow{\eta}\T\o\T^\b\xrightarrow{\T f_k}\T\o H
\end{align*}
Since mates preserve pasting diagrams:
\begin{align*}
&p\o f_i^\b=(p^\b\o f_i)^\b=(p^\b\o f_j)^\b=p\o f_j^\b
\end{align*}
Thus, we can define the unique morphism $\<f_1^\b\,f_n^\b\>\colon\id_\X\to\T_n\o H$. It is not hard to prove that the mate of the morphism
\begin{align*}
&[f_1\,f_n]\colon\T_n^\b\xrightarrow{\T_n^\b\<f_1^\b\,f_n^\b\>}\T^\b_n\o\T_n\o H\xrightarrow{\epsilon_nH}H
\end{align*}
is the unique morphism, for which $[f_1\,f_n]\o\pi_k^\b=f_k$, where $\pi_k^\b\colon\T^\b\to\T^\b_n$ is the mate of $\pi_k\colon\T_n\to\T$. Similarly, the universality of the vertical lift is preserved by mates.
\end{proof}

When a $2$-category admits opposite objects, the op-tangent structure $\TT^\b$ of Proposition~\ref{proposition:adjunctable-tangent-structure} of an adjunctable tangentad $(\X,\TT)$ becomes a tangent structure on the opposite object of $\X$. To see this, recall that a tangentad $(\X,\TT)$ can be defined by a Leung functor:
\begin{align*}
&\Leung[\TT]\colon\Weil\to\End_\CC(\X)=\CC(\X,\X)
\end{align*}
Similarly, an op-tangentad $(\X,\TT^\b)$ can be regarded as a Leung functor:
\begin{align*}
&\Leung[\TT^\b]\colon\Weil\to\End_{\CC^\co}(\X)=\CC^\co(\X,\X)
\end{align*}
When $\CC$ admits opposite objects, we have an equivalence of monoidal categories $\CC^\co(\X,\X)\simeq\CC(\X^\op,\X^\op)$. Thus, every op-tangentad on $\X$ defines a tangentad on $\X^\op$ whose Leung functor is given by:
\begin{align*}
&\Leung[\TT]\colon\Weil\xrightarrow{\Leung[\TT^\b]}\End_{\CC^\co}(\X)=\CC^\co(\X,\X)\simeq\CC(\X^\op,\X^\op)=\End_\CC(\X^\op)
\end{align*}

\begin{lemma}
\label{lemma:op-tangent-objects}
If $(\X,\TT^\b)$ is an op-tangentad of $\CC$ and $\CC$ admits opposite objects, then $(\X,\TT)$ defines a tangentad, where:
\begin{align*}
&\Leung[\TT]\colon\Weil\xrightarrow{\Leung[\TT^\b]}\End_{\CC^\co}(\X)=\CC^\co(\X,\X)\simeq\CC(\X^\op,\X^\op)=\End_\CC(\X)
\end{align*}
\end{lemma}
\begin{proof}
Since $(-)^\op$ is an equivalence, the induced functor $\CC^\co(\X,\X)\to\CC(\X^\op,\X^\op)$ is a strong monoidal functor and preserves pointwise limits; thus, $\Leung[\TT]$ is a Leung functor.
\end{proof}

\begin{theorem}
\label{theorem:opposite-tangent-objects}
If the $2$-category $\CC$ admits opposite objects, so does the full $2$-subcategory $\AdjTng(\CC)$ of $\Tng(\CC)$ of adjunctable tangentads, where the opposite of an adjunctable tangentad $(\X,\TT)$ is defined by $(\X^\op,(\TT^\b)^\op)$ and $\TT^\b$ is the op-tangent structure of Proposition~\ref{proposition:adjunctable-tangent-structure}.
\end{theorem}
\begin{proof}
Since $(-)^\op$ sends adjunctions to (co)adjunctions, if $(\X,\TT)$ is adjunctable, so is $(\X^\op,(\TT^\b)^\op)$. Moreover, by employing the triangle identities of the unit and the counit, one can show that $(\TT^\b)^\b$ coincides with $\TT$. Thus, $((\X,\TT)^\op)^\op\cong(\X,\TT)$. Given a lax morphism of adjunctable tangentads $(F,\alpha)\colon(\X,\TT)\to(\X',\TT')$, by taking the mate $\alpha^\b$ of the distributive law $\alpha\colon F\o\T\Rightarrow\T'\o F$, one defines a colax op-tangent morphism $(F,\alpha^\b)\colon(\X,\TT^\b)\to(\X',{\TT'}^\b)$ between the corresponding op-tangent morphisms. In particular, $\alpha^\b$ is of type ${\T'}^\b\o F\Rightarrow F\o\T^\b$. By taking the opposite, one obtains a lax tangent morphism:
\begin{align*}
&(F^\op,(\alpha^\b)^\op)\colon(\X^\op,(\TT^\b)^\op)\to({\X'}^\op,({\TT'}^\b)^\op)
\end{align*}
Finally, a $2$-morphism $\varphi\colon(F,\alpha)\Rightarrow(G,\beta)$ of tangentads defines a $2$-morphism of tangentads $\varphi^\op\colon(F^\op,(\alpha^\b)^\op)\Rightarrow(G^\op,(\beta^\b)^\op)$.
\end{proof}

An important class of adjunctable tangent categories is the one of \textit{representable} tangent categories, whose tangent bundle functor is a representable functor. Here, we extend this notion to Cartesian tangentads.

\begin{definition}
\label{definition:representable-tangent-object}
A \textbf{representable tangentad} of a Cartesian $2$-category $\CC$ is a Cartesian tangentad $(\X,\TT)$ of $\CC$ equipped with a sequence of $1$-morphisms $D_n\colon\*\to\X$ of $\CC$ from the terminal object $\*$ of $\CC$ to $\X$ such that, for every positive integer $n$, the $1$-morphism:
\begin{align*}
&D_n\times-\colon\X\cong\*\times\X\xrightarrow{D_n\times\id_\X}\X\times\X\xrightarrow{\times}\X
\end{align*}
is a left adjoint to the $1$-morphism $\T_n\colon\X\to\X$.
\end{definition}

\begin{corollary}
\label{corollary:opposite-representable}
When $\CC$ is a Cartesian $2$-category with opposite objects, the opposite of a representable tangentad is an adjunctable tangentad.
\end{corollary}

Let $\El(\X)$ denote the category of global elements of $\X$, namely, the category whose objects are $1$-morphisms $a\colon\*\to\X$ and morphisms are $2$-morphisms $a\Rightarrow b\colon\*\to\X$.

\begin{definition}
\label{definition:infinitesimal-element}
A \textbf{microscopic element} of a Cartesian object $\X$ of a Cartesian $2$-category $\CC$ consists of an element $D\colon\*\to\X$ together with the following data:
\begin{description}
\item[Coprojection] A coprojection $\zeta\colon\1\to D$, where $\1\colon\*\to\X$ denotes the terminal element of $\X$, and for which the $n$-fold pointwise pushout $D_n\colon\1\to\X$ of $\zeta$ along itself is well-defined in the category $\El(\X)$ of elements of $\X$ and preserved by each functor $D^n\times-$, with $D^n\=D\times\dots\times D\colon\1\to\X$;

\item[Cosum morphism] A cosum $\delta\colon D\to D_2$, which makes $\zeta$ a cocommutative comonoid of the coslice category $\1/\El(\X)$;

\item[Multiplication] A multiplication $\mu\colon D\times D\to D$, which is coassociative, cocommutative, compatible with the coprojection, coadditive with respect to the cosum, and which satisfies the following condition:
\begin{equation*}
\begin{tikzcd}
D & {D\times D} & D \\
{\*} & D & {\*}
\arrow["{\<!\zeta,\id_D\>}", from=1-1, to=1-2]
\arrow["{!}"', from=1-1, to=2-1]
\arrow["\mu", from=1-2, to=2-2]
\arrow["{\<\id_D,!\zeta\>}"', from=1-3, to=1-2]
\arrow["{!}", from=1-3, to=2-3]
\arrow["\zeta"', from=2-1, to=2-2]
\arrow["\zeta", from=2-3, to=2-2]
\end{tikzcd}
\end{equation*}
\end{description}
Furthermore, the multiplication is universal, namely, the diagram
\begin{equation*}
\begin{tikzcd}
D & {D\times D} \\
{\*} & {D_2}
\arrow["{\<\id_D,\zeta\>}", from=1-1, to=1-2]
\arrow["{!}"', from=1-1, to=2-1]
\arrow["\xi", from=1-2, to=2-2]
\arrow["{\zeta\iota_1}"', from=2-1, to=2-2]
\end{tikzcd}
\end{equation*}
is a pushout diagram, where:
\begin{align*}
&\xi\colon D\times D\xrightarrow{\id_D\times\delta}D\times D_2\xrightarrow{[\mu\iota_1,\pi_1\iota_2]}D_2
\end{align*}
A microscopic object \textbf{has negatives} when it comes equipped with an extra morphism:
\begin{description}
\item[Conegation] $\kappa\colon D\to D$ such that:
\begin{equation*}
\begin{tikzcd}
D & {D_2} \\
{\*} & D
\arrow["\delta", from=1-1, to=1-2]
\arrow["{!}"', from=1-1, to=2-1]
\arrow["{[\id_D,\kappa]}", from=1-2, to=2-2]
\arrow["\zeta"', from=2-1, to=2-2]
\end{tikzcd}
\end{equation*}
\end{description}
An \textbf{infinitesimal element} (with negatives) of $\X$ is a microscopic element $D$ (with negatives) of $\X$, such that for every positive integer $n$, $D_n\times-$ admits a left adjoint.
\end{definition}

As proved in~\cite[Proposition~5.7]{cockett:tangent-cats}, representable tangent structures are equivalent to infinitesimal objects in the underlying category. We can extend this result for tangentads.

\begin{proposition}
\label{proposition:representable}
If $(\X,\TT)$ is a representable tangentad (with negatives) and $D\colon\*\to\X$ represents the tangent bundle morphism, $D$ comes equipped with the structure of an infinitesimal element (with negatives) of $\X$. Moreover, if $D$ is an infinitesimal element of $\X$ (with negatives), then the left adjoint of $D\times-$ defines a representable tangent structure on $\X$.
\end{proposition}
\begin{proof}
If $\TT$ is representable, $\T$ is exponentiable; thus, the structural morphisms of $\TT$ contravariantly translate into structural morphisms on $D$, which equip $D$ with an infinitesimal structure.
\par Conversely, one can see that $D\times-$ together with the structure maps listed above define an op-tangent structure on $\X$, namely, a tangent structure $[D\times-]$ on $\X$ in $\CC^\co$. Moreover, since each $D_n\times-$ admits a right adjoint in $\CC$, such a tangent structure is adjunctable in $\CC^\co$. Thus, by Proposition~\ref{proposition:adjunctable-tangent-structure}, we define an op-tangent structure on $\X$ in $\CC^\co$, adjoint to $[D\times-]$, namely, a tangent structure on $\X$ in $\CC$, whose morphisms $\T_n$ are right adjoint to $D_n\times-$. Thus, such a tangent structure is representable.
\end{proof}

\subsection{Tangentads with negatives as right Kan extensions}
\label{subsection:negatives-Kan-extension}
In~\cite[Section~3.3]{cockett:tangent-cats}, Cockett and Cruttwell showed that, under mild conditions, a tangent category $(\X,\TT)$ admits a subtangent structure $\TT^-$ of $\TT$ with negatives. Here, we would like to extend this construction in the formal context of tangentads and reinterpret it formally as a suitable right Kan extension.

\par Let us start by recalling Cockett and Cruttwell's construction. Let us assume $(\X,\TT)$ to be a tangent category for which, for every object $M$ and for each positive integer $k$, the tangent pullback of $s\colon\T_2M\to\T M$ along $z\colon M\to\T M$ exists:
\begin{equation}
\label{equation:negative-pullbacks}
\begin{tikzcd}
{\T^-_kM} & {\T_{2k}M} \\
M & {\T_kM}
\arrow["e_k", from=1-1, to=1-2]
\arrow["{p_k^-}"', from=1-1, to=2-1]
\arrow["\lrcorner"{anchor=center, pos=0.125}, draw=none, from=1-1, to=2-2]
\arrow["s_k", from=1-2, to=2-2]
\arrow["z_k"', from=2-1, to=2-2]
\end{tikzcd}
\end{equation}
The tangent structure $\TT^-$ is obtained by restricting $\TT$ along the natural transformation:
\begin{align*}
&\epsilon\colon\T^-M\xrightarrow{e}\T_2M\xrightarrow{\pi_1}\T M
\end{align*}
Furthermore, $\TT^-$ comes equipped with a negation morphism defined as follows:
\begin{equation*}
\begin{tikzcd}
{\T^-M} && {\T_2M} \\
& {\T^-M} & {\T_2M} \\
& M & {\T M}
\arrow["e", from=1-1, to=1-3]
\arrow["n"{description}, dashed, from=1-1, to=2-2]
\arrow["{p^-}"', bend right, from=1-1, to=3-2]
\arrow["\tau", from=1-3, to=2-3]
\arrow["e", from=2-2, to=2-3]
\arrow["{p^-}"', from=2-2, to=3-2]
\arrow["\lrcorner"{anchor=center, pos=0.125}, draw=none, from=2-2, to=3-3]
\arrow["s", from=2-3, to=3-3]
\arrow["z"', from=3-2, to=3-3]
\end{tikzcd}
\end{equation*}
The main result of this section consists of showing that $\TT^-$ can be constructed via a right Kan extension.

\begin{proposition}
\label{proposition:negatives-via-right-kan-extension}
If a tangent category $(\X,\TT)$ admits the tangent pullbacks of Equation~\eqref{equation:negative-pullbacks}, the Leung functor:
\begin{align*}
&\Leung[\TT]\colon\Weil\to\End(\X)
\end{align*}
admits a pointwise right Kan extension:
\begin{align*}
&\Leung^-[\TT]\colon\Weil^-\to\End(\X)
\end{align*}
along the inclusion $\Weil\to\Weil^-$ and the corresponding tangent structure with negatives is precisely $\TT^-$.
\end{proposition}
\begin{proof}
Let us start by defining $\Leung^-[\TT]$ as the Leung functor associated with the tangent structure with negatives $\TT^-$, defined by Cockett and Cruttwell.

\par Since $\Weil$ is generated by all Weil algebras $W_n$, a monoidal natural transformation $\epsilon$:
\begin{equation*}
\begin{tikzcd}
\Weil && {\End(\X)} \\
& {\Weil^-}
\arrow[""{name=0, anchor=center, inner sep=0}, "{\Leung[\TT]}", from=1-1, to=1-3]
\arrow[from=1-1, to=2-2]
\arrow["{\Leung^-[\TT]}"', from=2-2, to=1-3]
\arrow["\epsilon"', shorten >=3pt, Rightarrow, from=2-2, to=0]
\end{tikzcd}
\end{equation*}
is fully specified by natural transformations:
\begin{align*}
&\epsilon_n\colon\Leung^-[\TT](W_n)\to\Leung[\TT](W_n^-)
\end{align*}
Let us define each $\epsilon_n$ as follows:
\begin{align*}
&\epsilon_n\colon\T_n^-M\xrightarrow{e_n}\T_{2n}M\xrightarrow{(\pi_1)_n}\T M
\end{align*}
As shown by Cockett and Cruttwell, $\epsilon_n$ makes $\TT^-$ into a subtangent structure of $\TT$, thus $\epsilon$ is a well-defined monoidal natural transformation. Let us show that $\epsilon$ satisfies the universal property of a right Kan extension. Consider another monoidal natural transformation:
\begin{equation*}
\begin{tikzcd}
\Weil && {\End(\X)} \\
& {\Weil^-}
\arrow[""{name=0, anchor=center, inner sep=0}, "{\Leung[\TT]}", from=1-1, to=1-3]
\arrow[from=1-1, to=2-2]
\arrow["{\Leung^-[\TT]}"', from=2-2, to=1-3]
\arrow["\theta"', shorten >=3pt, Rightarrow, from=2-2, to=0]
\end{tikzcd}
\end{equation*}
Evaluated on the Weil algebra $W$, $\theta$ is a natural transformation:
\begin{align*}
&\theta\colon\T^-M\to\T M
\end{align*}
Let us define $\xi\colon\T^-M\to\T^-M$ as the morphism induced by the universal property of the tangent pullbacks~\eqref{equation:negative-pullbacks}:
\begin{equation*}
\begin{tikzcd}
{\T^-M} && {\T^-_2M} \\
& {\T^-M} & {\T_2M} \\
& M & {\T M}
\arrow["{\<\id,n\>}", from=1-1, to=1-3]
\arrow["\xi"{description}, dashed, from=1-1, to=2-2]
\arrow["{p^-}"', bend right, from=1-1, to=3-2]
\arrow["{\theta_2}", from=1-3, to=2-3]
\arrow["e", from=2-2, to=2-3]
\arrow["{p^-}"', from=2-2, to=3-2]
\arrow["\lrcorner"{anchor=center, pos=0.125}, draw=none, from=2-2, to=3-3]
\arrow["s", from=2-3, to=3-3]
\arrow["z"', from=3-2, to=3-3]
\end{tikzcd}
\end{equation*}
Thus, we shall compute:
\begin{align*}
&\quad\epsilon\o\xi\\
=&\quad\pi_1\o e\o\xi\\
=&\quad\pi_1\o\theta_2\o\<\id,n\>\\
=&\quad\theta
\end{align*}
Therefore, by extending $\xi$ to each $W_n$ in the obvious way, we obtain a monoidal natural transformation $\xi\colon\Leung^-[\TT]\Rightarrow\Leung^-[\TT]$ such that $\epsilon\o\xi=\theta$. Finally, suppose that $\xi'\colon\Leung^-[\TT]\Rightarrow\Leung^-[\TT]$ satisfies the same equation as $\xi$. Thus:
\begin{align*}
&\quad e\o\xi\\
=&\quad\pi_1\o\theta_2\o\<\id,n\>\\
=&\quad\pi_1\o(\xi'e)_2\o\<\id,n\>\\
=&\quad e\o\xi'\o\pi_1\o\<\id,n\>\\
=&\quad e\o\xi'
\end{align*}
Moreover, $p^-\o\xi'=p^-$, since $\xi'$ needs to be compatible with the tangent structure $\TT^-$. Therefore, by the universal property of the tangent pullback~\eqref{equation:negative-pullbacks}, $\xi'=\xi$.
\end{proof}

\begin{definition}
\label{definition:right-deniable-tangent-object}
A \textbf{right deniable} tangentad is a tangentad $(\X,\TT)$ whose Leung functor $\Leung[\TT]$ admits a pointwise right Kan extension along the inclusion $\Weil\to\Weil^-$ as a Leung functor.
\end{definition}

Definition~\ref{definition:right-deniable-tangent-object} suggests another possibility of equipping a tangent structure with negatives.

\begin{definition}
\label{definition:left-deniable-tangent-object}
A \textbf{left deniable} tangentad is a tangentad $(\X,\TT)$ whose Leung functor $\Leung[\TT]$ admits a pointwise left Kan extension along the inclusion $\Weil\to\Weil^-$ as a Leung functor.
\end{definition}

Left deniability is yet to be explored. Future work will investigate this concept.


\section{Applications of the formal approach}
\label{section:known-examples}
In this section, we re-interpret some of the constructions of tangent category theory, such as tangent fibrations, tangent monads, or reverse tangent categories, from the viewpoint of tangentads. We exploit this reformulation to prove some results: we show that a tangent monad admits the construction of algebras provided the underlying monad does; we recall the Grothendieck equivalence between tangent fibrations and tangent indexed categories; we prove that tangent split restriction categories are equivalent to $\M$-tangent categories.

\subsection{Tangent monads and the construction of algebras}
\label{subsection:tangent-monads}
In~\cite{cockett:tangent-monads}, Cockett, Lemay, and Lucyshyn-Wright introduced the notion of a tangent monad. In this section, we prove that tangent monads are precisely tangentads in the $2$-category $\Mnd$ of monads. We employ this characterization to prove that the tangent category of algebras of a tangent monad is precisely the Eilenberg-Moore object of the tangent monad, as defined by Street in~\cite{street:formal-theory-monads}.

\par Since the $2$-category of monads $\Mnd(\CC)$ plays an important role in this section, let us start by recalling this construction. The objects of $\Mnd(\CC)$ are tuples $(\X;S,\eta,\mu)$ formed by an object $\X$ of $\CC$ together with a monad $(S,\eta,\mu)$ on $\X$, which is a monoid in the monoidal category $\End(\X)$.

\par A $1$-morphism of $\Mnd(\CC)$ from a monad $(\X,S)$ to a monad $(\X',S')$ consists of a $1$-morphism $F\colon\X\to\X'$ of $\CC$ together with a $2$-morphism $\alpha\colon S'\o F\Rightarrow S\o F$, satisfying the following compatibility with the unit $\eta\colon\id_\X\Rightarrow S$ and the multiplication $\mu\colon S^2\Rightarrow S$:
\begin{equation*}
\begin{tikzcd}
{S'\o F} && {F\o S} \\
F && F
\arrow["\alpha", from=1-1, to=1-3]
\arrow["{\eta'F}", from=2-1, to=1-1]
\arrow[Rightarrow, no head, from=2-1, to=2-3]
\arrow["{F\eta}"', from=2-3, to=1-3]
\end{tikzcd}\hfill\quad
\begin{tikzcd}
{S'^2\o F} & {S'\o F\o S} & {F\o S^2} \\
{S'\o F} && {F\o S}
\arrow["{S'\alpha}", from=1-1, to=1-2]
\arrow["{\mu\T}"', from=1-1, to=2-1]
\arrow["{\alpha S}", from=1-2, to=1-3]
\arrow["{\T\mu}", from=1-3, to=2-3]
\arrow["\alpha"', from=2-1, to=2-3]
\end{tikzcd}
\end{equation*}
Finally, a $2$-morphism $\varphi\colon(F,\alpha)\Rightarrow(G,\beta)\colon(\X,S)\to(\X',S')$ consists of a $2$-morphism $\varphi\colon F\Rightarrow G$ of $\CC$, compatible with the distributive laws:
\begin{equation*}
\begin{tikzcd}
{S'\o F} && {F\o S} \\
{S'\o G} && {G\o S}
\arrow["\alpha", from=1-1, to=1-3]
\arrow["{S'\varphi}"', from=1-1, to=2-1]
\arrow["{\varphi S}", from=1-3, to=2-3]
\arrow["\beta"', from=2-1, to=2-3]
\end{tikzcd}
\end{equation*}

\begin{definition}
\label{definition:tangent-monad}
A \textbf{tangent monad} in a $2$-category $\CC$ consists of a monad in the $2$-category $\Tng(\CC)$ of tangentads of $\CC$.
\end{definition}

Let us unwrap this definition. A monad of $\Tng(\CC)$ consists of a tangentad $(\X,\TT)$ of $\CC$ equipped with a monad $(S,\eta,\mu)$ on the category $\X$ with a distributive law $\alpha\colon S\o\T\Rightarrow\T\o S$ between the monad $S$ and the tangent bundle functor.

\par Moreover, $\alpha$ is required to be compatible with the tangent structure $\TT$ and with the unit $\eta\colon\id_\X\Rightarrow S$ and the multiplication $\mu\colon S^2\Rightarrow S$ of the monad:
\begin{equation*}
\begin{tikzcd}
{S\o\T} && {\T\o S} \\
\T && \T
\arrow["\alpha", from=1-1, to=1-3]
\arrow["{\eta\T}", from=2-1, to=1-1]
\arrow[Rightarrow, no head, from=2-1, to=2-3]
\arrow["{\T\eta}"', from=2-3, to=1-3]
\end{tikzcd}\hfill\quad
\begin{tikzcd}
{S^2\o\T} & {S\o\T\o S} & {\T\o S^2} \\
{S\o\T} && {\T\o S}
\arrow["{S\alpha}", from=1-1, to=1-2]
\arrow["{\mu\T}"', from=1-1, to=2-1]
\arrow["{\alpha S}", from=1-2, to=1-3]
\arrow["{\T\mu}", from=1-3, to=2-3]
\arrow["\alpha"', from=2-1, to=2-3]
\end{tikzcd}
\end{equation*}
When the ambient $2$-category $\CC$ is the $2$-category $\Cat$ of categories, tangent monads correspond precisely to the notion of tangent monads introduced by Cockett, Lemay, and Lucyshyn-Wright~\cite[Definition~19]{cockett:tangent-monads}.

\begin{proposition}
\label{proposition:tangent-monads-vs-tangent-objects}
The $2$-category $\Tng(\Mnd(\CC))$ of tangentads of the $2$-category of monads of a $2$-category $\CC$ is isomorphic to the $2$-category $\Mnd(\Tng(\CC))$ of monads of the $2$-category of tangentads of $\CC$.
\end{proposition}
\begin{proof}
First, consider a monad $S$ on an object $\X$ of $\CC$. A tangent structure on $S$ comprises a morphism $(\T,\alpha)\colon(\X,S)\to(\X,S)$ of monads, which is a $1$-morphism $\T\colon\X\to\X$ of $\CC$ together with a $2$-morphism $\alpha\colon S\o\T\Rightarrow S\o\T$ of $\CC$ compatible with the multiplication and the unit of $S$. The structural $2$-morphisms of the tangent structure are $2$-morphisms $p,z,s,l,$ and $c$, respectively, of $\CC$, compatible with the distributive law $\alpha$, and satisfying the axioms of a tangent structure.

\par The tuple $(\T,p,z,s,l,c)$ defines a tangent structure of $\X$, and $(S,\alpha)$ becomes a monad over $(\X,\TT)$ in the $2$-category $\Tng(\CC)$ of tangentads over $\CC$. So, every object $(\X,S;\TT,\alpha)$ of $\Tng(\Mnd(\CC))$ defines an object $(\X,\TT;S,\alpha)$ of $\Mnd(\Tng(\CC))$.

\par Let $(F,\varphi;\beta)\colon(\X,S;\TT,\alpha)\to(\X',S';\TT',\alpha')$ be a lax tangent morphism of $\Tng(\Mnd(\CC))$. This is a $1$-morphism $F\colon\X\to\X'$ of $\CC$ together with a distributive law $\varphi\colon S'\o F\Rightarrow F\o S$, compatible with the monad structures, and a lax distributive law $\beta\colon F\o\T\Rightarrow\T'\o F$, compatible with $\varphi,\alpha$, and $\alpha'$.

\par This corresponds to a morphism $(F,\beta;\varphi)\colon(\X,\TT;S,\alpha)\to(\X',\TT';S',\alpha')$ of monads of $\Tng(\CC)$. Finally, take into consideration a lax $2$-morphism $\theta\colon(F,\varphi;\beta)\Rightarrow(G,\psi;\gamma)$ between two lax tangent $1$-morphisms $(F,\varphi;\beta),(G,\psi;\gamma)\colon(\X,S;\TT,\alpha)\to(\X',S';\TT',\alpha')$.

\par This consists of a $2$-morphism $\theta\colon F\Rightarrow G$, compatible with $\varphi$ and $\psi$, and with $\alpha$ and $\alpha'$. However, this is also a $2$-morphism $\theta\colon(F,\beta;\varphi)\Rightarrow(G,\gamma;\psi)$ between the corresponding morphisms of monads of $\Tng(\CC)$.
\end{proof}

\begin{remark}
\label{remark:Tng-co-Mnd}
In light of Proposition~\ref{proposition:tangent-monads-vs-tangent-objects}, it is natural to wonder whether or not $\Tng_\co(\Mnd(\CC))$ and $\Mnd(\Tng_\co(\CC))$ are also isomorphic $2$-categories, for a $2$-category $\CC$. However, this is not the case. Notice that objects of $\Tng_\co(\Mnd(\CC))$ are tangentads over the $2$-category of monads over $\CC$. Thus, the objects of $\Tng_\co(\Mnd(\CC))$ are the same objects as those of $\Tng(\Mnd(\CC))$. On the contrary, the objects of $\Mnd(\Tng_\co(\CC))$ are monads over the $2$-category $\Tng_\co(\CC)$, that are tuples $(\X,\TT;S,\beta)$, where $\beta\colon\T\o S\Rightarrow S\o\T$, since $(S,\beta)\colon(\X,\TT)\nto(\X,\TT)$ is a colax tangent $1$-morphism. To fix this discrepancy, one can consider the $2$-category $\Mnd_\co(\CC)$ of monads, colax $1$-morphisms of monads $(F,\beta)\colon(\X,S)\nto(\X',S')$ which are $1$-morphisms $F\colon\X\to\X'$ together with a $2$-morphism:
\begin{equation*}
\begin{tikzcd}
\X & \X \\
{\X'} & {\X'}
\arrow["S", from=1-1, to=1-2]
\arrow["{S'}"', from=2-1, to=2-2]
\arrow["F"', from=1-1, to=2-1]
\arrow["F", from=1-2, to=2-2]
\arrow["\beta"{description}, Rightarrow, from=1-2, to=2-1]
\end{tikzcd}
\end{equation*}
compatible with the multiplication and the unit of the monads, and $2$-morphisms $\theta\colon(F,\beta)\Rightarrow(G,\gamma)$ which consists of $2$-morphisms $\theta\colon F\Rightarrow G$, compatible with $\beta$ and $\gamma$. Adopting a dual argument, one can show that $\Tng_\co(\Mnd_\co(\CC))$ and $\Mnd_\co(\Tng_\co(\CC))$ are isomorphic $2$-categories.
\end{remark}

Street in~\cite{street:formal-theory-monads} established the universal property of the Eilenberg-Moore category of a monad, a.k.a. the category of algebras of a monad, and formalized this construction for arbitrary monads.

\par When $\CC$ is the $2$-category $\Cat$ of categories, \cite[Proposition~20]{cockett:tangent-cats} establishes that the category of algebras $\Alg(S)$ of a tangent monad $(S,\alpha)$ over a tangent category $(\X,\TT)$ comes equipped with a tangent structure strictly preserved by the forgetful functor $\U\colon\Alg(S)\to\X$.

\par Concretely, the tangent bundle functor $\T^S\colon\Alg(S)\to\Alg(S)$ of the tangent structure $\TT^S$ on the category of algebras $\Alg(S)$ of a tangent monad $(S,\alpha)$ sends and algebra $A$ of $S$ with structure map $\theta\colon SA\to A$ to the algebra $\T A$ with structure map:
\begin{align*}
&S\T A\xrightarrow{\alpha}\T SA\xrightarrow{\T\theta}\T A
\end{align*}
Moreover, it sends a morphism $f\colon A\to B$ of algebras of $S$ to $\T f$. The structural natural transformations of $\TT^S$ are defined by the corresponding structural natural transformations of $\TT$. When $\TT$ has negatives, so does $\TT^S$, with negation $n$ as in $\TT$.

\par Thanks to the formal approach, we show that this construction corresponds to Street's formal notion of the Eilenberg-Moore object of a tangent monad. First, recall that a $2$-category $\CC$ \textit{admits the construction of algebras} when the $2$-functor
\begin{align*}
&\Inc\colon\CC\to\Mnd(\CC)
\end{align*}
which sends each object $\X$ of $\CC$ to the trivial monad $(\X,\1)$, $\1$ be the identity on $\X$, admits a right adjoint $\Alg\colon\Mnd(\CC)\to\CC$. When $\CC$ is the $2$-category $\Cat$ of categories, $\Alg$ is precisely the $2$-functor which sends a monad $(\X,S)$ to the category of algebras $\Alg(S)$ of $S$.

\begin{theorem}
\label{theorem:formal-tangent-monads-admit-algebras}
If the $2$-category $\CC$ admits the construction of algebras, so does the $2$-category $\Tng(\CC)$ of tangentads of $\CC$. In particular, the Eilenberg-Moore object $\Alg(S)$ of a tangent monad $(S,\alpha)$ on a tangentad $(\X,\TT)$ of $\CC$ comes equipped with a tangent structure strictly preserved by the forgetful $1$-morphism $\U\colon(\Alg(S),\TT^S)\to(\X,\TT)$.
\end{theorem}
\begin{proof}
First, notice that by definition the $2$-functor $\Alg\colon\Mnd(\CC)\to\CC$ is a right adjoint, thus, by Lemma~\ref{lemma:right-adjoints-preserve-pointwise-limits}, it preserves pointwise limits. Furthermore, the $2$-functor $\Inc\colon\CC\to\Mnd(\CC)$ is also a right adjoint, as proved in~\cite[Theorem~1]{street:formal-theory-monads}. In particular, it is the right adjoint of the forgetful $2$-functor $\Mnd(\CC)\to\CC$ which sends a pair $(\X,S)$ to $\X$.

\par Therefore, by Lemma~\ref{lemma:right-adjoints-preserve-pointwise-limits}, both $\Inc$ and $\Alg$ are pointwise pullback-preserving $2$-functors and they form an adjunction $\Inc\dashv\Alg$ in $\TwoCat_\pp$.

\par As a $2$-functor, $\Tng$ preserves adjunctions, thus $\Tng(\Inc)\colon\Tng(\CC)\leftrightarrows\TngMnd(\CC)\colon\Tng(\Alg)$ forms an adjunction. Thanks to Proposition~\ref{proposition:tangent-monads-vs-tangent-objects}, the $2$-category $\TngMnd(\CC)$ of tangent monads is isomorphic to the $2$-category $\Mnd\Tng(\CC)$ of monads in the $2$-category of tangentads of $\CC$. Therefore, $\Tng(\CC)$ admits the construction of algebras.

\par To prove that the right adjoint $\U\colon(\Alg(S),\TT^S)\to(\X,\TT)$ is a strict tangent morphism, first recall the construction of the adjunction $F\colon\X\leftrightarrows\Alg(S)\colon\U$ associated to a generic monad, as in~\cite{street:formal-theory-monads}. The $2$-functors $\Alg$ and $\Inc$ form an adjunction $\Inc\dashv\Alg$, whose counit, for every monad $S$ over an object $\X$, is a morphism of monads $(\Alg(S),\1_{\Alg(S)})\to(\X,S)$, where $\1_{\Alg(S)}$ denotes the trivial monad over $\Alg(S)$. In particular, the underlying $1$-morphism of the counit is the morphism $\U\colon\Alg(S)\to\X$, which represents the right adjoint in the adjunction $F\dashv\U$, associated with the monad $S$.

\par Notice that the counit of the induced adjunction $\Alg\colon\Tng(\CC)\leftrightarrows\Mnd(\Tng(\CC))\colon\Inc$ is precisely given by $\Tng(\epsilon)$. By definition, given a natural transformation $\theta\colon F\Rightarrow G$ from a (pointwise pullback-preserving) $2$-functor $F\colon\CC\to\CC'$ to another $2$-functor $G\colon\CC\to\CC'$, the corresponding natural transformation $\Tng(\theta)\colon\Tng(F)\Rightarrow\Tng(G)$ is defined, for every tangentad $(\X,\TT)$ of $\CC$, as the tangent morphism $(F\X,F\TT)\to(G\X,G\TT)$, whose underlying $1$-morphism is $\theta\colon F\X\to G\X$.

\par The distributive law is just the identity, since $\theta(F\T)=(G\T)\theta$. In particular, this applies to the counit of the adjunction $\Inc\dashv\Alg$ and therefore, the right adjoint $\U\colon(\Alg(S),\TT^S)\to(\X,\TT)$ is a strict tangent morphism.
\end{proof}

\begin{remark}
\label{remark:co-monad-do-not-have-algebras}
One of the key facts employed in the proof of Theorem~\ref{theorem:formal-tangent-monads-admit-algebras} is that $\TngMnd(\CC)\cong\Mnd(\Tng(\CC))$. One could wonder if, assuming that $\CC$ admits the construction of algebras, then also $\Tng_\co(\CC)$ would admit this construction. However, as pointed out in Remark~\ref{remark:Tng-co-Mnd}, $\Tng_\co(\Mnd(\CC))$ is not isomorphic to $\Mnd(\Tng_\co(\CC))$.
\end{remark}

As a consequence of Theorem~\ref{theorem:formal-tangent-monads-admit-algebras}, a tangent monad on a tangent category admits the construction of algebras. The next corollary shows that the Eilenberg-Moore object of a tangent monad is precisely the tangent category of algebras of the tangent monad, as described by Cockett, Lemay, and Lucyshyn-Wright.

\begin{corollary}
\label{corollary:algebras-tangent-monads}
The $2$-category $\TngCat$ of tangent categories admits the construction of algebras. Moreover, the Eilenberg-Moore object $(\Alg(S),\TT^S)$ of a tangent monad $(S,\alpha)$ on a given tangent category $(\X,\TT)$ is precisely the tangent category described by Cockett, Lemay, and Lucyshyn-Wright in~\cite{cockett:tangent-monads}.
\end{corollary}
\begin{proof}
Thanks to~\cite[Theorem~7]{street:formal-theory-monads}, the $2$-category $\Cat$ of categories admits the construction of algebras and the $2$-functor $\Alg$ sends a monad to the corresponding Eilenberg-Moore category. By Theorem~\ref{theorem:formal-tangent-monads-admit-algebras} also $\Tng(\Cat)$ admits the construction of algebras. However, as noticed in Example~\ref{example:tangent-categories}, $\Tng(\Cat)$ is the $2$-category $\TngCat$ of tangent categories.

\par To prove the second part, recall that the $2$-functor $\Alg\colon\Mnd\to\Cat$ sends a morphism of monads $(F,\alpha)\colon(\X,S)\to(\X,S)$ to the functor $\Alg(F,\alpha)$ which sends an algebra $A$ of $S$ with structure map $\theta\colon SA\to A$, to the algebra $FA$ of $S'$ with structure map:
\begin{align*}
&S'FA\xrightarrow{\alpha}FSA\xrightarrow{F\theta}FA
\end{align*}
Moreover, $\Alg$ sends a natural transformation $\varphi\colon(F,\alpha)\Rightarrow(G,\beta)$ between two morphisms of monads $(F,\alpha),(G,\beta)\colon(\X,S)\to(\X',S')$, to the natural transformation $\Alg(F,\alpha)\Rightarrow\Alg(G,\beta)$, defined by $\varphi$.

\par Recall also that the $2$-functor:
\begin{align*}
&\Alg\colon\Mnd(\TngCat)\to\TngCat
\end{align*}
defined in Theorem~\ref{theorem:formal-tangent-monads-admit-algebras}, under the identification $\Mnd(\TngCat)=\TngMnd$, sends a tangentad $(\X,\TT;S,\alpha)$ of $\Mnd$, whose tangent bundle morphism $(\T,\alpha)\colon(\X,S)\to(\X,S)$ is given by the tangent bundle functor $\T\colon\X\to\X$ together with the distributive law $\alpha\colon S\o\T\Rightarrow\T\o S$, to the tangentad $(\Alg(\X,S),\Alg(\TT))$. The tangent bundle morphism is given by $\Alg(\T,\alpha)$ which sends an algebra $A$ of $S$ with structure map $\theta\colon SA\to A$ to the algebra $\T A$ with structure map:
\begin{align*}
&S\T A\xrightarrow{\alpha}\T SA\xrightarrow{\T\theta}\T A
\end{align*}
Moreover, it sends a morphism $f\colon A\to B$ of algebras of $S$ to $\T f$. Finally, all the natural transformations $p,z,s,l$, and $c$ are precisely given by the corresponding natural transformations of $\TT$.
\end{proof}

\begin{remark}
\label{remark:Tng-co-Mnd-no-algebras}
Corollary~\ref{corollary:algebras-tangent-monads} explains why $\Tng_\co(\CC)$ does not admit the construction of algebras. Take $\CC$ to be the $2$-category $\Cat$ of categories and a monad $(S,\beta)$ over $\Tng_\co(\Cat)$, which consists of a monad $S$ over a category $\X$, together with a colax distributive law $\beta\colon\T\o S\Rightarrow S\o\T$. Let $A$ be an algebra with structure map $SA\to A$ of $S$. In order to lift the tangent structure to $\Alg(S)$, we employed the distributive law $\alpha\colon S\o\T\Rightarrow\T\o S$ and defined the tangent bundle of $A$ as the algebra $\T A$ with structure map $\T\theta\o\alpha\colon S\T A\to\T SA\to\T A$. However, the colax distributive law $\beta$ is pointing in the wrong direction.
\end{remark}

Street, in~\cite{street:formal-theory-monads}, noticed that a distributive law $\gamma\colon S'\o S\Rightarrow S\o S'$ between two monads $(S,\eta,\mu)$ and $(S',\eta',\mu')$ can be regarded as a monad in the $2$-category $\Mnd$ of monads. Since a tangent monad is a tangentad of $\Mnd$, we can introduce the notion of a distributive law between tangent monads as a tangent monad in the $2$-category $\Mnd$ of monads, namely, a monad in the $2$-category $\Tng(\Mnd)\cong\Mnd(\Tng)$.

\begin{definition}
\label{definition:distributive-law-tangent-monad}
A \textbf{distributive law} of tangent monads of $\CC$ consists of a tangent monad in the $2$-category $\Mnd(\CC)$ of monads of $\CC$.
\end{definition}

Let us unpack Definition~\ref{definition:distributive-law-tangent-monad}. Given two tangent monads $(S,\alpha)$ and $(S',\alpha')$ on a tangentad $(\X,\TT)$ of $\CC$, a distributive law $\gamma\colon(S',\alpha')\o(S,\alpha)\Rightarrow(S,\alpha)\o(S',\alpha')$ of tangent monads consists of a $2$-morphism $\gamma\colon S'\o S\Rightarrow S\o S'$, compatible with the monad structures:
\begin{equation*}
\begin{tikzcd}
S && {S\o S'} \\
\\
{S'\o S} && {S'}
\arrow["{S\eta'}", from=1-1, to=1-3]
\arrow["{\eta'S}"', from=1-1, to=3-1]
\arrow["\gamma"{description}, from=3-1, to=1-3]
\arrow["{\eta S'}"', from=3-3, to=1-3]
\arrow["{S'\eta}", from=3-3, to=3-1]
\end{tikzcd}\hfill\quad
\begin{tikzcd}
{S'^2\o S} & {S'\o S\o S'} & {S\o S'^2} \\
{S'\o S} && {S\o S'} \\
{S'\o S^2} & {S\o S'\o S} & {S^2\o S'}
\arrow["{S'\gamma}", from=1-1, to=1-2]
\arrow["{\mu'S}"', from=1-1, to=2-1]
\arrow["{\gamma S'}", from=1-2, to=1-3]
\arrow["{S\mu'}", from=1-3, to=2-3]
\arrow["\gamma"{description}, from=2-1, to=2-3]
\arrow["{S'\mu}", from=3-1, to=2-1]
\arrow["{\gamma S}"', from=3-1, to=3-2]
\arrow["{S\gamma}"', from=3-2, to=3-3]
\arrow["{\mu S'}"', from=3-3, to=2-3]
\end{tikzcd}
\end{equation*}
and with the distributive laws:
\begin{equation*}
\begin{tikzcd}
{S'\o S\o\T} & {S'\o\T\o S} & {\T\o S'\o S} \\
{S\o S'\o\T} & {S\o\T\o S'} & {\T\o S\o S'}
\arrow["{S'\alpha}", from=1-1, to=1-2]
\arrow["{\gamma\T}"', from=1-1, to=2-1]
\arrow["{\alpha' S}", from=1-2, to=1-3]
\arrow["{\T\gamma}", from=1-3, to=2-3]
\arrow["{S\alpha'}"', from=2-1, to=2-2]
\arrow["{\alpha S'}"', from=2-2, to=2-3]
\end{tikzcd}
\end{equation*}

A distributive law of monads allows one to compose two monads together. In particular, if $(S,\eta,\mu)$ and $(S',\eta',\mu')$ are two monads on an object $\X$ of a $2$-category $\CC$ and $\gamma\colon S'\o S\Rightarrow S\o S'$ is a distributive law between them, then $S\o S'$ comes equipped with the structure of a monad of $\CC$ whose unit and multiplication are defined as follows:
\begin{align*}
&\eta\o_\gamma\eta'\colon\id_\X\xrightarrow{\eta}S\xrightarrow{S\eta'}S\o S'\\
&\mu\o\gamma\mu'\colon S\o S'\o S\o S'\xrightarrow{S\gamma S'}S^2\o S'^2\xrightarrow{\mu S'^2}S\o S'^2\xrightarrow{S\mu'}S\o S'
\end{align*}
We can harness this formal argument to immediately prove that a distributive law of two tangent monads forms a new tangent monad.

\begin{proposition}
\label{proposition:composition-tangent-monads}
Given two tangent monads $(S,\alpha)$ and $(S',\alpha')$ on a tangentad $(\X,\TT)$ of $\CC$ and a distributive law $\gamma\colon(S',\alpha')\o(S,\alpha)\Rightarrow(S,\alpha)\o(S',\alpha')$ of tangent monads, the composition monad $S\o_\gamma S'$ induced by the distributive law $\gamma$ is again a tangent monad on $(\X,\TT)$.
\end{proposition}

\subsection{Tangent fibrations}
\label{subsection:tangent-fibrations}
In~\cite[Section~5]{cockett:differential-bundles}, Cockett and Cruttwell showed that each tangent category equipped with a class $\Dsply$ of morphisms stable under tangent pullbacks and under the tangent bundle functor defines a fibration whose fibre over an object $M$ contains the maps of $\Dsply$ whose target is $M$. Furthermore, such a fibration is compatible with the tangent bundle functors.

\par Cockett and Cruttwell distilled the property of this class of fibrations into the notion of a tangent fibration. In~\cite{lanfranchi:grothendieck-tangent-cats}, we proved that tangent fibrations are equivalent to tangentads in the $2$-category $\Fib$ of fibrations and exploited this classification to construct a genuine Grothendieck construction between tangent fibrations and tangent indexed categories.

\par This section is dedicated to recalling the notion of a tangent fibration and showing that tangent fibrations are precisely tangentads in the $2$-category $\Fib$ of fibrations.

\par For starters, let us briefly recall that a fibration $\Pi\colon\X'\to\X$ between two categories consists of a functor for which every morphism $f\colon M\to M'$ of the base category $\X$ admits a Cartesian lift $\varphi_f^{E'}\colon f^\*E'\to E'$ for every object $E'$ on the fibre over $M'$, namely, $\Pi(E')=M'$.

\par A Cartesian morphism $\varphi\colon E'\to E''$ is a morphism of the total category $\X'$ for which any morphism $\psi\colon E\to E''$ of $\X'$ and any morphism $g\colon\Pi(E)\to\Pi(E')$ making the following diagram commutes
\begin{equation*}
\begin{tikzcd}
& {\Pi(E')} \\
{\Pi(E)} && {\Pi(E'')}
\arrow["{\Pi(\varphi)}", from=1-2, to=2-3]
\arrow["g", from=2-1, to=1-2]
\arrow["{\Pi(\psi)}"', from=2-1, to=2-3]
\end{tikzcd}
\end{equation*}
induce a unique morphism $\xi\colon E\to E'$ of $\X'$ such that $\Pi(\xi)=g$ and the following diagram commutes:
\begin{equation*}
\begin{tikzcd}
& {E'} \\
E && {E''}
\arrow["\varphi", from=1-2, to=2-3]
\arrow["\xi", dashed, from=2-1, to=1-2]
\arrow["\psi"', from=2-1, to=2-3]
\end{tikzcd}
\end{equation*}
A fibration $\Pi\colon\X'\to\X$ equipped with a choice of Cartesian lift $\varphi_f^{E'}\colon f^\*E'\to E'$ for each morphism $f\colon M\to M'$ of $\X$ and each object $E'$ of the fibre over $M'$ is known as a \textbf{cloven fibration} and the choice $(f,E')\mapsto\varphi_f^{E'}$ is called the \textbf{cleavage} of the fibration.

\par In the following, every fibration is assumed to be cloven and we denote by $\varphi_f^{E'}$ the Cartesian lift of $f\colon M\to M'$ along $E'\in\Pi^{-1}(M')$ defined by the cleavage. Moreover, we omit the superscript $E'$ when $E'$ is clear from the context. A fibration $\Pi\colon\X'\to\X$ is denoted as a triple $(\X,\X',\Pi)$.\newline

\par Suppose $\Pi_\o\colon\X_\o'\to\X_\o$ and $\Pi_\b\colon\X_\b'\to\X_\b$ are two fibrations and $F'\colon\X'_\o\to\X'_\b$ and $F\colon\X_\o\to\X_\b$ are two functors such that $\Pi_\b\o F'=F\o\Pi_\o$. Given a morphism $f\colon M\to M'$ of $\X_\o$, one can lift $Ff\colon FM\to FM'$ to $\X_\b'$ and obtain a Cartesian lift $\varphi_{Ff}\colon(Ff)^\*(F'E')\to F'E'$ of $Ff$ over $F'E'\in\X_\b'$. \par However, the cleavage of $\Pi_\o$ also defines a Cartesian lift $\varphi_f\colon f^\*E'\to E'$ of $f$ over $E'$. Therefore, since $\Pi_\b F'\varphi_f=F\Pi_\o\varphi_f=Ff$, also $F'\varphi_f$ is a lift of $Ff$. By employing the universal property of the Cartesian lift $\varphi_{Ff}$, we obtain a unique morphism $\xi_f\colon F'(f^\*E')\dashrightarrow(Ff)^\*(F'E')$ making the following diagram commutes:
\begin{equation*}
\begin{tikzcd}
& {(Ff)^\*(F'E')} \\
{F'(f^\*E')} && {F'E'}
\arrow["{\varphi_{Ff}}", from=1-2, to=2-3]
\arrow["{\xi_f}", dashed, from=2-1, to=1-2]
\arrow["{F'\varphi_f}"', from=2-1, to=2-3]
\end{tikzcd}
\end{equation*}
A morphism of fibrations $(F,F')\colon(\X_\o,\X_\o',\Pi_\o)\to(\X_\b,\X_\b',\Pi_\b)$ consists of a pair of functors $F\colon\X_\o\to\X_\b$ and $F'\colon\X_\o'\to\X_\b'$ strictly commuting with the two fibrations, namely, $\Pi_\b\o F'=F\o\Pi_\o$, and preserving the Cartesian lifts, meaning the unique morphism $\xi_f\colon F'(f^\*E')\dashrightarrow(Ff)^\*(F'E')$ induced by the universality of $\varphi_{Ff}$ is an isomorphism for every morphism $f\colon M\to M'$ of $\X_\o$ and every object $E'\in\Pi_\o^{-1}(M')$.

\begin{definition}
\label{definition:tangent-fibration}
Given two tangent categories $(\X',\TT')$ and $(\X,\TT)$, a \textbf{tangent fibration} $\Pi\colon(\X',\TT')\to(\X,\TT)$ consists of a fibration $\Pi\colon\X'\to\X$ whose underlying functor is a strict tangent morphism and such that the tangent bundle functors preserve the Cartesian lifts, namely, $(\T,\T')\colon(\X,\X',\Pi)\to(\X,\X',\Pi)$ is a morphism of fibrations.
\end{definition}

To compare tangent fibrations with tangentads, let us recall the definition of the $2$-category $\Fib$ of fibrations.
\begin{description}
\item[Objects] An object of $\Fib$ consists of a (cloven) fibration $\Pi\colon\X'\to\X$;

\item[$1$-morphisms] A $1$-morphism is a morphism of fibrations $(F,F')\colon(\X_\o,\X_\o',\Pi_\o)\to(\X_\b,\X_\b',\Pi_\b)$;

\item[$2$-morphisms] A $2$-morphism $(F,F')\Rightarrow(G,G')\colon(\X_\o,\X_\o',\Pi_\o)\to(\X_\b,\X_\b',\Pi_\b)$ consists of pair of natural transformations $\varphi\colon F\Rightarrow G$ and $\varphi'\colon F'\Rightarrow G'$ such that:
\begin{align*}
&\Pi_\b\varphi=\varphi'_{\Pi_\o}
\end{align*}
\end{description}

\begin{proposition}
\label{proposition:tangent-fibrations-vs-tangent-objects}
A tangent fibration $\Pi\colon(\X',\TT')\to(\X,\TT)$ between two tangent categories is equivalent to a tangentad in the $2$-category $\Fib$ of fibrations.
\end{proposition}
\begin{proof}
A tangentad in the $2$-category $\Fib$ consists of a fibration $\Pi\colon\X'\to\X$ together with an endofunctor $(\T,\T')\colon\Pi\to\Pi$ of fibrations and a family of $2$-morphisms of fibrations $(p,p')\colon(\T,\T')\Rightarrow\id_{\Pi}$, $(z,z')\colon\id_\Pi\Rightarrow(\T,\T')$, $(s,s')\colon(\T_2,\T_2')\Rightarrow(\T,\T')$, $(l,l')\colon(\T,\T')\Rightarrow(\T^2,{\T'}^2)$, and $(c,c')\colon(\T^2,{\T'}^2)\Rightarrow(\T^2,{\T'}^2)$, satisfying suitable axioms. Notice that the $n$-fold pullback of the projection $(p,p')$ along itself consists of the pair of functors $(\T_n,\T_n')$. To prove this statement, we need to show that $(\T_n,\T_n')$ preserves Cartesian lifts. Let us consider the diagram in the base category $\X$:
\begin{equation*}
\begin{tikzcd}
{\T_nA'} & {\T_nA} \\
{\T A'} & {\T A}
\arrow["{\T_nf}", from=1-1, to=1-2]
\arrow["{\pi_k}"', from=1-1, to=2-1]
\arrow["{\pi_k}", from=1-2, to=2-2]
\arrow["{\T f}"', from=2-1, to=2-2]
\end{tikzcd}
\end{equation*}
where $f\colon A'\to A$ is a morphism of $\X$ and let $E$ be an object in the fibre over $A$. We need to show that the morphism $\xi_n\colon\T_n'(f^\*E)\to(\T_nf)^\*\T'_nE$ induced by the universal property of the Cartesian lift $\varphi_{\T_nf}$ is invertible. First, notice that the $k$-th projection $\pi_k\colon\T_n\Rightarrow\T_n$ induces a unique morphism $\xi_{\pi_k}\colon(\T_nf)^\*\T_n'E\to(\T f)^\*\T'E$, making the following diagram commutes:
\begin{equation*}
\begin{tikzcd}
{(\T_nf)^\*\T_n'E} & {\T'_nE} \\
{(\T f)^\*\T'E} & {\T'E}
\arrow["{\varphi_{\T_nf}}", from=1-1, to=1-2]
\arrow["{\xi_{\pi_k}}"', dashed, from=1-1, to=2-1]
\arrow["{\pi_k'}", from=1-2, to=2-2]
\arrow["{\varphi_{\T f}}"', from=2-1, to=2-2]
\end{tikzcd}
\end{equation*}
Since $(\T,\T')$ preserves Cartesian lifts, the unique morphism $\xi\colon\T'(f^\*E)\to(\T f)^\*(\T'E)$ is invertible, therefore, we can define a family of morphisms:
\begin{align*}
&(\T_nf)^\*(\T_n'E)\xrightarrow{\xi_{\pi_k}}(\T f)^\*(\T'E)\xrightarrow{\xi}\T'(f^\*E)
\end{align*}
By employing the universal property of the pullback $\T_n'$, we obtain a morphism $(\T_nf)^\*(\T_n'E)\to(\T_nf)^\*\T'E$ which inverts $\xi_n$.
\par Since the $n$-fold pullback of the projection along itself is the morphism of fibrations $(\T_n,\T_n')$, it is not hard to prove that $(\X,\TT)$ and $(\X',\TT')$ are tangent categories. Furthermore, since $\Pi\o\T'=\T\o\Pi$, $\Pi p'=p_\Pi$, $\Pi z'=z_\Pi$, $\Pi s'=s_\Pi$, $\Pi l'=l_\Pi$, and $\Pi c'=c_\Pi$, $\Pi$ is a strict tangent morphism. Therefore, a tangentad of $\Fib$ is a tangent fibration.
\par Conversely, the data of a tangent fibration can easily be rearranged to produce a tangentad of $\Fib$.
\end{proof}

\subsection{Tangent indexed categories and the Grothendieck construction}
\label{subsection:tangent-indexed-categories}
In~\cite{lanfranchi:grothendieck-tangent-cats}, A Grothendieck construction in the context of tangent categories was introduced by lifting the $2$-equivalence between (cloven) fibrations and indexed categories to tangentads. In particular, tangent indexed categories were introduced as tangentads in the $2$-category $\Indx$ of indexed categories as the ``correct'' notion of indexed categories in the context of tangent category theory. In this section, we briefly recall this definition and the Grothendieck equivalence.
\par For starters, recall that $\Indx$ is the $2$-category so defined:
\begin{description}
\item[Objects] An object of $\Indx$ consists of a pair $(\X,\IND)$ formed by a category $\X$ and an indexed category, a.k.a. a pseudofunctor, $\IND\colon\X^\op\to\Cat$;

\item[$1$-morphisms] A $1$-morphism $(F,\{F^{A}\}_A,\{\xi^f\}_f)$ of $\Indx$ from $(\X_\o,\IND_\o)$ to $(\X_\b,\IND_\b)$ comprises a functor $F\colon\X_\o\to\X_\b$ together with a collection of functors $F^A\colon\X_\o^A\to\X_\b^{FA}$, indexed by the objects $A$ of $\X_\o$, where $\X_\o^A\=\IND_\o(A)$ and $\X_\b^{FA}\=\IND_\b(FA)$, and a collection of $2$-morphisms $\xi^f\colon F^A\o f^\*_\o\Rightarrow(Ff)^\*_\b\o F^B$, called \textbf{distributors}, indexed by the morphisms $f\colon A\to B$ of $\X_\o$, where $f^\*_\o\=\IND_\o(f)$ and $(Ff)^\*_\b\=\IND_\b(Ff)$. Moreover, the $\xi^f$ are compatible with the unitors and compositors of the indexed categories as follows:
\begin{equation*}
\begin{tikzcd}
{F^AE} & {(F\id_A)^\*_\b F^AE} \\
{F^A\id_{A\o}^\*E} & {(\id_{FA})^\*_\b F^AE}
\arrow["{\Unit_\b F^A}", from=1-1, to=1-2]
\arrow["{F^A\Unit_\o}"', from=1-1, to=2-1]
\arrow[Rightarrow, no head, from=1-2, to=2-2]
\arrow["{\kappa^{\id_A}}"', from=2-1, to=2-2]
\end{tikzcd}\hfill\quad
\begin{tikzcd}
{F^Af^\*_\o g^\*_\o K} & {(Ff)^\*_\b F^{B}g^\*_\o K} & {(Ff)^\*_\b (Fg)^\*_\b F^{C}K} \\
{F^A(g\o f)^\*K} & {(F(g\o f))^\*_\b F^{C}K} & {(FfFg)^\*_\b F^{C}K}
\arrow["{\kappa^f g^\*_\o}", from=1-1, to=1-2]
\arrow["{F^A\Comp_\o}"', from=1-1, to=2-1]
\arrow["{(Ff)^\*_\b\kappa^g}", from=1-2, to=1-3]
\arrow["{\Comp_\b F^{C}}", from=1-3, to=2-3]
\arrow["{\kappa^{g\o f}}"', from=2-1, to=2-2]
\arrow[Rightarrow, no head, from=2-2, to=2-3]
\end{tikzcd}
\end{equation*}
where $\Unit_\o$, $\Comp_\o$, $\Unit_\b$, and $\Comp_\b$ denote the unitors and the compositors of $\IND_\o$ and $\IND_\b$, respectively, and where $f\colon B\to C$ and $g\colon A\to B$ and $E\in\X_\o^A$, and $K\in\X_\o^C$;

\item[$2$-morphisms] A $2$-morphism $(\varphi,\{\varphi^A\}_A)$ of $\Indx$ from $(F,\{F^{A}\}_A,\{\xi^f\}_f)$ to $(G,\{G^A\}_A,\{\kappa^f\}_f)$ comprises a natural transformation $\varphi\colon F\Rightarrow G$ together with a collection of natural transformations $\varphi^A\colon F^A\Rightarrow\varphi_\b^\*\o G^A$ which are compatible with the distributors as follows:
\begin{equation*}
\begin{tikzcd}
{F^A\o f_\o^\*} & {\varphi^\*\o G^A\o f_\o^\*} \\
& {\varphi^\*\o(Gf)_\b^\*\o G^{B}} \\
{(Ff)_\b^\*\o F^{B}} & {(Ff)_\b^\*\o\varphi^\*\o G^{B}}
\arrow["{\varphi^Af_\o^\*}", from=1-1, to=1-2]
\arrow["{\xi^f}"', from=1-1, to=3-1]
\arrow["{\varphi^\*\kappa^f}", from=1-2, to=2-2]
\arrow["\cong", from=2-2, to=3-2]
\arrow["{(Ff)_\b^\*\varphi^{B}}"', from=3-1, to=3-2]
\end{tikzcd}
\end{equation*}
where $f\colon A\to B$.
\end{description}

\begin{definition}
\label{definition:tangent-indexed-category}
A \textbf{tangent indexed category} is a tangentad in the $2$-category $\Indx$ of indexed categories. Moreover, a tangent indexed category \textbf{with negatives} is a tangentad with negatives in the $2$-category $\Indx$.
\end{definition}

Let us unpack Definition~\ref{definition:tangent-indexed-category}. A tangent indexed category consists of the following data:
\begin{description}
\item[Base tangent category] A base tangent category $(\X,\TT)$;

\item[Indexed category] An indexed category $\IND\colon\X^\op\to\Cat$;

\item[Indexed tangent bundle functor] A collection of functors $\T\^A\colon\X\^A\to\X\^{\T A}$, indexed by the objects $A$ of $\X$;

\item[Tangent distributors] A collection of natural transformations $\kappa\^f$, indexed by the morphisms $f\colon A\to B$ of $\X$:
\begin{equation*}
\begin{tikzcd}
{\X\^{B}} & {\X\^A} \\
{\X\^{\T B}} & {\X\^{\T A}}
\arrow["{\T\^{B}}"', from=1-1, to=2-1]
\arrow["{\T\^A}", from=1-2, to=2-2]
\arrow["{f^\*}", from=1-1, to=1-2]
\arrow["{(\T f)^\*}"', from=2-1, to=2-2]
\arrow["{\kappa\^f}"{description}, Rightarrow, from=1-2, to=2-1]
\end{tikzcd}
\end{equation*}

\item[Indexed projection] A collection of natural transformations $p\^A\colon\T\^A\Rightarrow p^\*$, indexed by the objects $A$ of $\X$
\begin{equation*}
\begin{tikzcd}
{\X\^A} & {\X\^{\T A}} \\
{\X\^A} & {\X\^{\T A}}
\arrow["{\T\^A}", from=1-1, to=1-2]
\arrow["{p^\*}"', from=2-1, to=2-2]
\arrow[Rightarrow, no head, from=1-1, to=2-1]
\arrow[Rightarrow, no head, from=1-2, to=2-2]
\arrow["{p\^A}"{description}, Rightarrow, from=1-2, to=2-1]
\end{tikzcd}
\end{equation*}
satisfying the following property:
\begin{equation*}
\begin{tikzcd}
{\T\^A\o f^\*} & {p^\*\o f^\*} \\
{(\T f)^\*\o\T\^{B}} & {(\T f)^\*\o p^\*}
\arrow["{p\^A f^\*}", from=1-1, to=1-2]
\arrow["{\kappa\^f}"', from=1-1, to=2-1]
\arrow["\Comp", from=1-2, to=2-2]
\arrow["{(\T f)^\*p\^{B}}"', from=2-1, to=2-2]
\end{tikzcd}
\end{equation*}

\item[Indexed zero] A collection of natural transformations $z\^A\colon\id_{\X\^A}\Rightarrow z^\*\o\T\^A$, indexed by the objects $A$ of $\X$
\begin{equation*}
\begin{tikzcd}
{\X\^A} & {\X\^A} \\
{\X\^{\T A}} & {\X\^A}
\arrow[Rightarrow, no head, from=1-1, to=1-2]
\arrow["{\T\^A}"', from=1-1, to=2-1]
\arrow["{z^\*}"', from=2-1, to=2-2]
\arrow[Rightarrow, no head, from=1-2, to=2-2]
\arrow["{z\^A}"{description}, Rightarrow, from=1-2, to=2-1]
\end{tikzcd}
\end{equation*}
satisfying the following property:
\begin{equation*}
\begin{tikzcd}
{f^\*} & {z^\*\o\T\^A\o f^\*} \\
{f^\*\o z^\*\o\T\^{B}} & {z^\*\o(\T f)^\*\o\T\^{B}}
\arrow["{z\^A f^\*}", from=1-1, to=1-2]
\arrow["{f^\*z\^{B}}"', from=1-1, to=2-1]
\arrow["{z^\*\kappa\^f}", from=1-2, to=2-2]
\arrow["\Comp", from=2-2, to=2-1]
\end{tikzcd}
\end{equation*}

\item[Indexed $n$-fold pullback] For every positive integer $n$, a collection of functors $\T\^A_n\colon\X\^A\to\X\^{\T_nA}$, indexed by the objects $A$ of $\X$, together with a collection of natural transformations $\kappa_n\^f\colon\T\^A_n\o f^\*\Rightarrow(\T_nf)^\*\o\T_n\^{B}$, indexed by the morphism $f\colon A\to B$ of $\X$
\begin{equation*}
\begin{tikzcd}
{\X\^{B}} & {\X\^A} \\
{\X\^{\T_nB}} & {\X\^{\T_nA}}
\arrow["{f^\*}", from=1-1, to=1-2]
\arrow["{\T\^A_n}", from=1-2, to=2-2]
\arrow["{\T\^{B}_n}"', from=1-1, to=2-1]
\arrow["{(\T_nf)^\*}"', from=2-1, to=2-2]
\arrow["{\kappa_n\^f}"{description}, Rightarrow, from=1-2, to=2-1]
\end{tikzcd}
\end{equation*}
and a collection of natural tranformations $\pi\^A_k\colon\T\^A_n\Rightarrow\pi_k\^*\o\T\^A$, indexed by the objects $A$ of $\X$:
\begin{equation*}
\begin{tikzcd}
{\X\^A} & {\X\^{\T_nA}} \\
{\X\^{\T A}} & {\X\^{\T_nA}}
\arrow["{\T\^A}"', from=1-1, to=2-1]
\arrow["{\pi_k^\*}"', from=2-1, to=2-2]
\arrow[Rightarrow, no head, from=1-2, to=2-2]
\arrow["{\T\^A_n}", from=1-1, to=1-2]
\arrow["{\pi\^A_k}"{description}, Rightarrow, from=1-2, to=2-1]
\end{tikzcd}
\end{equation*}
satisfying the following property:
\begin{equation*}
\begin{tikzcd}
{\T_n\^A\o f^\*} & {\pi_k^\*\o\T\^A\o f^\*} \\
& {\pi_k^\*\o(\T f)^\*\o\T\^{B}} \\
{(\T_nf)^\*\o\T_n\^{B}} & {(\T_nf)^\*\o\pi_k^\*\o\T\^{B}}
\arrow["{\pi_k\^Af^\*}", from=1-1, to=1-2]
\arrow["{\kappa_n\^ff^\*}"', from=1-1, to=3-1]
\arrow["{\pi_k^\*\kappa\^f}", from=1-2, to=2-2]
\arrow["\Comp", from=2-2, to=3-2]
\arrow["{(\T_nf)^\*\pi_k\^{B}}"', from=3-1, to=3-2]
\end{tikzcd}
\end{equation*}

\item[Indexed sum morphism] A collection of natural transformations $s\^A\colon\T_2\^A\Rightarrow s^\*\o\T\^A$, indexed by the objects $A$ of $\X$
\begin{equation*}
\begin{tikzcd}
{\X\^A} & {\X\^{\T_2A}} \\
{\X\^{\T A}} & {\X\^{\T_2A}}
\arrow["{\T\^A_2}", from=1-1, to=1-2]
\arrow[Rightarrow, no head, from=1-2, to=2-2]
\arrow["{\T\^A}"', from=1-1, to=2-1]
\arrow["{s^\*}"', from=2-1, to=2-2]
\arrow["{s\^A}"{description}, Rightarrow, from=1-2, to=2-1]
\end{tikzcd}
\end{equation*}
satisfying the following property:
\begin{equation*}
\begin{tikzcd}
{\T_2\^A\o f^\*} & {s^\*\o\T\^A\o f^\*} \\
& {s^\*\o(\T f)^\*\o\T\^{B}} \\
{(\T_2f)^\*\o\T_2\^{B}} & {(\T_2f)^\*\o s^\*\o\T\^{B}}
\arrow["{s\^Af^\*}", from=1-1, to=1-2]
\arrow["{\kappa_2\^ff^\*}"', from=1-1, to=3-1]
\arrow["{s^\*\kappa\^f}", from=1-2, to=2-2]
\arrow["\Comp", from=2-2, to=3-2]
\arrow["{(\T_2f)^\*s\^{B}}"', from=3-1, to=3-2]
\end{tikzcd}
\end{equation*}

\item[Indexed vertical lift] A collection of natural transformations $l\^A\colon\T\^A\Rightarrow l^\*\o\T\^{\T A}\o\T\^A$, indexed by the objects $A$ of $\X$
\begin{equation*}
\begin{tikzcd}
{\X\^A} & {\X\^{\T A}} \\
{\X\^{\T A}} \\
{\X\^{\T^2A}} & {\X\^{\T A}}
\arrow["{\T\^A}", from=1-1, to=1-2]
\arrow[Rightarrow, no head, from=1-2, to=3-2]
\arrow["{\T\^A}"', from=1-1, to=2-1]
\arrow["{\T\^{\T A}}"', from=2-1, to=3-1]
\arrow["{l^\*}"', from=3-1, to=3-2]
\arrow["{l\^A}"{description}, Rightarrow, from=1-2, to=3-1]
\end{tikzcd}
\end{equation*}
satisfying the following property:
\begin{equation*}
\begin{tikzcd}
{\T\^A\o f^\*} & {l^\*\o\T\^{\T A}\o\T\^A\o f^\*} \\
& {l^\*\o\T\^{\T A}\o(\T f)^\*\o\T\^{B}} \\
& {l^\*\o(\T^2f)^\*\o\T\^{\T B}\o\T\^{B}} \\
{(\T f)^\*\o\T\^{B}} & {(\T f)^\*\o l^\*\o\T\^{\T B}\o\T\^{B}}
\arrow["{l\^Af^\*}", from=1-1, to=1-2]
\arrow["{\kappa\^f}"', from=1-1, to=4-1]
\arrow["{l^\*\T\^{\T A}\kappa\^f}", from=1-2, to=2-2]
\arrow["{l^\*\kappa\^{\T f}\T\^{B}}", from=2-2, to=3-2]
\arrow["\Comp", from=3-2, to=4-2]
\arrow["{(\T f)^\*l\^{B}}"', from=4-1, to=4-2]
\end{tikzcd}
\end{equation*}

\item[Indexed canonical flip] A collection of natural transformations $c\^A\colon\T\^{\T A}\o\T\^A\Rightarrow c^\*\o\T\^{\T A}\o\T\^A$, indexed by the objects $A$ of $\X$
\begin{equation*}
\begin{tikzcd}
{\X\^A} & {\X\^{\T A}} \\
{\X\^{\T A}} \\
{\X\^{\T^2A}} & {\X\^{\T^2A}}
\arrow["{\T\^A}", from=1-1, to=1-2]
\arrow["{\T\^A}", from=1-2, to=3-2]
\arrow["{c^\*}"', from=3-1, to=3-2]
\arrow["{c\^A}"{description}, Rightarrow, from=1-2, to=3-1]
\arrow["{\T\^A}"', from=1-1, to=2-1]
\arrow["{\T\^{\T A}}"', from=2-1, to=3-1]
\end{tikzcd}
\end{equation*}
satisfying the following property:
\begin{equation*}
\begin{tikzcd}
{\T\^{\T A}\o\T\^A\o f^\*} & {c^\*\o\T\^{\T A}\o\T\^A\o f^\*} \\
& {c^\*\o\T\^{\T A}\o(\T f)^\*\o\T\^{B}} \\
{\T\^{\T A}\o(\T f)^\*\o\T\^{B}} & {c^\*\o(\T^2f)^\*\o\T\^{\T B}\o\T\^{B}} \\
{(\T^2f)^\*\o\T\^{\T B}\o\T\^{B}} & {(\T^2f)^\*\o c^\*\o\T\^{\T B}\o\T\^{B}}
\arrow["{c\^Af^\*}", from=1-1, to=1-2]
\arrow["{\T\^{\T A}\kappa\^f}"', from=1-1, to=3-1]
\arrow["{c^\*\T\^{\T A}\kappa\^f}", from=1-2, to=2-2]
\arrow["{c^\*\kappa\^{\T f}\T\^{B}}", from=2-2, to=3-2]
\arrow["{\kappa\^{\T f}\T\^{B}}"', from=3-1, to=4-1]
\arrow["\Comp", from=3-2, to=4-2]
\arrow["{(\T^2f)^\*c\^{B}}"', from=4-1, to=4-2]
\end{tikzcd}
\end{equation*}
\end{description}
If the indexed tangent category has negatives, then we also have:
\begin{description}
\item[Indexed negation] A collection of natural transformations $n\^A\colon\T\^A\Rightarrow n^\*\o\T\^A$, indexed by the objects $A$ of $\X$
\begin{equation*}
\begin{tikzcd}
{\X\^A} & {\X\^{\T A}} \\
{\X\^{\T A}} & {\X\^{\T A}}
\arrow["{\T\^A}", from=1-1, to=1-2]
\arrow["{\T\^A}"', from=1-1, to=2-1]
\arrow["{n^\*}"', from=2-1, to=2-2]
\arrow[Rightarrow, no head, from=1-2, to=2-2]
\arrow["{n\^A}"{description}, Rightarrow, from=1-2, to=2-1]
\end{tikzcd}
\end{equation*}
satisfying the following property:
\begin{equation*}
\begin{tikzcd}
{\T\^A\o f^\*} & {n^\*\o\T\^A\o f^\*} \\
& {n^\*\o(\T f)^\*\o\T\^{B}} \\
{(\T f)^\*\o\T\^{B}} & {(\T f)^\*\o n^\*\o\T\^{B}}
\arrow["{n\^Af^\*}", from=1-1, to=1-2]
\arrow["{\kappa\^f}"', from=1-1, to=3-1]
\arrow["{n^\*\kappa\^f}", from=1-2, to=2-2]
\arrow["\Comp", from=2-2, to=3-2]
\arrow["{(\T f)^\*n\^{B}}"', from=3-1, to=3-2]
\end{tikzcd}
\end{equation*}
\end{description}
These morphisms also need to satisfy the axioms of a tangentad. For the sake of completeness, let us state the main result of~\cite{lanfranchi:differential-bundles-operadic-affine-schemes} which motivated Definition~\ref{definition:tangent-indexed-category}.

\begin{theorem}[Grothendieck construction in the context of tangent categories]
\label{theorem:full-Grothendieck}
The Grothendieck correspondence between fibrations and indexed categories lifts to a $2$-equivalence between the $2$-category $\TngFib$ of tangent fibrations and the $2$-category $\TngIndx\=\Tng(\Indx)$ of tangent indexed categories:
\begin{align*}
&\TngFib\simeq\TngIndx
\end{align*}
Similarly, the $2$-categories $\TngFib_\co$ and $\TngIndx_\co\=\Tng_\co(\Indx)$ are equivalent:
\begin{align*}
&\TngFib_\co\simeq\TngIndx_\co
\end{align*}
Moreover, these two equivalences can be restricted to strong and strict morphisms, yielding two other equivalences:
\begin{align*}
&\TngFib_\cong\simeq\TngIndx_\cong\=\Tng_\cong(\Indx)\\
&\TngFib_=\simeq\TngIndx_=\=\Tng_=(\Indx)
\end{align*}
\end{theorem}

\subsection{Reverse tangent categories}
\label{subsection:reverse-tangent-categories}
Reverse tangent categories were introduced by Cruttwell and Lemay in~\cite{cruttwell:reverse-tangent-cats} to axiomatize the operation of reverse differentiation in the context of tangent categories. Reverse differentiation, categorically introduced in~\cite{cockett:reverse-derivative-cats}, is a formalization of the back-propagation operation employed in machine-learning techniques. While Cartesian reverse differential categories describe reverse differentiation in the context of Euclidean spaces, reverse tangent categories extend this operation for arbitrary locally linear geometric spaces of parameters. Reverse tangent categories also axiomatize the cotangent bundle of a manifold in the context of tangent category theory.

\par In this section, we recall the definition of a reverse tangent category and show that reverse tangent categories form a special class of tangentads in the $2$-category of dagger fibrations. In Example~\ref{example:dual-fibration}, we have mentioned that the $2$-category $\Fib$ of fibrations comes with opposite objects, where the involution $(-)^\op$ sends a fibration to its dual.

\par Concretely, the fibres of the dual fibration of a cloven fibration $\Pi$ are the opposite of the corresponding fibres of $\Pi$. Kock in~\cite{kock:dual-fibration} discusses the notion of dual fibrations for arbitrary non-necessarily cloven fibrations. For the sake of simplicity, throughout the section, we assume that every fibration is cloven. However, we expect every result to extend easily in the general context.

\par In~\cite{cockett:differential-bundles}, Cockett and Cruttwell showed that (display) differential bundles of a tangent category can be organized into a tangent fibration. For a definition of differential bundles, we refer to the cited paper.

\par In~\cite{lanfranchi:grothendieck-tangent-cats}, the author showed that his construction yields an adjunction between tangent fibrations and indexed tangent categories, which, however, is not an equivalence, in contrast with the genuine Grothendieck equivalence between tangent fibrations and tangent indexed categories.

\par In general, the opposite of a tangent category is not again a tangent category, unless the tangent bundle functor and all $n$-fold pullbacks along the projection admit left adjoints (cf.~\cite[Proposition~5.17]{cockett:tangent-cats}). This makes the dual of a tangent fibration, not a tangent fibration anymore.

\par Despite that the dual of a tangent fibration cannot be equipped with a tangent structure, \cite[Lemma~32]{cruttwell:monoidal-reverse-differential-categories} establishes that each morphism of fibrations $(F,F')\colon\Pi_\o\to\Pi_\b$, induces a morphism of fibrations $(F,F')^\b=(F,F'^\b)\colon\Pi_\o^\b\to\Pi_\b^\b$ between the dual fibrations. In the context of cloven fibrations, $(F,F')^\b$ is the morphism of fibrations which corresponds to the morphism of indexed categories:
\begin{align*}
&(\X_\o)_M^\b=(\X_\o)_M^\op\xrightarrow{F_M^\op}(\X_\b)_{FM}^\op=(\X_\b)_{FM}^\b
\end{align*}
via the Grothendieck construction, where $(\X_\o)_M=\Pi_\o^{-1}(M)$ denotes the fibre over an object $M\in\X_\o$ of $\Pi_\o$, $(\X_\b)_{FM}=\Pi_\b^{-1}(FM)$ is the fibre over $FM$ of $\Pi_\b$, and $F_M\colon(\X_\o)_M\to(\X_\b)_{FM}$ indicates the functor induced by $(F,F')$ between the fibres.

\par Thanks to this correspondence of morphisms of fibrations, the tangent bundle functor $\bar\T\colon\DB(\X,\TT)\to\DB(\X,\TT)$ induces a functor $\bar\T^\b\colon\DB^\b(\X,\TT)\to\DB^\b(\X,\TT)$.

\par It is important to notice that the equivalence of categories $(-)^\op\colon\Cat\to\Cat$ does not preserve natural transformations. In particular, $(-)^\op$ extends to a $2$-equivalence between $\Cat$ and $\Cat^\co$, where the latter denotes the $2$-category of categories, functors, and the opposite of natural transformations. This is the main reason why the dual of a tangent fibration is not, in general, a tangent fibration, since the dualization does not preserve the structural natural transformations.

\par Let us recall the definition of a reverse tangent category, introduced by Cruttwell and Lemay in~\cite{cruttwell:reverse-tangent-cats}.

\begin{definition}
\label{definition:reverse-tangent-category}
A \textbf{reverse tangent category} consists of a tangent category $(\X,\TT)$ and an involution $(-)^\*\colon\DB(\X,\TT)\to\DB^\b(\X,\TT)$ on the tangent fibration $\DB(\X,\TT)\to(\X,\TT)$ of differential bundles of $(\X,\TT)$ (cf.~\cite[Corollary~5.9]{cockett:differential-bundles}). Concretely, the involution consists of a morphism of fibrations between the fibration $\DB(\X,\TT)\to(\X,\TT)$ and its dual fibration $\DB^\b(\X,\TT)\to(\X,\TT)$ whose base functor is the identity and whose total functor is the identity on objects, together with a natural isomorphism $\xi\colon(-)^\*\o(-)^\*\Rightarrow\id_{\DB(\X,\TT)}$. Moreover, the following diagram is required to commute:
\begin{equation*}
\begin{tikzcd}
{\DB(\X,\TT)} & {\DB(\X,\TT)} \\
{\DB^\b(\X,\TT)} & {\DB^\b(\X,\TT)}
\arrow["{\bar\T}", from=1-1, to=1-2]
\arrow["{(-)^\*}"', from=1-1, to=2-1]
\arrow["{(-)^\*}", from=1-2, to=2-2]
\arrow["{{\bar\T}^\b}"', from=2-1, to=2-2]
\end{tikzcd}
\end{equation*}
\end{definition}

A dagger fibration, introduced in~\cite{cockett:reverse-derivative-cats}, consists of a fibration equipped with an involution. Let us recall this definition.

\begin{definition}
\label{definition:dagger-fibration}
A \textbf{dagger} (\textbf{cloven}) \textbf{fibration} consists of a (cloven) fibration $\Pi\colon\X'\to\X$ together with a morphism $(-)^\*\colon\Pi\to\Pi^\b$ of fibrations, called the \textbf{involution}, between $\Pi$ and its dual fibration $\Pi^\b$, whose base functor is the identity and whose total functor $(-)^\*\colon\X'\to\X'^\b$ is the identity on objects, and with a natural isomorphism $\xi\colon(-)^\*\o(-)^\*\Rightarrow\id_{\X'}$. When the fibration is assumed to be cloven, the involution is required to preserve the cleavage, by sending Cartesian cloven morphisms to Cartesian cloven morphisms.
\end{definition}

\begin{remark}
\label{remark:dagger-fibration}
In~\cite[Definition~33]{cockett:reverse-derivative-cats}, the isomorphism $\xi\colon(-)^\*\o(-)^\*\Rightarrow\id_{\X'}$ of Definition~\ref{definition:dagger-fibration} was assumed to be the identity. However, to compare dagger fibrations with the notion of reverse tangent categories proposed in~\cite{cruttwell:reverse-tangent-cats}, we needed to weaken this condition and assume $\xi$ to be a, non-necessarily trivial, isomorphism.
\end{remark}

In order to compare reverse tangent categories with tangentads of dagger fibrations, first we need to introduce $1$ and $2$-morphisms of dagger fibrations.

\begin{definition}
\label{definition:morphism-dagger-fibration}
A \textbf{morphism of dagger fibrations} from a dagger fibration $\Pi_\o\colon\X_\o'\to\X_\o$ with involution $(-)_\o^\*$ to a dagger fibration $\Pi_\b\colon\X_\b'\to\X_\b$ with involution $(-)_\b^\*$, consists of a morphism $(F,F')\colon\Pi_\o\to\Pi_\b$ of fibrations for which the following diagram commutes:
\begin{equation*}
\begin{tikzcd}
{\X'} & {\X'} \\
{\X'^\b} & {\X'^\b}
\arrow["{F'}", from=1-1, to=1-2]
\arrow["{(-)^\*}"', from=1-1, to=2-1]
\arrow["{(-)^\*}", from=1-2, to=2-2]
\arrow["{F'^\b}"', from=2-1, to=2-2]
\end{tikzcd}
\end{equation*}
\end{definition}

\begin{definition}
\label{definition:2-morphism-dagger-fibration}
A \textbf{$2$-morphism of dagger fibrations} $(\varphi,\varphi')\colon(F,F')\Rightarrow(G,G')\colon\Pi_\o\to\Pi_\b$ consists of a $2$-morphism of fibrations $(\varphi,\varphi')\colon(F,F')\Rightarrow(G,G')$.
\end{definition}

Dagger fibrations together with their morphisms and $2$-morphisms form a $2$-category denoted by $\DFib$.

\begin{proposition}
\label{proposition:reverse-tangent-fibrations-vs-tangent-objects}
A reverse tangent category is a tangentad in the $2$-category $\DFib$ of dagger fibrations.
\end{proposition}
\begin{proof}
By Proposition~\ref{proposition:tangent-fibrations-vs-tangent-objects}, tangent fibrations are equivalent to tangentads in the $2$-category $\Fib$ of fibrations. Moreover, the forgetful $2$-functor $\DFib\to\Fib$ which forgets the dagger structure preserves the $1$ and the $2$-morphisms. Therefore, the underlying fibration of a tangentad of $\DFib$ is a tangent fibration. Moreover, such a fibration comes equipped with an involution $(-)^\*$. Furthermore, the tangent bundle functors $(\T,\T')\colon\Pi\to\Pi$ need to be compatible with the involution. Finally, a reverse tangent fibration is precisely a tangent category whose fibration $\DB(\X,\TT)\to(\X,\TT)$ of differential bundles comes equipped with an involution, compatible with the tangent bundle functors. Therefore, the dagger fibration of differential bundles of a reverse tangent category is a tangentad in $\DFib$.
\end{proof}

It is important to realize that reverse tangent categories only count for a small family of tangentads of $\DFib$. Indeed, in general, the underlying fibration of a tangentad of $\DFib$ will not be the fibration of differential bundles of a tangent category. This suggests the following definition.

\begin{definition}
\label{definition:dagger-tangent-fibration}
A \textbf{dagger tangent fibration} consists of a tangentad of the $2$-category $\DFib$ of dagger fibrations.
\end{definition}

\subsection{Display and \'etale tangent categories}
\label{subsection:tangent-display-systems}
Often in tangent category theory, one needs to work with a class of morphisms stable under tangent pullbacks and under the tangent bundle functor. Such a class of morphisms is called a tangent display system. In~\cite[Definition~2.2]{cruttwell:tangent-display-maps}, Cruttwell and the author proposed an alternative approach to the ``pullback problem'' by looking at \textit{tangent display maps} which form the maximal tangent display system of a tangent category. We suggest the reader consult~\cite{cruttwell:tangent-display-maps} for details.

\par In this section, we show that tangent display systems can be regarded as tangentads in the $2$-category of display systems. Let us start by recalling that a display system $\Dsply$ in a category $\X$ consists of a collection of morphisms of $\X$, called display maps, stable under pullbacks. Concretely, this means that, for each morphism $q\colon E\to M$ of $\Dsply$ and each morphism $f\colon N\to M$ of $\X$, the pullback of $q$ along $f$
\begin{equation*}
\begin{tikzcd}
f^\*E & E \\
N & M
\arrow[from=1-1, to=1-2]
\arrow["{q'}"', from=1-1, to=2-1]
\arrow["\lrcorner"{anchor=center, pos=0.125}, draw=none, from=1-1, to=2-2]
\arrow["q", from=1-2, to=2-2]
\arrow["f"', from=2-1, to=2-2]
\end{tikzcd}
\end{equation*}
exists and the corresponding morphism $q'\colon f^\*E\to N$ belongs to $\Dsply$. A \textit{display category} consists of a category equipped with a display system. In a display category $(\X,\Dsply)$, each pullback diagram whose vertical morphisms are display maps is called a \textit{display pullback}.

\par A display functor from a display category $(\X,\Dsply)$ to a display category $(\X',\Dsply')$ consists of a functor $F\colon\X\to\X'$ which sends each display map of $(\X,\Dsply)$ to a display map of $(\X',\Dsply')$ and which preserves display pullbacks, namely, the image of a display pullback of $(\X,\Dsply)$ along $F$ is a display pullback of $(\X',\Dsply')$.

\par In order to organize display categories and display functors into a $2$-category, we can choose between at least two different options of $2$-morphisms. In particular, we denote by $\DsplyCat$ the $2$-category of display categories, display functors, and natural transformations between display functors. Moreover, $\dDsplyCat$ denotes the $2$-category with the same objects and $1$-morphisms of $\DsplyCat$ and whose $2$-morphisms are \textit{display natural transformations}, namely, natural transformations $\varphi\colon F\Rightarrow G\colon(\X,\Dsply)\to(\X',\Dsply')$ between display functors for which the naturality diagrams
\begin{equation*}
\begin{tikzcd}
FE & GE \\
FM & GM
\arrow["\varphi", from=1-1, to=1-2]
\arrow["Fq"', from=1-1, to=2-1]
\arrow["Gq", from=1-2, to=2-2]
\arrow["\varphi"', from=2-1, to=2-2]
\end{tikzcd}
\end{equation*}
are display pullbacks for each display map $q\colon E\to M$ of $(\X,\Dsply)$.

\par Tangent display systems are the tangent categorical analog of a display system in a tangent category.

\begin{definition}
\label{definition:tangent-display-system}
In a tangent category $(\X,\TT)$, a \textbf{tangent display system} consists of a collection $\Dsply$ of morphisms of $\X$, called \textbf{tangent display maps}, stable under tangent pullbacks and the tangent bundle functor. Concretely, for each $q\colon E\to M$ of $\Dsply$ and each morphism $f\colon N\to M$ of $\X$, the tangent pullback of $q$ along $f$:
\begin{equation*}
\begin{tikzcd}
f^\*E & E \\
N & M
\arrow[from=1-1, to=1-2]
\arrow["{q'}"', from=1-1, to=2-1]
\arrow["\lrcorner"{anchor=center, pos=0.125}, draw=none, from=1-1, to=2-2]
\arrow["q", from=1-2, to=2-2]
\arrow["f"', from=2-1, to=2-2]
\end{tikzcd}
\end{equation*}
exists and the corresponding morphism $q'\colon f^\*E\to N$ belongs to $\Dsply$. Furthermore, the tangent bundle functor sends tangent display maps to tangent display maps. A \textbf{display tangent category} $(\X,\TT;\Dsply)$ comprises a tangent category $(\X,\TT)$ and a tangent display system of $(\X,\TT)$.
\end{definition}

The formal approach does not recover the notion of a tangent display system, exactly. Instead, it produces a stronger version.

\begin{definition}
\label{definition:strong-tangent-display-system}
A tangent display system $\Dsply$ of a tangent category $(\X,\TT)$ is \textbf{strong} when each $n$-fold pullback $\T_n$ of the projection along itself sends each tangent display map to a tangent display map. A \textbf{strong display tangent category} is a display tangent category $(\X,\TT:\Dsply)$ whose tangent display system is strong.
\end{definition}

\begin{proposition}
\label{proposition:tangent-display-system-vs-tangent-objects}
Strong tangent display categories are equivalent to tangentads in the $2$-category $\DsplyCat$.
\end{proposition}
\begin{proof}
Is it not hard to see that a tangentad of $\DsplyCat$ consists of a tangent category $(\X,\TT)$ together with a display system $\Dsply$. Furthermore, since the tangent bundle functor is a display functor, it has to send each display map to another display map and each display pullback to a display pullback. In particular, this implies that $\Dsply$ is stable under the tangent bundle functor and for each display map $q\colon E\to M$ and each $f\colon N\to M$ the pullback of $q$ along $f$ is preserved by all iterates of $\T$ and the corresponding morphism $q'\colon f^\*E\to N$ also belongs to $\Dsply$. Finally, since each $\T_n$ is a display functor, $\Dsply$ is also stable under $\T_n$. This implies that $\Dsply$ is a strong tangent display system for $(\X,\TT)$.
\end{proof}

Despite the discrepancy between display tangent categories and tangentads in the $2$-category of display categories, one can ``strengthen'' a tangent display system by formally closing it by all functors $\T_n$, concretely, given a display tangent category $(\X,\TT;\Dsply)$ let $\bar\Dsply$ be the collection of all morphisms of $\X$ of type $\T_{k_1}\dots\T_{k_n}q$ for each $q\in\Dsply$ and each collection of non-negative integers $k_1\,k_n$.

\begin{lemma}
\label{lemma:strong-and-non-strong-display-tangent-cats}
Given a display tangent category $(\X,\TT;\Dsply)$, $\bar\Dsply$ defines a strong tangent display system for $(\X,\TT)$. Moreover, the inclusion $\Dsply\subseteq\bar\Dsply$ induces a strict display tangent morphism $(\X,\TT;\Dsply)\hookrightarrow(\X,\TT;\bar\Dsply)$. 
\end{lemma}
\begin{proof}
Let us suppose that $q\colon E\to M$ admits all tangent pullbacks and let us prove that so does each $\T_nq\colon\T_nE\to\T_nM$. Consider a morphism $f\colon N\to\T_nM$ and each diagram:
\begin{equation*}
\begin{tikzcd}
& {\T_nE} & {\T E} \\
N & {\T_nM} & {\T M}
\arrow["{\pi_k}", from=1-2, to=1-3]
\arrow["{\T_nq}", from=1-2, to=2-2]
\arrow["{\T q}", from=1-3, to=2-3]
\arrow["f"', from=2-1, to=2-2]
\arrow["{\pi_k}"', from=2-2, to=2-3]
\end{tikzcd}
\end{equation*}
Since $q$ admits all tangent pullback, the pullback of $\T q$ along $\pi_k\o f$ is well-defined for each $k=1\,n$:
\begin{equation*}
\begin{tikzcd}
{P_k} & {\T_nE} & {\T E} \\
N & {\T_nM} & {\T M}
\arrow[dashed, from=1-1, to=1-2]
\arrow[bend left, from=1-1, to=1-3]
\arrow[from=1-1, to=2-1]
\arrow["{\pi_k}", from=1-2, to=1-3]
\arrow["{\T_nq}", from=1-2, to=2-2]
\arrow["{\T q}", from=1-3, to=2-3]
\arrow["f"', from=2-1, to=2-2]
\arrow["{\pi_k}"', from=2-2, to=2-3]
\end{tikzcd}
\end{equation*}
By the universal property, we get a morphism $P_k\dashrightarrow\T_nE$. Since the outer square and the right square are tangent pullback diagrams, by~\cite[Lemma~2.3]{cruttwell:tangent-display-maps} also the left square is a tangent pullback diagram, proving that $\T_nq$ admits all tangent pullbacks.
\par In particular, each $\T_{k_1}\dots\T_{k_n}q$ admits all tangent pullbacks. If $\T_{k_1}\dots\T_{k_n}q$ is a morphism of $\bar\Dsply$, then $\T_n(\T_{k_1}\dots\T_{k_n}q)=\T_n\T_{k_1}\dots\T_{k_n}q$ which, by definition, is another morphism of $\bar\Dsply$. This proves that $\bar\Dsply$ is a strong tangent display system. Finally, by taking $k_1=k_n=0$, it is immediate to see that $\Dsply\subseteq\bar\Dsply$ and that this inclusion defines a strict display tangent morphism.
\end{proof}

One can also recover classes of display \'etale maps as tangentads. For starters, recall the notion of an \'etale map in tangent category, as introduced by MacAdam in~\cite{macadam:phd-thesis}.

\begin{definition}
\label{definition:etale-map}
An \textbf{\'etale map} in a tangent category $(\X,\TT)$ consists of a morphism $q\colon E\to M$ for which the naturality diagram of the projection with $q$:
\begin{equation*}
\begin{tikzcd}
{\T E} & E \\
{\T M} & M
\arrow["p", from=1-1, to=1-2]
\arrow["{\T q}"', from=1-1, to=2-1]
\arrow["q", from=1-2, to=2-2]
\arrow["p"', from=2-1, to=2-2]
\end{tikzcd}
\end{equation*}
is a tangent pullback. In a display tangent category $(\X,\TT;\Dsply)$, a \textbf{display \'etale map} is a display map which is also an \'etale map. Finally, a display tangent category is an \textbf{\'etale tangent category} if each of its display maps is \'etale.
\end{definition}

\'Etale maps recover the usual notion of \'etale maps in differential geometry (cf.~\cite[Proposition~2.29]{cruttwell:tangent-display-maps}).

\begin{lemma}
\label{lemma:etale-maps}
A morphism $q\colon E\to M$ of a tangent category $(\X,\TT)$ is \'etale if and only if the naturality diagrams of $q$ with all of the structural natural transformations of $\TT$ are tangent pullback diagrams.
\end{lemma}
\begin{proof}
The proof of this lemma can be found in~\cite[Lemma~3.35]{cruttwell:tangent-display-maps}.
\end{proof}

\begin{lemma}
\label{lemma:strong-tangent-display-system-etale}
If a tangent display system $\Dsply$ contains only \'etale maps of a tangent category $(\X,\TT)$, then $\Dsply$ is strong.
\end{lemma}
\begin{proof}
Using a similar argument as the one used in Lemma~\ref{lemma:etale-maps} one can prove that the naturality diagrams of the projections $\pi_k\colon\T_n\Rightarrow\T$ of an \'etale map are tangent pullback diagrams. From this, it is not hard to realize that $\T_n$ preserves the tangent display maps of $\Dsply$ since $\Dsply$ is stable under the tangent bundle functor and tangent pullbacks.
\end{proof}

\begin{proposition}
\label{proposition:etale-tangent-display-systems-vs-tangent-objects}
\'Etale tangent categories are equivalent to tangentads in the $2$-category $\dDsplyCat$.
\end{proposition}
\begin{proof}
Proposition~\ref{proposition:tangent-display-system-vs-tangent-objects} proved that a display tangent category is equivalent to a tangentad of $\DsplyCat$. Since $\dDsplyCat$ is a $2$-subcategory of $\DsplyCat$, a tangentad of $\dDsplyCat$ is also a display tangent category $(\X,\TT;\Dsply)$. Moreover, since the $2$-morphisms of $\dDsplyCat$ are display natural transformations, each structural $2$-morphism of the tangent structure $\TT$ is display. In particular, the projection is display, meaning that the naturality diagram of the projection with each tangent display map $q\colon E\to M$ is a tangent pullback, which implies that $q$ is \'etale. Finally, Lemma~\ref{lemma:etale-maps} shows also that if $q\colon E\to M$ is a display \'etale map of a display tangent category $(\X,\TT;\Dsply)$, all the naturality diagrams of each structural natural transformation of $\TT$ with $q$ is a tangent pullback, meaning that the structural natural transformations of $\TT$ are display natural transformations, making $(\X,\TT;\Dsply)$ into a tangentad of $\dDsplyCat$.
\end{proof}

One can employ the formal approach to prove that each display tangent category defines a tangent fibration. Let us start by recalling that each display category $(\X,\Dsply)$ is associated with a fibration $\Pi\colon\Dsply\to\X$ which sends each display map to its codomain, where $\Dsply$ denotes the category whose objects are display maps of $(\X,\Dsply)$ and whose morphisms are commutative squares in $\X$. In particular, the Cartesian lift of a morphism $f\colon N\to M$ of $\X$ on a display map $q\colon E\to M$ is the second projection $\varphi_f\colon f^\*E\to E$ of the pullback diagram:
\begin{equation*}
\begin{tikzcd}
{f^\*E} & E \\
N & M
\arrow["{\varphi_f}", from=1-1, to=1-2]
\arrow["{q'}"', from=1-1, to=2-1]
\arrow["\lrcorner"{anchor=center, pos=0.125}, draw=none, from=1-1, to=2-2]
\arrow["q", from=1-2, to=2-2]
\arrow["f"', from=2-1, to=2-2]
\end{tikzcd}
\end{equation*}
Furthermore, each display functor $F\colon(\X,\Dsply)\to(\X',\Dsply')$ defines a morphism of fibrations:
\begin{equation*}
\begin{tikzcd}
\Dsply & {\Dsply'} \\
\X & {\X'}
\arrow["F", from=1-1, to=1-2]
\arrow["\Pi"', from=1-1, to=2-1]
\arrow["{\Pi'}", from=1-2, to=2-2]
\arrow["F"', from=2-1, to=2-2]
\end{tikzcd}
\end{equation*}
In particular, such a morphism preserves the Cartesian lifts since $F$ sends display pullbacks to display pullbacks. Finally, a natural transformation $\varphi\colon F\Rightarrow G\colon(\X,\Dsply)\to(\X',\Dsply')$ defines a natural transformation between morphisms of fibrations:
\begin{equation*}
\begin{tikzcd}
\Dsply && {\Dsply'} \\
\\
\X && {\X'}
\arrow[""{name=0, anchor=center, inner sep=0}, "F", bend left, from=1-1, to=1-3]
\arrow[""{name=1, anchor=center, inner sep=0}, "G"', bend right, from=1-1, to=1-3]
\arrow["\Pi"', from=1-1, to=3-1]
\arrow["{\Pi'}", from=1-3, to=3-3]
\arrow[""{name=2, anchor=center, inner sep=0}, "F", bend left, from=3-1, to=3-3]
\arrow[""{name=3, anchor=center, inner sep=0}, "G"', bend right, from=3-1, to=3-3]
\arrow["\varphi", shorten <=5pt, shorten >=5pt, Rightarrow, from=0, to=1]
\arrow["\varphi", shorten <=5pt, shorten >=5pt, Rightarrow, from=2, to=3]
\end{tikzcd}
\end{equation*}
Therefore, we obtain a $2$-functor:
\begin{align*}
&\FF\colon\DsplyCat\to\Fib
\end{align*}

\begin{lemma}
\label{lemma:FF-preserves-pullbacks}
The $2$-functor $\FF$ is pointwise pullback-preserving (see Definition~\ref{definition:pullback-preserving}).
\end{lemma}
\begin{proof}
Consider four display endofunctors $F,G,H,K\colon(\X,\Dsply)\to(\X,\Dsply)$ and four natural transformations $\varphi\colon F\to G$, $\psi\colon H\to G$, $\pi_1\colon K\to F$, and $\pi_2\colon K\to H$ such that:
\begin{equation*}
\begin{tikzcd}
K & H \\
F & G
\arrow["{\pi_2}", from=1-1, to=1-2]
\arrow["{\pi_1}"', from=1-1, to=2-1]
\arrow["\lrcorner"{anchor=center, pos=0.125}, draw=none, from=1-1, to=2-2]
\arrow["\psi", from=1-2, to=2-2]
\arrow["\varphi"', from=2-1, to=2-2]
\end{tikzcd}
\end{equation*}
is a pointwise pullback diagram in the category $\End(\X,\Dsply)$. Thus, the display functor $K$ sends an object $M$ of $\X$ to the object $KM$ of $\X$ such that:
\begin{equation*}
\begin{tikzcd}
KM & HM \\
FM & GM
\arrow["{\pi_2}", from=1-1, to=1-2]
\arrow["{\pi_1}"', from=1-1, to=2-1]
\arrow["\lrcorner"{anchor=center, pos=0.125}, draw=none, from=1-1, to=2-2]
\arrow["\psi", from=1-2, to=2-2]
\arrow["\varphi"', from=2-1, to=2-2]
\end{tikzcd}
\end{equation*}
is a pullback diagram of $\X$. The $2$-functor $\FF$ sends a display functor $F$ to a $1$-morphism of fibrations which coincide with $F$ on the base and which is the restriction of $F$ to the display maps on the total categories. Thus, $\FF$ sends $K$ to a $1$-morphism of fibrations which is $K$ on the base and the restriction of $K$ on display maps on the total categories. Furthermore, it sends a natural transformation $\theta$ to a $2$-morphism of fibrations which coincides with $\theta$ on the base and that is defined by the naturality square diagrams of $\theta$ on the total categories. Therefore, for each display map $q\colon E\to M$:
\begin{equation*}
\begin{tikzcd}
KE &&& HE \\
& KM & HM \\
& FM & GM \\
FE &&& GE
\arrow["{\pi_2}", from=1-1, to=1-4]
\arrow["Kq"{description}, from=1-1, to=2-2]
\arrow["{\pi_1}"', from=1-1, to=4-1]
\arrow["Hq"{description}, from=1-4, to=2-3]
\arrow["\psi", from=1-4, to=4-4]
\arrow["{\pi_2}", from=2-2, to=2-3]
\arrow["{\pi_1}"', from=2-2, to=3-2]
\arrow["\lrcorner"{anchor=center, pos=0.125}, draw=none, from=2-2, to=3-3]
\arrow["\psi", from=2-3, to=3-3]
\arrow["\varphi"', from=3-2, to=3-3]
\arrow["Fq"{description}, from=4-1, to=3-2]
\arrow["\varphi"', from=4-1, to=4-4]
\arrow["Gq"{description}, from=4-4, to=3-3]
\end{tikzcd}
\end{equation*}
is a pullback diagram in $\Dsply$. Thus, $\FF$ preserves pointwise pullbacks.
\end{proof}

Thanks to Lemma~\ref{lemma:FF-preserves-pullbacks}, we can now revisit Cockett and Cruttwell's claim according to which each display tangent category defines a tangent fibration as a direct application of the formal approach.

\begin{theorem}
\label{theorem:tangent-display-cats-tangent-fibrations}
Each display tangent category $(\X,\TT;\Dsply)$ defines a tangent fibration $\Pi\colon(\Dsply,\TT)\to(\X,\TT)$.
\end{theorem}
\begin{proof}
Thanks to Lemma~\ref{lemma:FF-preserves-pullbacks}, the $2$-functor $\FF$ preserves pointwise pullbacks, therefore, we can apply $\Tng$ on $\FF$ and obtain a $2$-functor $\Tng(\DsplyCat)\to\Tng(\Fib)$.
\par By Proposition~\ref{proposition:tangent-fibrations-vs-tangent-objects} and Proposition~\ref{proposition:tangent-display-system-vs-tangent-objects} the tangentads of $\DsplyCat$ are strong tangent display categories and the ones of $\Fib$ are tangent fibrations. Furthermore, each display tangent category $(\X,\TT;\Dsply)$ can be made into a strong display tangent category by Lemma~\ref{lemma:strong-and-non-strong-display-tangent-cats}. Thus, each display tangent category can be sent to a strong display tangent category which is sent by $\Tng(\FF)$ to a tangent fibration.
\end{proof}

\subsection{Tangent restriction categories and their \texorpdfstring{$\M$-tangent categories}{M-tangent categories}}
\label{subsection:restriction-tangent-categories}
A \textit{restriction category} is a categorical context for partial maps. Concretely, this is a category equipped with an assignment for each morphism $f\colon M\to N$ of an idempotent $\bar f\colon M\to M$, called the \textit{restriction idempotent} of $f$, satisfying the following axioms:
\begin{enumerate}
\item For all $f\colon M\to N$:
\begin{align*}
&f\o\bar f=f
\end{align*}

\item Two idempotents $\bar f$ and $\bar g$ commutes provided $f$ and $g$ share the same domain;

\item For all $f$ and $g$:
\begin{align*}
&\bar{g\o\bar f}=\bar g\o\bar f
\end{align*}
provided $f$ and $g$ share the same domain;

\item For all $f$ and $g$:
\begin{align*}
&\bar g\o f=f\o\bar{g\o f}
\end{align*}
provided the codomain of $f$ coincides with the domain of $g$.
\end{enumerate}
We invite the reader consult~\cite{cockett:restriction-categories-I}.

\par The archetypal example of a restriction category is the category of sets and partial maps, for which the idempotent $\bar f\colon M\to M$ of a morphism $f\colon M\to N$ is the partial function which is the identity on the domain of $f$ and not defined elsewhere. Total maps of a restriction category are those morphisms whose restriction idempotent is the identity.

\par A tangent restriction category is a restriction category equipped with a tangent structure compatible with the restriction structure. First, recall the notion of a restriction pullback. We suggest the interested reader consult~\cite{cockett:restriction-categories-III-colims}.

\begin{definition}
\label{definition:restriction-pullbacks}
In a restriction category, a \textbf{restriction pullback} between two maps $f\colon B\to A$ and $g\colon C\to A$ consists of an object $B\times_A^RC$ together with two total maps $\pi_1\colon B\times_A^RC\to B$ and $\pi_2\colon B\times_A^RC\to C$ such that
\begin{equation*}
\begin{tikzcd}
{B\times_A^RC} & C \\
B & A
\arrow["{\pi_2}", from=1-1, to=1-2]
\arrow["{\pi_1}"', from=1-1, to=2-1]
\arrow["g", from=1-2, to=2-2]
\arrow["f"', from=2-1, to=2-2]
\end{tikzcd}
\end{equation*}
commutes on the nose, and for every pair of maps $b\colon X\to B$ and $c\colon X\to C$ such that
\begin{equation*}
\begin{tikzcd}
X & X & C \\
X & B & A
\arrow["{\bar b}", from=1-1, to=1-2]
\arrow["{\bar c}"', from=1-1, to=2-1]
\arrow["c", from=1-2, to=1-3]
\arrow["g", from=1-3, to=2-3]
\arrow["b"', from=2-1, to=2-2]
\arrow["f"', from=2-2, to=2-3]
\end{tikzcd}
\end{equation*}
there exists a unique morphism $\<b,c\>^R\colon X\to B\times_A^RC$ such that:
\begin{equation*}
\begin{tikzcd}
X & X & X \\
B & {B\times_A^RC} & C
\arrow["b"', from=1-1, to=2-1]
\arrow["{\bar c}"', from=1-2, to=1-1]
\arrow["{\bar b}", from=1-2, to=1-3]
\arrow["{\<b,c\>^R}"{description}, from=1-2, to=2-2]
\arrow["c", from=1-3, to=2-3]
\arrow["{\pi_1}", from=2-2, to=2-1]
\arrow["{\pi_2}"', from=2-2, to=2-3]
\end{tikzcd}
\end{equation*}
\end{definition}

A restriction functor $F$ is a functor which preserves the restriction idempotents:
\begin{align*}
&F\bar f=\bar{Ff}
\end{align*}
There are at least two candidates to define $2$-morphisms of restriction categories. A \textit{total natural transformation} $\varphi\colon F\Rightarrow G$ between two restriction functors consists of a natural transformation $\varphi_M\colon FM\to GM$, natural in $M$, such that each morphism $\varphi_M$ is total, namely, $\bar\varphi_M=\id_{FM}$.

\par A \textit{restriction natural transformation} $\varphi\colon F\Rightarrow G$ is a collection of total maps $\varphi_M\colon FM\to GM$, indexed by the object of the domain restriction category, such that the following diagram commutes for each morphism $f\colon M\to N$:
\begin{equation*}
\begin{tikzcd}
FM & FM & GM \\
FN && GN
\arrow["{\bar{Ff}}", from=1-1, to=1-2]
\arrow["Ff"', from=1-1, to=2-1]
\arrow["\varphi", from=1-2, to=1-3]
\arrow["Gf", from=1-3, to=2-3]
\arrow["\varphi"', from=2-1, to=2-3]
\end{tikzcd}
\end{equation*}
We denote by $\RestrCat$ and by $\RestrCat_\lax$ the $2$-categories of restriction categories, restriction functors, and total natural and restriction natural transformations, respectively.

\par In~\cite{cockett:tangent-cats}, Cockett and Cruttwell introduced the concept of a tangent restriction category. 

\begin{definition}
\label{definition:restriction-tangent-category}
A \textbf{tangent restriction category} consists of a restriction category $\X$ equipped with a \textbf{tangent restriction structure}, which consists of a restriction functor $\T\colon\X\to\X$ together with a collection of total natural transformations $p\colon\T\Rightarrow\id_\X$, $z\colon\id_\X\Rightarrow\T$, $s\colon\T_2\Rightarrow\T$, $l\colon\T\Rightarrow\T^2$, $c\colon\T^2\Rightarrow\T^2$ satisfying the same axioms of Definition~\ref{definition:tangent-category}, where, however, the fundamental tangent limits are replaced with restriction pullbacks whose universality is preserved by all iterates of $\T$.
\end{definition}

One could hope to compare tangent restriction categories with tangentads of either one of the $2$-categories $\RestrCat$ or $\RestrCat_\lax$. However, in both scenarios, the fundamental tangent limits in the definition of tangentads do not match with the restriction pullbacks of Definition~\ref{definition:restriction-tangent-category}. In particular, one finds that the fundamental tangent limits are total pullbacks, namely, pullback diagrams in the subcategory of total maps.

\par Despite this incompatibility, in~\cite{cockett:restriction-categories-III-colims}, Cockett and Lack showed that in a split restriction category, restriction limits coincide with ordinary limits on total maps. By specializing their result for restriction pullbacks, we conclude that in a split restriction category, restriction pullbacks coincide with total pullbacks, namely, pullback diagrams in the subcategory of total maps.

\par This observation leads to a characterization of tangent split restriction categories, namely, tangent restriction categories whose restriction idempotents split, as tangentads. First, let us introduce the appropriate $2$-category. $\sRestrCat$ denotes the $2$-category of split restriction categories, restriction functors, and total natural transformations. Notice that since the splitting of an idempotent is always unique up to isomorphism, each restriction functor preserves the splittings.

\begin{lemma}
\label{lemma:T-n-is-restriction}
Let $(\X,\TT)$ be a tangent restriction category and consider the $n$-fold restriction pullback $\T_n\colon\X\to\X$ of the projection along itself.
\begin{enumerate}
\item The projections $\pi_k\colon\T_nA\to\T A$ are total natural transformations;

\item The functors $\T_n\colon\X\to\X$ are restriction functors.
\end{enumerate}
\end{lemma}
\begin{proof}
Let us consider a morphism $f\colon A\to B$. Let us compute the following:
\begin{align*}
&\quad\bar{\bar{\T f}\o\pi_i}\\
=&\quad\bar{\bar{f\o p}\o\pi_i}\\
=&\quad\bar{f\o p\o\pi_i}\\
=&\quad\bar{f\o p\o\pi_j}\\
=&\quad\bar{\bar{\T f}\o\pi_j}
\end{align*}
for any $i,j=1\,n$, where we used that $\bar{\T f}=\bar{\bar p\o\T f}=\bar{p\o\T f}=\bar{f\o p}$, since $p$ is a total natural transformation. So, in particular:
\begin{align*}
&\bar{\T f\o\pi_i}=\bar{\bar{\T f}\o\pi_i}=\bar{\bar{\T f}\o\pi_j}=\bar{\T f\o\pi_j}
\end{align*}
Let us prove that $\pi_i\colon\T_n\Rightarrow\T$ is natural. First, recall that, by definition of an $n$-fold restriction pullback
\begin{align*}
&\pi_i\o\T_nf=\T f\o\pi_i\o e_i
\end{align*}
where $e_i$ denotes the composition of all idempotents $\bar{\T f\o\pi_j}$ for every $j\neq i$. However, since $\bar{\T f\o\pi_j}=\bar{\T f\o\pi_i}$, we obtain that:
\begin{align*}
&\quad\pi_i\o\T_nf\\
=&\quad\T f\o\pi_i\o\bar{\T f\o\pi_i}\\
=&\quad\T f\o\bar{\T f}\o\pi_i\\
=&\quad\T f\o\pi_i
\end{align*}
where we used that $f\o\bar{g\o f}=\bar g\o f$ provided that $f$ and $g$ are composable. This shows that $\pi_i$ is natural. Let us now prove that $\T_n$ is a restriction functor. First, notice that, by the universal property of the restriction pullback, $\T_n\bar f\colon\T_nA\to\T_nA$ is the unique morphism satisfying
\begin{align*}
&\pi_i\o\T_n\bar f=\bar{\T f}\o\pi_i\o e_i=\bar{\T f}\o\pi_i
\end{align*}
for every $i=1\,n$. Let us show that also $\bar{\T_nf}$ satisfies the same equations. First, notice that, since $\pi_i$ is a total natural transformation:
\begin{align*}
&\bar{\T_nf}=\bar{\bar\pi_i\o\T_nf}=\bar{\pi_i\o\T_nf}=\bar{\T f\o\pi_i}
\end{align*}
Therefore, we can compute:
\begin{align*}
&\pi_i\o\bar{\T_nf}=\pi_i\o\bar{\T f\o\pi_i}=\bar{\T f}\o\pi_i
\end{align*}
This proves that $\T_n$ preserves restriction idempotents.
\end{proof}

\begin{proposition}
\label{proposition:split-restriction-tangent-categories}
Split tangent restriction categories are $2$-equivalent to tangentads in the $2$-category $\sRestrCat$ of split restriction categories, restriction functors, and total natural transformations.
\end{proposition}
\begin{proof}
A tangentad in $\sRestrCat$ consists of a split restriction category $\X$ together with a restriction endofunctor $\T\colon\X\to\X$ and the structural total natural transformations $p$, $z$, $s$, $l$, and $c$, satisfying the axioms of a tangent structure. In particular, the fundamental tangent limits, in this context, are total pullbacks. Since $\X$ is a split restriction category, these total pullbacks become restriction pullbacks. Thus, a tangentad in $\sRestrCat$ is a tangent split restriction category.

\par Conversely, consider a tangent split restriction category $(\X,\TT)$. This consists of a split restriction category $\X$, a restriction endofunctor $\T$ and a collection of structural total natural transformations $p$, $z$, $s$, $l$, and $c$. Thanks to Lemma~\ref{lemma:T-n-is-restriction}, the projections $\pi_i\colon\T_n\Rightarrow\T$ of the $n$-fold restriction pullbacks of the projection along itself are total natural transformations and each $\T_n$ is a restriction functor. Thus, since restriction pullbacks in a split restriction category are equivalent to total pullbacks, $\T_n$ is the pointwise $n$-fold pullback of $p\colon\T\Rightarrow\id_\X$ in the category $\End(\X)$. This proves that a tangent split restriction category is a tangentad of $\sRestrCat$.
\end{proof}

Despite the discrepancy between tangent restriction categories and tangentads, one can formally split the restriction idempotents of a tangent restriction category.

\begin{lemma}
\label{lemma:formal-split-restriction-tangent-cats}
For each tangent restriction category $(\X,\TT)$ the split restriction category $\bar\X$ obtained by formally splitting the restriction idempotents of $\X$ (cf.~\cite[Proposition~2.26]{cockett:restriction-categories-I})comes equipped with a tangent restriction structure $\bar\TT$. Moreover, the inclusion restriction functor $\X\hookrightarrow\bar\X$ lifts to a strict tangent restriction morphism $(\X,\TT)\hookrightarrow(\bar\X,\bar\TT)$.
\end{lemma}
\begin{proof}
The objects of $\bar\X$ are pairs $(M,e)$ formed by an object $M$ of $\X$ together with a restriction idemponent $e\colon M\to M$ on $M$. Morphisms $(M,e)\to(M',e')$ of $\bar\X$ are morphisms $f\colon M\to M'$ of $\X$ which intertwin with the restriction idempotents, namely, $e\o f=f\o e'$. The tangent bundle functor $\bar\T$ on $\bar\X$ sends $(M,e)$ to $(\T M,\T e)$ and each $f$ to $\T f$. The structural natural transformations are defined as in $(\X,\TT)$, since by naturality they commute with the corresponding restriction idempotents. The $n$-fold restriction pullback $\bar\T_n$ sends $(M,e)$ to $(\T_nM,\T_ne)$ and each $f$ to $\T_nf$. Since the inclusion $\X\hookrightarrow\bar\X$ preserves total pullbacks, namely, pullbacks on total maps, and since restriction pullbacks in a tangent split restriction category are equivalent to total pullbacks, it is not hard to prove that $\bar\T_n$ is indeed a restriction $n$-fold tangent pullback. Using a similar argument, one can prove the universal property of the vertical lift. Finally, the formal splitting preserves equational axioms. Therefore, $(\bar\X,\bar\TT)$ is a tangent restriction category and the inclusion $\X\hookrightarrow\bar\X$ extends to a strict tangent morphism $(\X,\TT)\hookrightarrow(\bar\X,\bar\TT)$.
\end{proof}

In~\cite[Theorem~3.4]{cockett:restriction-categories-III-colims}, Cockett and Lack proved that split restriction categories are $2$-equivalent to $\M$-categories. Proposition~\ref{proposition:split-restriction-tangent-categories} suggests lifting such a $2$-equivalence to tangent split restriction categories, by applying the $2$-functor $\Tng$ to Cockett and Lack's $2$-equivalence. First, recall their construction.

\begin{definition}
\label{definition:M-categories}
A \textbf{display system of monics} of a category $\X$ consists of a collection $\M$ of morphisms of $\X$ satisfying the following conditions:
\begin{itemize}
\item Each morphism $f\in\M$ is a monomorphism of $\X$;

\item $\M$ forms a display system of $\X$;

\item $\M$ is closed under composition;

\item $\M$ contains all isomorphisms of $\X$.
\end{itemize}
An \textbf{$\M$-category} consists of a category $\X$ equipped with a display system of monics.
\end{definition}

\begin{remark}
\label{remark:naming-display-system-monics}
In~\cite{cockett:restriction-categories-I}, a display system of monics is called a \textit{stable} system of monics.
\end{remark}

\begin{definition}
\label{definition:M-functors}
An \textbf{$\M$-functor} from an $\M$-category $(\X,\M)$ to an $\M$-category $(\X',\M')$ consists of a functor $F\colon\X\to\X'$ which sends each monic of $\M$ to a monic of $\M'$ and which preserves each pullback diagram of type
\begin{equation*}
\begin{tikzcd}
D & B \\
C & A
\arrow[from=1-1, to=1-2]
\arrow[from=1-1, to=2-1]
\arrow["\lrcorner"{anchor=center, pos=0.125}, draw=none, from=1-1, to=2-2]
\arrow["m", from=1-2, to=2-2]
\arrow["f"', from=2-1, to=2-2]
\end{tikzcd}
\end{equation*}
where $m$ is a monic of $\M$.
\end{definition}

\begin{definition}
\label{definition:M-natural-transformation}
An \textbf{$\M$-natural transformation} from an $\M$-functor $F\colon(\X,\M)\to(\X',\M')$ to an $\M$-functor $G\colon(\X,\M)\to(\X',\M')$ consists of a natural transformation $\varphi\colon F\Rightarrow G$ for which, for every monic $m$ of $\M$, the naturality diagram
\begin{equation}
\label{equation:M-natural-transformation}
\begin{tikzcd}
FA & GA \\
FB & GB
\arrow["\varphi", from=1-1, to=1-2]
\arrow["Fm"', from=1-1, to=2-1]
\arrow["Gm", from=1-2, to=2-2]
\arrow["\varphi"', from=2-1, to=2-2]
\end{tikzcd}
\end{equation}
is a pullback diagram.
\end{definition}

\begin{remark}
\label{remark:M-natural-transformations}
Notice that, the pullback diagram of Equation~\eqref{equation:M-natural-transformation} is preserved by every $\M$-functor from $(\X',\M')$, since $Gm$ is a monic in $\M'$.
\end{remark}

$\M$-categories, $\M$-functors, and $\M$-natural transformations form a $2$-category denoted by $\MCat$.

\par Cockett and Lack proved that $\MCat$ and $\sRestrCat$ are $2$-equivalent. In particular, each $\M$-category $(\X,\M)$ is equivalent to a split restriction category $\Par(\X,\M)$ whose objects are the same as the one of $\X$ and morphisms $f\colon A\nto B$ are (classes of isomorphisms of) spans:
\begin{equation*}
\begin{tikzcd}
& {\bar A} \\
A && B
\arrow["{m_f}"', from=1-2, to=2-1]
\arrow["{[f]}", from=1-2, to=2-3]
\end{tikzcd}
\end{equation*}
whose left leg $m_f$ is a monic of $\M$.

\par Our goal is to extend this result to tangent split restriction categories. The first step is introducing the correct notion of a display system of monics in this context.

\begin{definition}
\label{definition:M-tangent-category}
A \textbf{tangent display system of monics} of a tangent category $(\X,\TT)$ consists of a collection $\M$ of monomorphisms of $(\X,\TT)$ satisfying the following conditions:
\begin{itemize}
\item Each monomorphism $m$ of $\M$ is an \'etale map of $(\X,\TT)$;

\item $\M$ forms a tangent display system of $(\X,\TT)$;

\item $\M$ is closed under composition;

\item $\M$ contains all isomorphisms of $(\X,\TT)$.
\end{itemize}
An \textbf{tangent $\M$-category} consists of a tangent category equipped with a tangent display system of monics.
\end{definition}

\begin{remark}
\label{remark:strong-tangent-display-system}
Thanks to Lemma~\ref{lemma:strong-tangent-display-system-etale}, a tangent display system of monics is automatically a strong tangent display system.
\end{remark}

\begin{proposition}
\label{proposition:M-tangent-category-vs-tangent-objects}
$\M$-tangent categories are equivalent to tangentads in the $2$-category $\MCat$ of $\M$-categories, $\M$-functors, and $\M$-natural transformations.
\end{proposition}
\begin{proof}
A tangentad in the $2$-category $\MCat$ of $\M$-categories consists of an $\M$-category $(\X,\M)$ together with an $\M$-endofunctor $\T\colon(\X,\M)\to(\X,\M)$ together with the structural $\M$-natural transformations $p$, $z$, $s$, $l$, and $c$. Since $\T$ preserves the monics of $\M$ and the display pullbacks, $\M$ becomes a tangent display system. Furthermore, since $p$ is $\M$-natural, each monic of $\M$ becomes an \'etale map. Therefore, tangentads of $\MCat$ are $\M$-tangent categories. Conversely, Lemma~\ref{lemma:etale-maps} establishes that a map is \'etale if and only if all naturality squares with the structural natural transformations of a tangent category are tangent pullbacks. Thus, the structural natural transformations of a tangent $\M$-category are $\M$-natural transformations. Thus, each tangent $\M$-category is a tangentad in $\MCat$.
\end{proof}

Let $\MTngCat$ denote $\Tng(\MCat)$.

\begin{theorem}
\label{theorem:split-restriction-cats-vs-M-cats}
The $2$-category $\MTngCat$ of $\M$-tangent categories is $2$-equivalent to the $2$-category $\TngsRestrCat$ of tangent split restriction categories. In particular, each tangent $\M$-category $(\X,\TT;\M)$ is sent to the tangent split restriction category $\Par(\X,\TT;\M)$ whose underlying split restriction category is $\Par(\X,\M)$ and whose tangent structure is induced by $\TT$.
\end{theorem}
\begin{proof}
To prove this result, one only needs to apply the $2$-functor $\Tng$ to the $2$-equivalence of~\cite[Theorem~3.4]{cockett:restriction-categories-I} and use Propositions~\ref{proposition:split-restriction-tangent-categories} and~\ref{proposition:M-tangent-category-vs-tangent-objects}.
\end{proof}

\begin{remark}
\label{remark:open-subobjects}
In~\cite[Definition~3.42]{cruttwell:tangent-display-maps}, the monics of a tangent display system of monics are called \textit{open subobjects}. This is due to the classification of open subsets in the tangent category of finite-dimensional smooth manifolds as tangent monic (display) \'etale maps~\cite[Lemma~3.41]{cruttwell:tangent-display-maps}. We suggest the reader consult~\cite{cruttwell:tangent-display-maps} for details.
\end{remark}

Thanks to Remark~\ref{remark:open-subobjects}, we can interpret Theorem~\ref{theorem:split-restriction-cats-vs-M-cats} by saying that the restriction domain of a partial map in a tangent restriction categories is always an open subobject to the domain. This also means that the tangent category of smooth partial functions between smooth manifolds defined over arbitrary domains does not form a tangent restriction category. To fix this, one might consider a more general concept.

\begin{definition}
\label{definition:quasi-split-restriction-tangent-category}
A \textbf{quasi-tangent restriction category} consists of a restriction category $\X$ whose split completion $\bar\X$ (cf. Lemma~\ref{lemma:formal-split-restriction-tangent-cats}) comes equipped with a tangent structure $\TT$ in the $2$-category $\sRestrCat_\lax$.
\end{definition}

Concretely, a quasi-tangent restriction category consists of a restriction category, a restriction endofunctor $\T\colon\X\to\X$ and restriction natural transformations $p$, $z$, $s$, $l$, and $c$, satisfying the usual axioms of a tangentad and for which the fundamental tangent limits are restriction pullbacks. In future work, we will study this generalization.

\subsection{Infinitesimal objects are tangentads of spans}
\label{subsection:infinitesimal-tangent-objects}
Infinitesimal objects were introduced by Cockett and Cruttwell in~\cite[Definition~5.6]{cockett:tangent-cats} to classify representable tangent categories, namely, Cartesian tangent categories whose tangent bundle functor is representable. In particular, \cite[Proposition~5.7]{cockett:tangent-cats} shows that the "exponentialization" of an infinitesimal object produces a representable tangent structure and vice versa.

\par In this section, we first show that infinitesimal objects of a Cartesian category $\X$ are tangentads in the $2$-co-category of (classes of isomorphisms of) spans of $\X$. Secondly, we interpret the relationship between infinitesimal objects and representable tangent categories as an equivalence between tangent structures.

\par In Definition~\ref{definition:infinitesimal-element}, we introduced infinitesimal elements in the context of tangentads. In the context of tangent categories, infinitesimal elements are called infinitesimal objects~\cite[Definition~5.6]{cockett:tangent-cats}. Moreover, we call a microscopic element of a tangent category, a \textbf{microscopic object}.

\begin{remark}
\label{remark:infinitesimal-object}
In the original definition of an infinitesimal object~\cite[Definition~5.6]{cockett:tangent-cats}, the $n$-fold pushout of $\zeta$ along itself was not required to be preserved by the functors $D^n\times(-)$ since this condition is already implied by requiring $D^n$ and $D_n$ to be exponentiable objects. Notice that this condition is necessary to express the compatibility between the multiplication $\mu$ and the cosum $\delta$.
\par Moreover, in the original definition, each object $D^n\=D\times\dots\times D$ was required to be exponentiable. However, when $D$ is exponentiable, namely, $D\times-$ admits a left adjoint, each $((\dots(-)^D\dots)^D)^D$ defines a left adjoint for $D^n\times-$, making $D^n$ exponentiable.
\end{remark}

Cockett and Cruttwell showed that if $D$ is an infinitesimal object of a Cartesian category $\X$, then $\T\=(-)^D\colon\X\to\X$ defines a tangent bundle functor, $(-)^\zeta\colon\T\Rightarrow\id_\X$ defines the projection, $(-)^!\colon\id_\X\Rightarrow\T$ defines the zero morphism, $(-)^\delta\colon\T_2=(-)^{D\star D}\Rightarrow\T$ defines the sum, $(-)^\mu\colon\T\Rightarrow\T^2=(-)^{D\times D}$ defines the vertical lift, and finally, $(-)^\tau\colon\T^2\to\T^2$ defines the canonical flip, where $\tau\colon D\times D\to D\times D$ is the canonical symmetry $\<\pi_2,\pi_1\>$.

\par This suggests interpreting an infinitesimal object as a tangentad in a suitable $2$-category. The first suggestion to pick the correct ambient $2$-category comes from noticing that the tangent bundle $1$-morphism should be the object $D$. A $2$-category in which morphisms are objects is the category of spans. Since we do not want to require the existence of all pullbacks, we decide to consider only spans whose legs are display morphisms, namely, morphisms which admit all pullbacks. We call \textit{display spans} such spans.

\par Given a category $\X$, let denote by $\DSpan(\X)$ the (strict) $2$-category whose objects are objects of $\X$, $1$-morphisms are isomorphism classes of display spans, and $2$-morphisms are isomorphism classes of morphisms of spans. The second hint comes from realizing that the exponential operation $M^{(-)}$ is contravariant. This suggests to inverts the direction of the $2$-morphisms of $\DSpan(\X)$, namely, to consider the $2$-category $\DSpan^\co(\X)$.

\begin{proposition}
\label{proposition:micro-objects-vs-tangent-objects}
Microscopic objects of a Cartesian category $\X$ are equivalent to tangentads in the $2$-category $\DSpan^\co(\X)$ over the terminal object $\1$ of $\X$.
\end{proposition}
\begin{proof}
Let us consider a display span over the terminal object:
\begin{equation*}
\begin{tikzcd}
& D \\
\1 && \1
\arrow["{q_l}"', from=1-2, to=2-1]
\arrow["{q_r}", from=1-2, to=2-3]
\end{tikzcd}
\end{equation*}
Since $\1$ is terminal, $q_l,q_r\colon D\to\1$ must coincide with the unique morphism $!\colon D\to\1$. The projection corresponds to a morphism of $\X$ of type $\zeta\colon\1\to D$, the zero morphism is just $!\colon D\to\1$, the $n$-fold pullbacks of the projection along itself correspond to the $n$-fold pushouts $D\star\dots\star D$ of $\zeta$ along itself, the sum morphism corresponds to $\delta\colon D\to D\star D$, the vertical lift corresponds to an associative and commutative multiplication $\mu\colon D\times D\to D$, since $D\times D$ represents the composition of the span $D$ with itself. Finally, the canonical flip coincides with an isomorphism $c\colon D\times D\to D\times D$. However, since $(D\times!)\o c=\T p\o c\T=p_\T=!\times D$, we have that $\pi_1\o c=\pi_2$. However, $\tau$ is the unique isomorphism such that $\pi_1\o\tau=\pi_2$, thus $c=\tau$. One can easily see that the equational axioms of a microscopic object correspond exactly to the equational axioms of a tangent structure on $\1$ in $\DSpan^\co(\X)$ of type $(D,\zeta,!,\delta,\mu,\tau)$. Finally, since $\End_{\DSpan^\co(\X)}(\1)\cong\X^\op$ as a category, the universality of the vertical lift and the existence of the $n$-fold tangent pullbacks correspond precisely to the existence of the $n$-fold pushouts of $\zeta\1\to D$ (preserved by $D^n\times(-)$) and the coequalizer axiom for the multiplication $\mu$. This shows the equivalence between tangentads of $\DSpan^\co(\X)$ and microscopic objects of $\X$.
\end{proof}

To obtain infinitesimal objects as tangentads, one must consider exponentiable objects of $\X$.

\begin{proposition}
\label{proposition:infinitesimal-object-as-tangent-object}
Infinitesimal objects of a Cartesian category $\X$ are equivalent to tangentads in the $2$-category $\DSpan^\co(\X)$ over the terminal object $\1$ of $\X$ for which each $n$-fold pullback $D_n$ of the projection along itself is exponentiable in $\X$.
\end{proposition}

\begin{remark}
\label{remark:infinitesimal-objects-as-tangent-objects}
One could wonder if the tangentads on the $2$-category of display spans over the full subcategory $\Exp(\X)$ of exponentiable objects of $\X$ could represent infinitesimal objects of $\X$. However, this is not the case, since $\End_{\DSpan^\co(\Exp(\X))}(\1)$ is not equivalent to the whole category $\X$, but only to $\Exp(\X)$. This implies that the universality of the vertical lift and the existence of the $n$-fold tangent pullbacks of the projection along itself only apply among those objects of $\X$ which are exponentiable, while for an infinitesimal object, the universality is required to hold on the entire category $\X$.
\end{remark}

One is left to wonder what the rest of the tangentads of $\DSpan^\co(\X)$ might represent. Let us introduce the following definition.

\begin{definition}
\label{definition:relative-infinitesimal-object}
A \textbf{relative microscopic object} on a category $\X$ on an object $M$ of $\X$ consists of a microscopic object in the display slice category $\X/M$ of $\X$ on $M$. Concretely, this corresponds to an object $D$ of $\X$ equipped with a morphism $q\colon D\to M$, a coprojection $\zeta\colon M\to D$, for which the $n$-fold pushout $D_n$ along $\zeta$ to itself exists and is preserved by all functors $D^n\times_M(-)$, where $D^n=D\times_M\dots\times_MD$; a cosum $\delta\colon D\to D_2$ making $\zeta$ a cocommutative comonoid on the coslice category $\id_M/(\X/M)$; finally, a multiplication $\mu\colon D\times_MD\to D$, associative and commutative, which satisfies the zero rule, is compatible with the cosum, and making the diagram:
\begin{equation*}
\begin{tikzcd}
D &&& M \\
{D\times D} & {D\times(D\star_MD)} & {(D\times D)\star_D(D\times D)} & {D\star D}
\arrow["q", from=1-1, to=1-4]
\arrow["{D\times\zeta}"', from=1-1, to=2-1]
\arrow["{\zeta\iota_1}", from=1-4, to=2-4]
\arrow["{\id_D\times\delta}", from=2-1, to=2-2]
\arrow["\cong", from=2-2, to=2-3]
\arrow["{\mu\star_M\pi_1}", from=2-3, to=2-4]
\end{tikzcd}
\end{equation*}
a pushout diagram in $\X/M$.\\
A \textbf{relative infinitesimal object} on $M$ is a relative microscopic object $D$ on $M$ for which each $D_n\to M$ is exponentiable in $\X/M$.
\end{definition}

It is immediate to see that relative microscopic/infinitesimal objects on the terminal object $\1$ of a Cartesian category $\X$ are precisely microscopic/infinitesimal objects of $\X$.

\begin{proposition}
\label{proposition:relative-infinitesimal-object-as-tangent-object}
Relative microscopic objects of a category $\X$ are equivalent to tangentads of the $2$-category $\DSpan^\co(\X)$.
\end{proposition}
\begin{proof}
The proof is essentially the same as in Proposition~\ref{proposition:infinitesimal-object-as-tangent-object}.
\end{proof}

We can now reinterpret the relationship between infinitesimal objects and representable Cartesian tangent categories as an equivalence between tangent structures. Indeed, an infinitesimal object $D$ of a Cartesian category $\X$ is a tangentad of $\DSpan^\co(\X)$ on the terminal object $\1$ of $\X$ for which each $D_n$ is an exponentiable object of $\X$.

\par Exponentiable objects of $\X$ are precisely those objects of $\End_{\DSpan^\co(\X)}(\1)\cong\X^\op$ which are coexponentiable. 
Recall also that a tangentad on $\DSpan^\co(\X)$ on $\1$ corresponds to a Leung functor $\Leung[D]\colon\Weil\to\End_{\DSpan^\co(\X)}(\1)$. By only considering those tangentads for which $D_n$ are exponentiable in $\X$, one can extend the Leung functor $\Leung[D]$ to a new Leung functor $\Leung[(-)^D]\colon\Weil\to\End_{\Cat}(\X)$, whose corresponding tangent structure is representable.

\par Conversely, if $\Leung[\TT]\colon\Weil\to\End_{\Cat}(\X)$ represents a representable Cartesian tangent category, one can "de-exponentialize" the tangent bundle functor and define a Leung functor $\Leung[D]\colon\Weil\to\End_{\DSpan^\co(\X)}(\1)$, namely, a tangentad of $\DSpan^\co(\X)$ on $\1$, for which $D_n$ are exponentiable in $\X$.



\end{document}